\documentclass[a4paper,oneside,english]{amsart}
\usepackage[T1]{fontenc}
\usepackage[latin9]{inputenc}
\usepackage{mathrsfs}
\usepackage{amstext}
\usepackage{amsthm}
\usepackage{amssymb}
\usepackage{stmaryrd}
\usepackage[all]{xy}

\makeatletter

\pdfpageheight\paperheight
\pdfpagewidth\paperwidth

\newcommand{\noun}[1]{\textsc{#1}}

\numberwithin{equation}{section}
\numberwithin{figure}{section}
\theoremstyle{plain}
\newtheorem{thm}{\protect\theoremname}
  \theoremstyle{definition}
  \newtheorem{defn}[thm]{\protect\definitionname}
  \theoremstyle{remark}
  \newtheorem{rem}[thm]{\protect\remarkname}
  \theoremstyle{plain}
  \newtheorem{prop}[thm]{\protect\propositionname}
  \theoremstyle{definition}
  \newtheorem{example}[thm]{\protect\examplename}
  \theoremstyle{plain}
  \newtheorem{lem}[thm]{\protect\lemmaname}

\usepackage[dvipsnames,svgnames,x11names,hyperref]{xcolor}
\usepackage[colorlinks=true,linkcolor=Dark Red, citecolor=Dark Green,linktoc=all]{hyperref}
\usepackage{lmodern}
\renewcommand*{\epsilon}{\varepsilon}
\usepackage{amssymb}
\usepackage{amsfonts}
\usepackage{egothic}
\usepackage[T1]{fontenc}

\usepackage{url}
\usepackage{amsthm}
\usepackage{aliascnt}
\usepackage{lipsum}
\usepackage{mathrsfs}
\usepackage{esint}
\usepackage{url}
\usepackage{amsmath}
\usepackage{dsfont}
\usepackage{bbm}
\usepackage{pdflscape}
\usepackage{bbm}
\usepackage{bussproofs}

\usepackage{tikz}
\tikzset{node distance=2.5cm, auto}
\usetikzlibrary{positioning}

\newcommand{\myar}[2]{\ar^-{#1}[#2]}
\newcommand{\myard}[2]{\ar_-{#1}[#2]}

\SelectTips{cm}{10}

\makeatletter
\def\matrixobject@{%
  \edef \next@{={\DirectionfromtheDirection@ }}%
  \expandafter \toks@ \next@ \plainxy@
  \let\xy@@ix@=\xyq@@toksix@
  \xyFN@ \OBJECT@}
\let\xy@entry@@norm=\entry@@norm
\def\entry@@norm@patched{%
  \let\object@=\matrixobject@
  \xy@entry@@norm }
\AtBeginDocument{\let\entry@@norm\entry@@norm@patched}
\makeatother

\newcommand{\twocong}[2][0.5]{\ar@{}[#2] \save ?(#1)*{\cong}\restore}
\newcommand{\twoeq}[2][0.5]{\ar@{}[#2] \save ?(#1)*{=}\restore}
\newcommand{\ltwocell}[3][0.5]{\ar@{}[#2] \ar@{=>}?(#1)+/r 0.2cm/;?(#1)+/l 0.2cm/^{#3}}
\newcommand{\rtwocell}[3][0.5]{\ar@{}[#2] \ar@{=>}?(#1)+/l 0.2cm/;?(#1)+/r 0.2cm/^{#3}}
\newcommand{\utwocell}[3][0.5]{\ar@{}[#2] \ar@{=>}?(#1)+/d  0.2cm/;?(#1)+/u 0.2cm/_{#3}}
\newcommand{\dtwocell}[3][0.5]{\ar@{}[#2] \ar@{=>}?(#1)+/u  0.2cm/;?(#1)+/d 0.2cm/^{#3}}
\newcommand{\ultwocell}[3][0.5]{\ar@{}[#2] \ar@{=>}?(#1)+/dr  0.2cm/;?(#1)+/ul 0.2cm/^{#3}}
\newcommand{\urtwocell}[3][0.5]{\ar@{}[#2] \ar@{=>}?(#1)+/dl  0.2cm/;?(#1)+/ur 0.2cm/^{#3}}
\newcommand{\dltwocell}[3][0.5]{\ar@{}[#2] \ar@{=>}?(#1)+/ur  0.2cm/;?(#1)+/dl 0.2cm/^{#3}}
\newcommand{\drtwocell}[3][0.5]{\ar@{}[#2] \ar@{=>}?(#1)+/ul  0.2cm/;?(#1)+/dr 0.2cm/^{#3}}

\makeatother

\usepackage{babel}
  \providecommand{\definitionname}{Definition}
  \providecommand{\examplename}{Example}
  \providecommand{\lemmaname}{Lemma}
  \providecommand{\propositionname}{Proposition}
  \providecommand{\remarkname}{Remark}
\providecommand{\theoremname}{Theorem}

\begin{document}

\title{Universal Properties of Bicategories of Polynomials}

\author{Charles Walker}



\address{Department of Mathematics, Macquarie University, NSW 2109, Australia}

\email{charles.walker1@mq.edu.au}

\date{\today}
\begin{abstract}
We establish the universal properties of the bicategory of polynomials,
considering both cartesian and general morphisms between these polynomials.
A direct proof of these universal properties would be impractical
due to the complicated coherence conditions arising from polynomial
composition; however, in this paper we avoid most of these coherence
conditions using the properties of generic bicategories. 

In addition, we give a new proof of the universal properties of the
bicategory of spans, and also establish the universal properties of
the bicategory of spans with invertible 2-cells; showing how these
properties may be used to describe the universal properties of polynomials. 
\end{abstract}

\maketitle
\tableofcontents{}

\section{Introduction}

In this paper we are interested in two constructions on suitable categories
$\mathcal{E}$: the bicategory of spans $\mathbf{Span}\left(\mathcal{E}\right)$
as introduced by B{\'e}nabou \cite{ben1967}, and the bicategory
of polynomials $\mathbf{Poly}\left(\mathcal{E}\right)$ as introduced
by Gambino and Kock \cite{gambinokock}, and further studied by Weber
\cite{weber} (all to be reviewed in Section \ref{background}). Here
we wish to study the universal properties of these constructions;
that is, for an arbitrary bicategory $\mathscr{C}$ we wish to know
what it means to give a pseudofunctor $\mathbf{Span}\left(\mathcal{E}\right)\to\mathscr{C}$
or $\mathbf{Poly}\left(\mathcal{E}\right)\to\mathscr{C}$. 

In the case of spans, these results have already been established.
In particular, given any category $\mathcal{E}$ with pullbacks, one
can form a bicategory denoted $\mathbf{Span}\left(\mathcal{E}\right)$
whose objects are those of $\mathcal{E}$ and 1-cells are diagrams
in $\mathcal{E}$ of the form below
\[
\xymatrix@=1.5em{ & \bullet\ar[dl]_{s}\ar[rd]^{t}\\
\bullet &  & \bullet
}
\]
called spans. The universal property of this construction admits a
simple description since for every morphism $f$ in $\mathcal{E}$
we have adjunctions
\[
\vcenter{\hbox{\xymatrix@=1.5em{ & \bullet\ar[dl]_{\textnormal{id}}\ar[rd]^{f}\\
\bullet &  & \bullet
}
}}\qquad\dashv\qquad\vcenter{\hbox{\xymatrix@=1.5em{ & \bullet\ar[dl]_{f}\ar[rd]^{\textnormal{id}}\\
\bullet &  & \bullet
}
}}
\]
in $\mathbf{Span}\left(\mathcal{E}\right)$, and these adjunctions
generate all of $\mathbf{Span}\left(\mathcal{E}\right)$.

Indeed, it was proven by Hermida \cite[Theorem A.2]{hermida} that
composing with the canonical embedding $\mathcal{E}\hookrightarrow\mathbf{Span}\left(\mathcal{E}\right)$
describes an equivalence
\[
\begin{array}{c}
\underline{\textnormal{pseudofunctors }\mathbf{Span}\left(\mathcal{E}\right)\to\mathscr{C}}\\
\textnormal{Beck pseudofunctors }\mathcal{E}\to\mathscr{C}
\end{array}
\]
where a pseudofunctor $F_{\Sigma}\colon\mathcal{E}\to\mathscr{C}$
is Beck if for every morphism $f$ in $\mathcal{E}$ the 1-cell $F_{\Sigma}f$
has a right adjoint $F_{\Delta}f$ in $\mathscr{C}$ (such an $F_{\Sigma}$
is also known as a sinister pseudofunctor), and if the induced pair
of pseudofunctors 
\[
F_{\Sigma}\colon\mathcal{E}\to\mathscr{C},\qquad F_{\Delta}\colon\mathcal{E}^{\textnormal{op}}\to\mathscr{C}
\]
satisfy a Beck-Chevalley condition. A natural question to then ask
is what these sinister pseudofunctors correspond to when the Beck\textendash Chevalley
condition is dropped. This question was solved by Dawson, Paré, and
Pronk \cite[Theorem 2.15]{unispans} who showed that composing with
the canonical embedding describes an equivalence
\[
\begin{array}{c}
\underline{\textnormal{gregarious functors }\mathbf{Span}\left(\mathcal{E}\right)\to\mathscr{C}}\\
\textnormal{sinister pseudofunctors }\mathcal{E}\to\mathscr{C}
\end{array}
\]
where gregarious functors are the adjunction preserving normal\footnote{Here ``normal'' means the unit constraints are invertible.}
oplax functors. 

An important special case of this is when $\mathscr{C}=\mathbf{Cat}$,
where one may consider sinister pseudofunctors $\mathcal{E}\to\mathbf{Cat}$,
or equivalently cosinister\footnote{Here ``cosinister'' means arrows are sent to right adjoint 1-cells
instead of left adjoint 1-cells. This is the $F_{\Delta}$ of such
a pair $F_{\Sigma}$-$F_{\Delta}$.} pseudofunctors $\mathcal{E}^{\textnormal{op}}\to\mathbf{Cat}$. In
this case we recover the equivalence
\[
\begin{array}{c}
\underline{\textnormal{gregarious functors }\mathbf{Span}\left(\mathcal{E}\right)\to\mathbf{Cat}}\\
\textnormal{bifibrations over }\mathcal{E}
\end{array}
\]
Note also that on $\mathbf{Cat}/\mathcal{E}$ there is a KZ pseudomonad
$\varGamma_{\mathcal{E}}$ for opfibrations and a coKZ pseudomonad
$\varUpsilon_{\mathcal{E}}$ for fibrations. This yields (via a pseudo-distributive
law) the pseudomonad $\varGamma_{\mathcal{E}}\varUpsilon_{\mathcal{E}}$
for bifibrations satisfying the Beck\textendash Chevalley condition
\cite{glehn} (also known as fibrations with sums). The above equivalence
then restricts to
\[
\begin{array}{c}
\underline{\textnormal{pseudofunctors }\mathbf{Span}\left(\mathcal{E}\right)\to\mathbf{Cat}}\\
\textnormal{fibrations with sums over }\mathcal{E}
\end{array}
\]
An archetypal example of this is the codomain fibration over $\mathcal{E}$
corresponding to the canonical pseudofunctor $\mathbf{Span}\left(\mathcal{E}\right)\to\mathbf{Cat}$
defined by 
\[
\vcenter{\hbox{\xymatrix@=1.5em{ & T\ar[dl]_{s}\ar[rd]^{t}\\
X &  & Y
}
}}\;\;\mapsto\;\;\vcenter{\hbox{\xymatrix@=1.5em{\mathcal{E}/X\myar{\Delta_{s}}{r} & \mathcal{E}/T\myar{\Sigma_{t}}{r} & \mathcal{E}/Y}
}}
\]
where for every morphism $f$ in $\mathcal{E}$, the functor $\Sigma_{f}$
denotes composition with $f$, and the functor $\Delta_{f}$ denotes
pulling back along $f$.

When considering polynomials it is convenient to assume some extra
structure on $\mathcal{E}$. In particular, we will take $\mathcal{E}$
to be a category with finite limits, such that for each morphism $f$
in $\mathcal{E}$ the ``pullback along $f$'' functor $\Delta_{f}$
has a right adjoint $\Pi_{f}$. For such a category $\mathcal{E}$
(known as a locally cartesian closed category) one can form a bicategory
denoted $\mathbf{Poly}\left(\mathcal{E}\right)$ whose objects are
those of $\mathcal{E}$ and 1-cells are diagrams in $\mathcal{E}$
of the form below
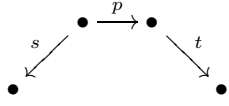
\[
\xymatrix@=1.5em{ & \bullet\ar[dl]_{s}\ar[r]^{p} & \bullet\ar[rd]^{t}\\
\bullet &  &  & \bullet
}
\]
called polynomials\footnote{The bicategory of polynomials can be defined on any category $\mathcal{E}$
with pullbacks \cite{weber}; however, we will assume local cartesian
closure for simplicity.}. One can also form a bicategory $\mathbf{Poly}_{c}\left(\mathcal{E}\right)$
with the same objects and 1-cells by being more restrictive on the
2-cells (that is, only taking ``cartesian'' morphisms of polynomials).

The purpose of this paper is to describe the universal properties
of these two bicategories of polynomials.

Similar to the case of spans, the universal property of $\mathbf{Poly}\left(\mathcal{E}\right)$
admits a simple description since for every morphism $f$ in $\mathcal{E}$
we have adjunctions
\begin{equation}
\vcenter{\hbox{\xymatrix@=1.5em{ & \bullet\ar[dl]_{\textnormal{id}}\ar[r]^{\textnormal{id}} & \bullet\ar[rd]^{f}\\
\bullet &  &  & \bullet
}
}}\;\;\dashv\;\;\vcenter{\hbox{\xymatrix@=1.5em{ & \bullet\ar[dl]_{f}\ar[r]^{\textnormal{id}} & \bullet\ar[rd]^{\textnormal{id}}\\
\bullet &  &  & \bullet
}
}}\;\;\dashv\;\;\vcenter{\hbox{\xymatrix@=1.5em{ & \bullet\ar[dl]_{\textnormal{id}}\ar[r]^{f} & \bullet\ar[rd]^{\textnormal{id}}\\
\bullet &  &  & \bullet
}
}}\label{adjtrip}
\end{equation}
in $\mathbf{Poly}\left(\mathcal{E}\right)$, and these adjunctions
generate all of $\mathbf{Poly}\left(\mathcal{E}\right)$ (to be shown
in Proposition \ref{polyadjprop}). Using this fact, we show that
in the case of polynomials with general 2-cells, composition with
the embedding $\mathcal{E}\hookrightarrow\mathbf{Poly}\left(\mathcal{E}\right)$
describes the equivalence
\begin{equation}
\begin{array}{c}
\textnormal{ pseudofunctors }\mathbf{Poly}\left(\mathcal{E}\right)\to\mathscr{C}\\
\overline{\textnormal{DistBeck pseudofunctors }\mathcal{E}\to\mathscr{C}}
\end{array}\label{polygeneraluni}
\end{equation}
where a pseudofunctor $F_{\Sigma}\colon\mathcal{E}\to\mathscr{C}$
is DistBeck if for every morphism $f$ in $\mathcal{E}$ the 1-cell
$F_{\Sigma}f$ has two successive right adjoints $F_{\Delta}f$ and
$F_{\Pi}f$ (such an $F_{\Sigma}$ is called a 2-sinister pseudofunctor),
and if the induced triple of pseudofunctors
\[
F_{\Sigma}\colon\mathcal{E}\to\mathscr{C},\qquad F_{\Delta}\colon\mathcal{E}^{\textnormal{op}}\to\mathscr{C},\qquad F_{\Pi}\colon\mathcal{E}\to\mathscr{C}
\]
satisfies the earlier Beck-Chevalley condition on the pair $F_{\Sigma}$
and $F_{\Delta}$, in addition to a ``distributivity condition''
on the pair $F_{\Sigma}$ and $F_{\Pi}$. Forgetting the distributivity
condition yields the notion of a 2-Beck pseudofunctor, so that \eqref{polygeneraluni}
may be seen as a restriction of an equivalence
\[
\begin{array}{c}
\underline{\textnormal{gregarious functors }\mathbf{Poly}\left(\mathcal{E}\right)\to\mathscr{C}}\\
\textnormal{2-Beck pseudofunctors }\mathcal{E}\to\mathscr{C}
\end{array}
\]

Similar to earlier, an important special case of this is when $\mathscr{C}=\mathbf{Cat}$,
where one recovers the equivalence
\[
\begin{array}{c}
\textnormal{gregarious functors }\mathbf{Poly}\left(\mathcal{E}\right)\to\mathbf{Cat}\\
\overline{\textnormal{fibrations with sums and products over }\mathcal{E}}
\end{array}
\]

Note also that on $\mathbf{Fib}\left(\mathcal{E}\right)$ there is
a KZ pseudomonad $\Sigma_{\mathcal{E}}$ for fibrations with sums,
and a coKZ pseudomonad $\Pi_{\mathcal{E}}$ for fibrations with products.
This yields (via a pseudo-distributive law) a pseudomonad $\Sigma_{\mathcal{E}}\Pi_{\mathcal{E}}$
for fibrations with sums and products which satisfy a distributivity
condition \cite{glehn}. Here we recover the equivalence
\[
\begin{array}{c}
\textnormal{pseudofunctors }\mathbf{Poly}\left(\mathcal{E}\right)\to\mathbf{Cat}\\
\overline{\textnormal{distributive fibrations with sums and products over }\mathcal{E}}
\end{array}
\]
The codomain fibration is again an archetypal example of this, with
the corresponding canonical pseudofunctor $\mathbf{Poly}\left(\mathcal{E}\right)\to\mathbf{Cat}$
being defined by
\[
\vcenter{\hbox{\xymatrix@=1.5em{ & E\ar[dl]_{s}\ar[r]^{p} & B\ar[rd]^{t}\\
I &  &  & J
}
}}\;\;\mapsto\;\;\vcenter{\hbox{\xymatrix@=1.5em{\mathcal{E}/I\myar{\Delta_{s}}{r} & \mathcal{E}/E\myar{\Pi_{p}}{r} & \mathcal{E}/B\myar{\Sigma_{t}}{r} & \mathcal{E}/J}
}}
\]
which is how one assigns a polynomial to a polynomial functor.

Another example of this situation is given by taking $\mathcal{E}$
to be a regular locally cartesian closed category. In this case we
have the 2-Beck pseudofunctor $\textnormal{Sub}\colon\mathcal{E}\to\mathbf{Cat}$
which sends a morphism $f\colon X\to Y$ in $\mathcal{E}$ to the
existential quantifier $\exists_{f}\colon\textnormal{Sub}\left(X\right)\to\textnormal{Sub}\left(Y\right)$
mapping subobjects of $X$ to those of $Y$, which has the two successive
right adjoints $\Delta_{f}$ ``pullback along $f$'' and $\forall_{f}$
``universal quantification at $f$'', thus giving a gregarious functor
$\mathbf{Poly}\left(\mathcal{E}\right)\to\mathbf{Cat}$ defined by
the assignment
\[
\vcenter{\hbox{\xymatrix@=1.5em{ & E\ar[dl]_{s}\ar[r]^{p} & B\ar[rd]^{t}\\
I &  &  & J
}
}}\;\;\mapsto\;\;\vcenter{\hbox{\xymatrix@=1.5em{\textnormal{Sub}\left(I\right)\myar{\Delta_{s}}{r} & \textnormal{Sub}\left(E\right)\myar{\forall_{p}}{r} & \textnormal{Sub}\left(B\right)\myar{\exists_{t}}{r} & \textnormal{Sub}\left(J\right)}
}}
\]
The distributivity condition here then amounts to asking that $\mathcal{E}$
satisfies the internal axiom of choice.

With only cartesian morphisms we do not have the adjunctions on the
right in \eqref{adjtrip} since the units and counits of such adjunctions
are not cartesian in general, thus making the universal property of
$\mathbf{Poly}_{c}\left(\mathcal{E}\right)$ more complicated to state.
The universal property of this construction is described as an equivalence
\[
\begin{array}{c}
\underline{\textnormal{pseudofunctors }\mathbf{Poly}_{c}\left(\mathcal{E}\right)\to\mathscr{C}}\\
\textnormal{DistBeck triple }\mathcal{E}\to\mathscr{C}
\end{array}
\]
where a DistBeck triple $\mathcal{E}\to\mathscr{C}$ is a triple of
pseudofunctors 
\[
F_{\Sigma}\colon\mathcal{E}\to\mathscr{C},\qquad F_{\Delta}\colon\mathcal{E}^{\textnormal{op}}\to\mathscr{C},\qquad F_{\otimes}\colon\mathcal{E}\to\mathscr{C}
\]
such $F_{\Sigma}f\dashv F_{\Delta}f$ for all morphisms $f$ in $\mathcal{E}$,
with a Beck\textendash Chevalley condition satisfied for the pair
$F_{\Sigma}$ and $F_{\Delta}$, for which $F_{\Delta}$ and $F_{\otimes}$
are related via invertible Beck\textendash Chevalley coherence data
(as we do not have adjunctions $F_{\Delta}f\dashv F_{\otimes}f$ this
data does not come for free and must be given instead, subject to
suitable coherence axioms), such that the pair $F_{\Sigma}$ and $F_{\otimes}$
satisfy a distributivity condition as before\footnote{The distributivity data need not be given as it may be constructed
using the $F_{\Delta}$-$F_{\otimes}$ Beck coherence data and the
adjunctions $F_{\Sigma}f\dashv F_{\Delta}f$.}. There are also weakened versions of the universal property of $\mathbf{Poly}_{c}\left(\mathcal{E}\right)$
which arise from dropping these conditions.

An example of this is given by taking $\mathcal{E}$ to be the category
of finite sets $\mathbf{FinSet}$ and $\mathscr{C}$ to be the 2-category
of small categories $\mathbf{Cat}$. Taking $\left(\mathcal{A},\otimes,I\right)$
to be a symmetric monoidal category such that $\mathcal{A}$ has finite
coproducts, we can assign to any finite set $n$ the category $\mathcal{A}^{n}$
and to any morphism $f\colon m\to n$ the  functors
\[
\begin{aligned}\textnormal{lan}_{f}\colon & \mathcal{A}^{m}\to\mathcal{A}^{n}, &  & \left(a_{i}\colon i\in m\right)\mapsto\left(\Sigma{}_{x\in f^{-1}\left(j\right)}\;a_{x}\colon j\in n\right)\\
\left(-\right)\circ f\colon & \mathcal{A}^{n}\to\mathcal{A}^{m}, &  & \left(a_{j}\colon j\in n\right)\mapsto\left(a_{f\left(i\right)}\colon i\in m\right)\\
\otimes_{f}\colon & \mathcal{A}^{m}\to\mathcal{A}^{n}, &  & \left(a_{i}\colon i\in m\right)\mapsto\left(\otimes_{x\in f^{-1}\left(j\right)}\;a_{x}\colon j\in n\right)
\end{aligned}
\]
This gives the data of a Beck triple (that is a DistBeck triple without
requiring the distributivity condition). The distributivity condition
here holds precisely when the functor $X\otimes\left(-\right)\colon\mathcal{A}\to\mathcal{A}$
preserves finite coproducts for all $X\in\mathcal{A}$.

The reader should note that proving the universal properties concerning
the polynomial construction is much more complex than that of the
span construction. This is since composition of polynomials is significantly
more complicated; this is especially evident in calculations involving
associativity of polynomial composition being respected by an oplax
or pseudofunctor, or calculations involving horizontal composition
of general polynomial morphisms.

Fortunately, we are able to avoid these calculations to some extent.
This is done by exploiting the fact that both $\mathbf{Span}\left(\mathcal{E}\right)$
and $\mathbf{Poly}_{c}\left(\mathcal{E}\right)$ are ``generic bicategories''
\cite{WalkerGeneric}, that is a bicategory $\mathscr{A}$ with the
property that each composition functor 
\[
\circ_{X,Y,Z}\colon\mathscr{A}_{Y,Z}\times\mathscr{A}_{X,Y}\to\mathscr{A}_{X,Z}
\]
 admits generic factorisations. The main result of \cite{WalkerGeneric}
shows that oplax functors out of such bicategories admit a much simpler
description; thus allowing for a simple description of oplax functors
out $\mathbf{Span}\left(\mathcal{E}\right)$ and $\mathbf{Poly}_{c}\left(\mathcal{E}\right)$.
A problem here is that the bicategory $\mathbf{Poly}\left(\mathcal{E}\right)$
does not enjoy this property. However, as $\mathbf{Poly}_{c}\left(\mathcal{E}\right)$
embeds into $\mathbf{Poly}\left(\mathcal{E}\right)$ and both bicategories
have the same composition the universal property of the former will
assist in proving the latter.

In Section \ref{background} we give the necessary background for
this paper. We recall the definitions and basic properties of the
bicategories of spans and polynomials, the notions of lax, oplax and
gregarious functors, the basic properties of the mates correspondence,
and the basic properties of generic bicategories.

In Section \ref{unispans} we give a proof of the universal properties
of spans using the properties of generic bicategories. This is to
give a complete and detailed proof of these properties demonstrating
our method, before applying it the  more complicated setting of polynomials
later on.

In Section \ref{unispaniso} we give a proof of the universal properties
of spans with invertible 2-cells. This is necessary since the universal
properties of polynomials with cartesian 2-cells will be described
in terms of this property. 

In Section \ref{unipolycart} we give a proof of the universal properties
of polynomials with cartesian 2-cells. It is in this section that
our method is of the most use; indeed in our proof we completely avoid
coherences involving composition of distributivity pullbacks (the
worst coherence conditions which would arise in a direct proof).

In Section \ref{unipoly} we give a proof of the universal properties
of polynomials with general 2-cells, by using the corresponding properties
for polynomials with cartesian 2-cells and checking some additional
coherence conditions concerning naturality with respect to these more
general 2-cells.

\section{Background\label{background}}

In this section we give the necessary background knowledge for this
paper. 

\subsection{The bicategory of spans}

Before studying the bicategory of polynomials we will study the simpler
and more well known construction of the bicategory of spans, as introduced
by B{\'e}nabou \cite{ben1967}.
\begin{defn}
Suppose we are given a category $\mathcal{E}$ with chosen pullbacks.
We may then form \emph{bicategory of spans} in $\mathcal{E}$, denoted
$\mathbf{Span}\left(\mathcal{E}\right)$, with objects those of $\mathcal{E}$,
1-cells $A\nrightarrow B$ given diagrams in $\mathcal{E}$ of the
form
\[
\xymatrix@=1em{ & X\ar[rd]^{q}\ar[ld]_{p}\\
A &  & B
}
\]
called spans, composition of 1-cells given by taking the chosen pullback
\[
\xymatrix@=1em{ &  & \bullet\ar[rd]^{\pi_{2}}\ar[dl]_{\pi_{1}}\\
 & X\ar[rd]^{q}\ar[ld]_{p} &  & Y\ar[rd]^{s}\ar[ld]_{r}\\
A &  & B &  & C
}
\]
and 2-cells $\nu$ given by those morphisms between the vertices of
two spans which yield commuting diagrams of the form 
\[
\xymatrix@=1em{ & X\ar[rd]^{q}\ar[ld]_{p}\ar[dd]^{\nu}\\
A &  & B\\
 & Y\ar[ru]_{r}\ar[lu]^{s}
}
\]
The identity 1-cells are given by identity spans $\xymatrix@=1em{X & X\ar[r]^{1_{X}}\ar[l]_{1_{X}} & X}
$ and composition extends to 2-cells by the universal property of pullbacks.
The essential uniqueness of the limit of a diagram 
\[
\xymatrix@=1em{ & X\ar[rd]^{q}\ar[ld]_{p} &  & Y\ar[rd]^{s}\ar[ld]_{r} &  & Z\ar[rd]^{v}\ar[ld]_{u}\\
A &  & B &  & C &  & D
}
\]
yields the associators, making $\mathbf{Span}\left(\mathcal{E}\right)$
into a bicategory.
\end{defn}

We denote by $\mathbf{Span}_{\textnormal{iso}}\left(\mathcal{E}\right)$
the bicategory as defined above, but only taking the invertible 2-cells.

\subsection{The bicategory of polynomials}

In the earlier defined bicategory of spans the morphisms may be viewed
as multivariate linear maps (matrices). In this subsection we recall
the bicategory of polynomials, whose morphisms may be viewed as multivariate
polynomials, and whose study has applications in areas including theoretical
computer science (under the name of containers and indexed containers
\cite{containers}) and the theory of $W$-types \cite{martinlof,wellfoundedtrees}.

Before we can define this bicategory we must recall the notion of
distributivity pullback as given by Weber \cite{weber}.
\begin{defn}
Given two composable morphisms $u\colon X\to A$ and $f\colon A\to B$
in a category $\mathcal{E}$ with pullbacks, we say that:
\begin{enumerate}
\item a \emph{pullback around} $\left(f,u\right)$ is a diagram 
\[
\xymatrix@=1.5em{T\ar[r]^{p}\ar[d]_{q} & X\ar[r]^{u} & A\ar[d]^{f}\\
Y\ar[rr]_{r} &  & B
}
\]
such that the outer rectangle is a pullback, and a \emph{morphism
of pullbacks around }$\left(f,u\right)$ is a pair of morphisms $s\colon T\to T'$
and $t\colon Y\to Y'$ such that $p's=p$, $q's=tq$ and $r=r't$;
\item a \emph{distributivity pullback} around $\left(f,u\right)$ is a terminal
object in the category of pullbacks around $\left(f,u\right)$.
\end{enumerate}
\end{defn}

We also recall the notion of an exponentiable morphism, a condition
which ensures the existence of such distributivity pullbacks.
\begin{defn}
We say a morphism $f\colon A\to B$ in a category $\mathcal{E}$ with
pullbacks is \emph{exponentiable} if the ``pullback along $f$''
functor $\Delta_{f}\colon\mathcal{E}/B\to\mathcal{E}/A$ has a right
adjoint. We will denote this right adjoint by $\Pi_{f}$ when it exists.
\end{defn}

\begin{rem}
Note that such an $f$ is exponentiable if and only if for every $u$
there exists a distributivity pullback around $\left(f,u\right)$
\cite{weber}.
\end{rem}

The following diagrams are to be the morphisms in the bicategory of
polynomials. 
\begin{defn}
A \emph{polynomial }$P\colon I\nrightarrow J$ in a category $\mathcal{E}$
with pullbacks is a diagram of the form
\[
\xymatrix@=1em{ & E\ar[r]^{p}\ar[ld]_{s} & B\ar[rd]^{t}\\
I &  &  & J
}
\]
where $p$ is exponentiable.
\end{defn}

We will also need the following universal property of polynomial composition.
\begin{prop}
\label{unipolycomp}\cite[Prop. 3.1.6]{weber} Suppose we are given
two polynomials $P\colon I\nrightarrow J$ and $Q\colon J\nrightarrow K$.
Consider a category with objects given by commuting diagrams of the
form
\[
\xymatrix@=1em{ &  & A_{1}\ar[dl]\ar[r] & A_{2}\ar[dl]\ar[dr]\ar[r] & A_{3}\ar[dr]\\
 & E\ar[r]\ar[ld] & B\ar[rd] &  & M\ar[r]\ar[ld] & N\ar[rd]\\
I &  &  & J &  &  & K
}
\]
for which the left and right squares are pullbacks (but not necessarily
the middle), and morphisms given by triples $\left(A_{i}\to B_{i}\colon i=1,2,3\right)$
rendering commutative the diagram
\[
\xymatrix@=1em{ &  & A_{1}\ar@{->}[d]\ar[r]\ar@{->}[ddl] & A_{2}\ar@{->}[d]\ar[r]\ar@{-->}[ddl]\ar@{-->}[ddr] & A_{3}\ar@{->}[d]\ar@{->}[ddr]\\
 &  & B_{1}\ar[dl]\ar[r] & B_{2}\ar[dl]\ar[dr]\ar[r] & B_{3}\ar[dr]\\
 & E\ar[r]\ar[ld] & B\ar[rd] &  & M\ar[r]\ar[ld] & N\ar[rd]\\
I &  &  & J &  &  & K
}
\]
Then in this category, the outside composite in the diagram formed
below (which is a polynomial $I\nrightarrow K$)
\begin{equation}
\xymatrix@=1em{ & \bullet\ar[dd]\ar[rr] &  & \bullet\ar[d]\ar[rr] &  & \bullet\ar[dd]\\
 &  & \ar@{}[]|-{\textnormal{pb}} & \bullet\ar[rd]\ar[ld] & \ar@{}[]|-{\textnormal{dpb}}\\
 & E\ar[r]^{p}\ar[ld]_{s} & B\ar[rd]^{t} & \ar@{}[]|-{\textnormal{pb}} & M\ar[r]^{q}\ar[ld]_{u} & N\ar[rd]^{v} &  &  & \;\\
I &  &  & J &  &  & K
}
\label{polycompdiagram}
\end{equation}
is a terminal object.
\end{prop}

\begin{defn}
Suppose we are given a locally cartesian closed category $\mathcal{E}$
with chosen pullbacks and distributivity pullbacks. We may then form
the \emph{bicategory of polynomials with cartesian 2-cells} in $\mathcal{E}$,
denoted $\mathbf{Poly}_{c}\left(\mathcal{E}\right)$, with objects
those of $\mathcal{E}$, 1-cells $A\nrightarrow B$ given by polynomials,
composition of 1-cells given by forming the diagram \eqref{polycompdiagram}
just above, and cartesian 2-cells given by pairs of morphisms $\left(\sigma,\nu\right)$
rendering commutative the diagram
\[
\xymatrix@=1.5em{ & E\ar[r]^{p}\ar[ld]_{s}\ar[dd]_{\sigma} & B\ar[rd]^{t}\ar[dd]^{\nu}\\
I & \ar@{}[r]|-{\textnormal{pb}} & \; & J\\
 & M\ar[r]_{q}\ar[ul]^{u} & N\ar[ur]_{v}
}
\]
such that the middle square is a pullback. The identity 1-cells are
given by identity polynomials $\xymatrix@=1em{X & X\ar[r]^{1_{X}}\ar[l]_{1_{X}} & X\ar[r]^{1_{X}} & X.}
$ Composition of 2-cells and the associators may be recovered from
Proposition \ref{unipolycomp} above.
\end{defn}

\begin{defn}
Suppose we are given a locally cartesian closed category $\mathcal{E}$
with chosen pullbacks and distributivity pullbacks. We may then form
the \emph{bicategory of polynomials with general 2-cells, }denoted
$\mathbf{Poly}\left(\mathcal{E}\right)$, with objects and 1-cells
as in $\mathbf{Poly}_{c}\left(\mathcal{E}\right)$, and 2-cells given
by diagrams as below on the left below
\[
\xymatrix{ & E\ar[dl]_{s}\ar[r]^{p} & B\ar[dr]^{t} &  &  &  & E\ar[dl]_{s}\ar[r]^{p} & B\ar[dr]^{t}\\
X & S_{1}\ar[d]_{f_{1}}\ar[r]^{pe_{1}}\ar[u]^{e_{1}}\ar@{}[rd]|-{\textnormal{pb}} & B\ar@{=}[u]\ar[d]^{g} & Y & \approx & X & S_{2}\ar[d]_{f_{2}}\ar[r]^{pe_{2}}\ar[u]^{e_{2}}\ar@{}[rd]|-{\textnormal{pb}} & B\ar@{=}[u]\ar[d]^{g} & Y\\
 & M\ar[r]_{q}\ar[ul]^{u} & N\ar[ur]_{v} &  &  &  & M\ar[r]_{q}\ar[ul]^{u} & N\ar[ur]_{v}
}
\]
regarded equivalent to the diagram on the right provided both indicated
regions are pullbacks.

For the other operations of this bicategory such as the composition
operation on 2-cells we refer the reader to the equivalence $\mathbf{Poly}\left(\mathcal{E}\right)\simeq\mathbf{PolyFun}\left(\mathcal{E}\right)$
\cite{gambinokock} where $\mathbf{PolyFun}\left(\mathcal{E}\right)$
is the bicategory of polynomial functors, described later in Example
\ref{canonicalemb}.
\end{defn}

\begin{rem}
Note that it suffices to give local equivalences $\mathbf{PolyFun}\left(\mathcal{E}\right)_{X,Y}\simeq\mathbf{Poly}\left(\mathcal{E}\right)_{X,Y}$
since from this it follows that the bicategorical structure on $\mathbf{PolyFun}\left(\mathcal{E}\right)$
endows the family of hom-categories $\mathbf{Poly}\left(\mathcal{E}\right)_{X,Y}$
with the structure of a bicategory via doctrinal adjunction \cite{doctrinal}.
This describes the bicategory structure on $\mathbf{Poly}\left(\mathcal{E}\right)$.
\end{rem}

\subsection{Morphisms of bicategories}

There are a few types of morphisms between bicategories we are interested
in for this paper. These include oplax functors, lax functors, pseudofunctors,
gregarious functors and sinister pseudofunctors. After the following
trivial definition we will recall these notions.
\begin{defn}
Given two bicategories $\mathscr{A}$ and $\mathscr{B}$ , a \emph{locally
defined functor} $F\colon\mathscr{A}\to\mathscr{B}$ consists of:
\begin{itemize}
\item for each object $X\in\mathscr{A}$ an object $FX\in\mathscr{B}$;
\item for each pair of objects $X,Y\in\mathscr{A}$, a functor $F_{X,Y}\colon\mathscr{A}_{X,Y}\to\mathscr{B}_{FX,FY}$;
\end{itemize}
subject to no additional conditions.
\end{defn}

It is one of the main points of this paper that many of the coherence
conditions arising from the associativity diagram \eqref{assocdiag}
for oplax functors out of the bicategories $\mathbf{Span}\left(\mathcal{E}\right)$
and $\mathbf{Poly}_{c}\left(\mathcal{E}\right)$ may be avoided (for
suitable categories $\mathcal{E}$). 
\begin{defn}
Given two bicategories $\mathscr{A}$ and $\mathscr{B}$, a \emph{lax
functor} $F\colon\mathscr{A}\to\mathscr{B}$ is a locally defined
functor $F\colon\mathscr{A}\to\mathscr{B}$ equipped with 

\begin{itemize}
\item for each object $X\in\mathscr{A}$, a 2-cell $\lambda_{X}\colon1_{FX}\to F1_{X}$;
\item for each triple of objects $X,Y,Z\in\mathscr{A}$ and pair of morphisms
$f\colon X\to Y$ and $g\colon Y\to Z$, a 2-cell $\varphi_{g,f}\colon Fg\cdot Ff\to Fgf$
natural in $g$ and $f$,
\end{itemize}
such that the constraints render commutative the associativity diagram
\begin{equation}
\xymatrix@=1em{Fh\cdot\left(Fg\cdot Ff\right)\myar{Fh\cdot\varphi_{g,f}}{rr} &  & Fh\cdot F\left(gf\right)\myar{\varphi_{h,gf}}{rr} &  & F\left(h\left(gf\right)\right)\\
\\
\left(Fh\cdot Fg\right)\cdot Ff\myard{\varphi_{h,g}\cdot Ff}{rr}\myar{a_{Fh,Fg,Ff}}{uu} &  & F\left(hg\right)\cdot Ff\myard{\varphi_{hg,f}}{rr} &  & F\left(\left(hg\right)f\right)\myard{F\left(\tilde{a}_{h,g,f}\right)}{uu}
}
\label{assocdiag}
\end{equation}
for composable morphisms $h$,$g$ and $f$. In addition, the nullary
constraint cells must render commutative the diagrams
\[
\xymatrix@=1em{Ff\cdot1_{FX}\myar{Ff\cdot\lambda_{X}}{rr}\myard{r_{Ff}}{dd} &  & Ff\cdot F\left(1_{X}\right)\myar{\varphi_{f,1_{X}}}{dd} &  &  & 1_{FY}\cdot Ff\myar{\lambda_{Y}\cdot Ff}{rr}\myard{l_{Ff}}{dd} &  & F\left(1_{Y}\right)\cdot Ff\myar{\varphi_{1_{Y},f}}{dd}\\
\\
Ff &  & F\left(f\cdot1_{X}\right)\myar{F\left(r_{f}\right)}{ll} &  &  & Ff &  & F\left(1_{Y}\cdot f\right)\myar{F\left(l_{f}\right)}{ll}
}
\]
for all morphisms $f\colon X\to Y$. If the direction of the constraints
$\varphi$ and $\lambda$ is reversed, this is the definition of an
\emph{oplax functor}. If the nullary constraints $\lambda$ are invertible
(in either the lax or oplax case) we then say our functor is \emph{normal.}
If both types of constraint cells $\varphi$ and $\lambda$ are required
invertible, then this is the definition of a \emph{pseudofunctor}. 

\end{defn}

\begin{example}
\label{canonicalemb} It is well known that given a category $\mathcal{E}$
with pullbacks there is a pseudofunctor $\mathbf{Span}\left(\mathcal{E}\right)\to\mathbf{Cat}$
which assigns an object $X\in\mathcal{E}$ to the slice category $\mathcal{E}/X$
and on spans is defined the assignment 
\[
\vcenter{\hbox{\xymatrix@=1.5em{ & B\ar[dl]_{s}\ar[rd]^{t}\\
I &  & J
}
}}\qquad\mapsto\qquad\vcenter{\hbox{\xymatrix{\mathcal{E}/I\ar[r]^{\Delta_{s}} & \mathcal{E}/B\ar[r]^{\Sigma_{t}} & \mathcal{E}/J}
}}
\]
where $\Sigma_{t}$ is the ``composition with $t$'' functor, and
$\Delta_{s}$ is the ``pullback along $s$'' functor (the right
adjoint of $\Sigma_{s}$).

If $\mathcal{E}$ is locally cartesian closed, meaning that for each
morphism $p$ the functor $\Delta_{p}$ has a further right adjoint
denoted $\Pi_{p}$, then there is also such a canonical functor out
of $\mathbf{Poly}\left(\mathcal{E}\right)$ \cite{gambinokock} and
$\mathbf{Poly}_{c}\left(\mathcal{E}\right)$ \cite{weber}, which
assigns an object $X\in\mathcal{E}$ to the slice category $\mathcal{E}/X$
and on polynomials is defined the assignment 
\[
\vcenter{\hbox{\xymatrix@=1.5em{ & E\ar[dl]_{s}\ar[r]^{p} & B\ar[rd]^{t}\\
I &  &  & J
}
}}\qquad\mapsto\qquad\vcenter{\hbox{\xymatrix{\mathcal{E}/I\ar[r]^{\Delta_{s}} & \mathcal{E}/E\ar[r]^{\Pi_{p}} & \mathcal{E}/B\ar[r]^{\Sigma_{t}} & \mathcal{E}/J}
}}
\]

A functor isomorphic to one as on the right above is known as a \emph{polynomial
functor.} The objects of $\mathcal{E}$, polynomial functors, and
strong natural transformations form a 2-category $\mathbf{PolyFun}\left(\mathcal{E}\right)$
\cite{gambinokock}.
\end{example}

\begin{rem}
In the subsequent sections we are interested in pseudofunctors mapping
into a general bicategory $\mathscr{C}$, not just $\mathbf{Cat}$,
however we will still use the above example to motivate our notation.
\end{rem}

The following is a special type of oplax functor which turns up when
studying the universal properties of the span construction \cite{unispans,spanconstruction}.
This notion will also be useful for studying the universal properties
of the polynomial construction.
\begin{defn}
\cite[Definition 2.4]{unispans} We say a normal oplax functor of
bicategories $F\colon\mathscr{A}\to\mathscr{B}$ is \emph{gregarious
}(also known as\emph{ jointed}) if for any pair of 1-cells $f\colon A\to B$
and $g\colon B\to C$ in $\mathscr{A}$ for which $g$ has a right
adjoint, the constraint cell $\varphi_{g,f}\colon F\left(gf\right)\to Fg\cdot Ff$
is invertible. 
\end{defn}

There is also an alternative characterization of gregarious functors
worth mentioning, which establishes gregarious functors as a natural
concept.
\begin{prop}
\cite[Propositions 2.8 and 2.9]{unispans} A normal oplax functor
of bicategories $F\colon\mathscr{A}\to\mathscr{B}$ is gregarious
if and only if it preserves adjunctions; that is, if for every adjunction
$\left(f\dashv u\colon A'\to A,\eta,\epsilon\right)$ in $\mathscr{A}$
there exists 2-cells $\overline{\eta}:1_{FA}\to Fu\cdot Ff$ and $\overline{\epsilon}:Ff\cdot Fu\to1_{FA'}$
which exhibit $Ff$ as left adjoint to $Fu$ and render commutative
the squares
\[
\xymatrix{F\left(1_{A}\right)\ar[r]^{F\eta}\ar[d]_{\lambda_{A}} & F\left(uf\right)\ar[d]^{\varphi_{u,f}} &  & F\left(fu\right)\ar[r]^{F\epsilon}\ar[d]_{\varphi_{f,u}} & F\left(1_{A'}\right)\ar[d]^{\lambda_{A'}}\\
1_{FA}\ar[r]_{\overline{\eta}} & Fu\cdot Ff &  & Ff\cdot Fu\ar[r]_{\overline{\epsilon}} & 1_{FA'}
}
\]
\end{prop}

We also need a notion of morphism between lax, oplax, gregarious or
pseudofunctors. It will be convenient here to use Lack's icons \cite{icons},
defined as follows.
\begin{defn}
Given two lax functors $F,G\colon\mathscr{A}\to\mathscr{B}$ which
agree on objects, an \emph{icon} $\alpha\colon F\Rightarrow G$ consists
of a family of natural transformations
\[
\xymatrix{\mathscr{A}_{X,Y}\ar@/^{1pc}/[rr]^{F_{X,Y}}\ar@/_{1pc}/[rr]_{G_{X,Y}}\ar@{}[rr]|-{\Downarrow\alpha_{X,Y}} &  & \mathscr{B}_{FX,FY}, & X,Y\in\mathscr{A}}
\]
with components rendering commutative the diagrams
\[
\xymatrix@=1em{Fg\cdot Ff\ar[rr]^{\varphi_{g,f}}\ar[dd]_{\alpha_{g}\ast\alpha_{f}} &  & F\left(gf\right)\ar[dd]^{\alpha_{gf}} &  &  &  & 1_{FX}\ar[rrdd]^{\omega_{X}}\ar[dd]_{\lambda_{X}}\\
\\
Gg\cdot Gf\ar[rr]_{\psi_{g,f}} &  & G\left(gf\right) &  &  &  & F1_{X}\ar[rr]_{\alpha_{1_{X}}} &  & G1_{X}
}
\]
for composable morphisms $f$ and $g$ in $\mathscr{A}$. Similarly,
one may define icons between oplax functors.
\end{defn}

An important point about icons is that there is a 2-category of bicategories,
oplax (lax) functors, and icons. For convenience, we make the following
definition.
\begin{defn}
We denote by $\mathbf{Icon}$ (resp. $\mathbf{Greg}$) the 2-category
of bicategories, pseudofunctors (resp. gregarious functors) and icons.
\end{defn}

Finally, we recall the notion of a sinister pseudofunctor, as well
as the notion of a sinister pseudofunctor which satisfies a certain
Beck condition. These notions are to be used regularly through the
paper.
\begin{defn}
\label{defSinBeck} Let $\mathcal{E}$ be a category seen as a locally
discrete 2-category, and let $\mathscr{C}$ be a bicategory. We say
a pseudofunctor $F\colon\mathcal{E}\to\mathscr{C}$ of bicategories
is \emph{sinister }if for every morphism $f$ in $\mathcal{E}$ the
1-cell $Ff$ has a right adjoint in $\mathscr{C}$.

Supposing further that $\mathcal{E}$ has pullbacks, for any pullback
square in $\mathcal{E}$ as on the left below, we may apply $F$ and
compose with pseudofunctoriality constraints giving an invertible
2-cell as in the middle square below, and then take mates to get a
2-cell as on the right below
\[
\xymatrix{\bullet\ar[d]_{g'}\myar{f'}{r} & \bullet\ar[d]^{g} &  & \bullet\ar[d]_{Fg'}\myar{Ff'}{r}\ar@{}[dr]|-{\cong} & \bullet\ar[d]^{Fg} &  & \bullet\myar{F_{\Sigma}f'}{r}\ar@{}[dr]|-{\Downarrow\mathfrak{b}_{f,g}^{f',g'}} & \bullet\\
\bullet\ar[r]_{f} & \bullet &  & \bullet\ar[r]_{Ff} & \bullet &  & \bullet\ar[r]_{F_{\Sigma}f}\ar[u]^{F_{\Delta}g'} & \bullet\ar[u]_{F_{\Delta}g}
}
\]
We say the sinister pseudofunctor $F\colon\mathcal{E}\to\mathscr{C}$
satisfies the \emph{Beck condition} if every such $\mathfrak{b}_{f,g}^{f',g'}$
as on the right above is invertible.

We will denote by $\mathbf{Sin}\left(\mathcal{E},\mathscr{C}\right)$
the category of sinister pseudofunctors $\mathcal{E}\to\mathscr{C}$
and invertible icons, and $\mathbf{Beck}\left(\mathcal{E},\mathscr{C}\right)$
the subcategory of sinister pseudofunctors satisfying the Beck condition.
\end{defn}

\begin{rem}
Note that $\mathfrak{b}_{f,g}^{f',g'}$ as above may be defined for
any commuting square, not just a pullback. We call such a $\mathfrak{b}_{f,g}^{f',g'}$
the \emph{Beck 2-cell} corresponding to the commuting square, but
should not expect it to be invertible if the square is not a pullback
(even if the Beck condition holds).
\end{rem}

\subsection{Mates under adjunctions}

We now recall the basic properties of mates \cite{2caty}. Given two
pairs of adjoint morphisms
\[
\eta_{1},\epsilon_{1}\colon f_{1}\dashv u_{1}\colon B_{1}\to A_{1},\qquad\eta_{2},\epsilon_{2}\colon f_{2}\dashv u_{2}\colon B_{2}\to A_{2}
\]
in a bicategory $\mathscr{A}$, we say that two 2-cells
\[
\xymatrix{A_{1}\myar{g}{r}\ar[d]_{f_{1}}\ar@{}[rd]|-{\Downarrow\alpha} & A_{2}\ar[d]^{f_{2}} &  &  & A_{1}\myar{g}{r}\ar@{}[rd]|-{\Downarrow\beta} & A_{2}\\
B_{1}\myard{h}{r} & B_{2} &  &  & B_{1}\myard{h}{r}\ar[u]^{u_{1}} & B_{2}\ar[u]_{u_{2}}
}
\]
are mates under the adjunctions $f_{1}\dashv u_{1}$ and $f_{2}\dashv u_{2}$
if $\beta$ is given by the pasting
\[
\xymatrix@=1em{ &  & A_{1}\myar{g}{rr}\ar[dd]|-{f_{1}}\ar@{}[rrdd]|-{\Downarrow\alpha} &  & A_{2}\ar[dd]|-{f_{2}}\ar[rr]^{\textnormal{id}} &  & A_{2}\\
 & \ar@{}[]|-{\;\;\;\;\Downarrow\epsilon_{1}} &  &  &  & \ar@{}[]|-{\Downarrow\eta_{2}\;\;\;\;}\\
B_{1}\ar[rr]_{\textnormal{id}}\ar@/^{0.7pc}/[urur]^{u_{1}} &  & B_{1}\myard{h}{rr} &  & B_{2}\ar@/_{0.7pc}/[urur]_{u_{2}}
}
\]
or equivalently, $\alpha$ is given by the pasting 
\[
\xymatrix@=1em{A_{1}\ar[rr]^{\textnormal{id}}\ar@/_{0.7pc}/[rdrd]_{f_{1}} &  & A_{1}\myar{g}{rr}\ar@{}[rdrd]|-{\Downarrow\beta} &  & A_{2}\ar@/^{0.7pc}/[rdrd]^{f_{2}}\\
 & \ar@{}[]|-{\;\;\;\;\Downarrow\eta_{1}} &  &  &  & \ar@{}[]|-{\Downarrow\epsilon_{2}\;\;\;\;}\\
 &  & B_{1}\myard{h}{rr}\ar[uu]|-{u_{1}} &  & B_{2}\ar[uu]|-{u_{2}}\ar[rr]_{\textnormal{id}} &  & B_{2}
}
\]
It follows from the triangle identities that taking mates in this
fashion defines a bijection between 2-cells $f_{2}g\to hf_{1}$ and
2-cells $gu_{1}\to u_{2}h$.

Moreover, it is well known that this correspondence is functorial.
Given another adjunction $\eta_{3},\epsilon_{3}\colon f_{3}\dashv u_{3}\colon B_{3}\to A_{3}$
and 2-cells as below

\[
\xymatrix{A_{1}\myar{g}{r}\ar[d]_{f_{1}}\ar@{}[rd]|-{\Downarrow\alpha_{l}} & A_{2}\ar[d]|-{f_{2}}\ar[r]^{m}\ar@{}[rd]|-{\Downarrow\alpha_{r}} & A_{3}\ar[d]^{f_{3}} &  & A_{1}\myar{g}{r}\ar@{}[rd]|-{\Downarrow\beta_{l}} & A_{2}\ar[r]^{m}\ar@{}[rd]|-{\Downarrow\beta_{r}} & A_{3}\\
B_{1}\myard{h}{r} & B_{2}\ar[r]_{n} & B_{3} &  & B_{1}\myard{h}{r}\ar[u]^{u_{1}} & B_{2}\ar[u]|-{u_{2}}\ar[r]_{n} & B_{3}\ar[u]_{u_{2}}
}
\]
where $\alpha_{l}$ and $\alpha_{r}$ respectively correspond to $\beta_{l}$
and $\beta_{r}$ under the mates correspondence, it follows that the
pasting of $\alpha_{l}$ and $\alpha_{r}$ corresponds to the pasting
of $\beta_{l}$ and $\beta_{r}$ under the mates correspondence. Moreover,
the analogous property holds for pasting vertically. These vertical
and horizontal pasting properties\footnote{There are also nullary pasting properties which we will omit.}
are often referred to as \emph{functoriality of mates}.
\begin{rem}
Given an adjunction $\eta,\epsilon\colon f\dashv u\colon B\to A$
the left square below
\[
\xymatrix{A\myar{f}{r}\ar[d]_{f}\ar@{}[rd]|-{\Downarrow\textnormal{id}} & A\ar[d]^{\textnormal{id}} &  &  & A\myar{f}{r}\ar@{}[rd]|-{\Downarrow\epsilon} & A\\
B\myard{\textnormal{id}}{r} & B &  &  & B\myard{\textnormal{id}}{r}\ar[u]^{u} & B\ar[u]_{\textnormal{id}}
}
\]
corresponds to the right above via the mates correspondence, allowing
one to see the counit of an adjunction as an instance of the mates
correspondence. A similar calculation may be done for the units. This
will allow us to see calculations involving units and counits as functoriality
of mates calculations.
\end{rem}

One consequence of the mates correspondence which will be of interest
to us is the following lemma; a special case of \cite[Lemma 2.13]{unispans},
showing that the component of an icon between gregarious functors
at a left adjoint 1-cell is invertible.
\begin{lem}
\label{mateinv} Suppose $F,G\colon\mathscr{A}\to\mathscr{B}$ are
gregarious functors between bicategories which agree on objects. Suppose
that $\alpha\colon F\Rightarrow G$ is an icon. Suppose that a given
1-cell $f\colon X\to Y$ has a right adjoint $u$ in $\mathscr{A}$
with unit $\epsilon$ and counit $\eta$. Then the 2-cell $\alpha_{f}\colon Ff\to Gf$
has an inverse given by the mate of $\alpha_{u}\colon Fu\to Gu$.
\end{lem}

\begin{proof}
As $f\dashv u$ we have $Ff\dashv Fu$ via counit
\[
\xymatrix@R=0.5em{Ff\cdot Fu\myar{\varphi_{f,u}^{-1}}{r} & F\left(fu\right)\myar{F\epsilon}{r} & F1_{Y}\myar{\lambda_{Y}}{r} & 1_{FY}}
\]
and unit
\[
\xymatrix@R=0.5em{1_{FX}\myar{\lambda_{X}^{-1}}{r} & F1_{X}\myar{F\eta}{r} & F\left(uf\right)\myar{\varphi_{u,f}}{r} & Fu\cdot Ff}
\]
and similarly $Gf\dashv Gu$. That the mate of $\alpha_{u}$ constructed
as the pasting 
\[
\xymatrix{FX\ar[r]^{1_{FX}}\ar[d]_{1_{FX}}\ar@{}[rd]|-{\Uparrow\lambda_{X}^{-1}} & FX\ar[r]^{1_{FX}}\ar[d]|-{F1_{X}}\ar@{}[rd]|-{\Uparrow F\eta} & FX\ar[r]^{Ff}\ar[d]|-{Fuf}\ar@{}[rd]|-{\Uparrow\varphi_{u,f}} & FY\ar[r]^{1_{FY}}\ar[d]|-{Fu}\ar@{}[rd]|-{\Uparrow\alpha_{u}} & FY\ar[r]^{1_{FY}}\ar[d]|-{Gu}\ar@{}[rd]|-{\Uparrow\psi_{f,u}} & FY\ar[r]^{1_{FY}}\ar[d]|-{Gfu}\ar@{}[rd]|-{\Uparrow G\epsilon} & FY\ar[r]^{1_{FY}}\ar[d]|-{G1_{Y}}\ar@{}[rd]|-{\Uparrow\omega_{Y}} & FY\ar[d]^{1_{FY}}\\
FX\ar[r]_{1_{FX}} & FX\ar[r]_{1_{FX}} & FX\ar[r]_{1_{FX}} & FX\ar[r]_{1_{FX}} & FX\ar[r]_{Gf} & FY\ar[r]_{1_{FY}} & FY\ar[r]_{1_{FY}} & FY
}
\]
is the inverse of $\alpha_{f}$ is a simple calculation which we will
omit (as the details are in \cite[Lemma 2.13]{unispans}).
\end{proof}
\begin{rem}
Under the conditions of the above lemma we have corresponding functors
$F^{\textnormal{co}},G^{\textnormal{co}}\colon\mathscr{A}^{\textnormal{co}}\to\mathscr{B}^{\textnormal{co}}$
which are adjunction preserving (gregarious), and an icon $\alpha^{\textnormal{co}}\colon G^{\textnormal{co}}\Rightarrow F^{\textnormal{co}}$.
Thus noting $u\dashv f$ in $\mathscr{A}^{\textnormal{co}}$ we see
that in $\mathscr{B}^{\textnormal{co}}$, $\alpha_{u}^{\textnormal{co}}\colon Gf\to Ff$
has an inverse given as the mate of $\alpha_{f}^{\textnormal{co}}$.
It follows that $\alpha_{u}$ has an inverse given as the mate of
$\alpha_{f}$ in $\mathscr{B}$.
\end{rem}

\subsection{Adjunctions of spans and polynomials}

Later on we will need to discuss gregarious functors out of bicategories
of spans and bicategories of polynomials, and so an understanding
of the adjunctions in these bicategories will be essential.

We first recall the classification of adjunctions in the bicategory
of spans. A proof of this classification is given in \cite[Proposition 2]{spansrelations},
but this proof does not readily generalize to the setting of polynomials.
We therefore give a simpler proof using the properties of the mates
correspondence.
\begin{prop}
\label{spanadjprop} Up to isomorphism, all adjunctions in $\mathbf{Span}\left(\mathcal{E}\right)$
are of the form
\begin{equation}
\vcenter{\hbox{\xymatrix@=1.5em{ & X\ar[dl]_{1_{X}}\ar[rd]^{f}\\
X &  & Y
}
}}\qquad\dashv\qquad\vcenter{\hbox{\xymatrix@=1.5em{ & X\ar[dl]_{f}\ar[rd]^{1_{X}}\\
Y &  & X
}
}}\label{spanadj}
\end{equation}
with unit and counit
\[
\xymatrix@=1.5em{ & X\ar[dl]_{1_{X}}\ar[rd]^{1_{X}}\ar[dd]^{\delta} &  &  &  &  & X\ar[dl]_{f}\ar[rd]^{f}\ar[dd]^{f}\\
X &  & X &  &  & Y &  & Y\\
 & X\times_{Y}X\ar[ru]_{\pi_{2}}\ar[lu]^{\pi_{1}} &  &  &  &  & Y\ar[ru]_{1_{Y}}\ar[ul]^{1_{Y}}
}
\]
where $\left(X\times_{Y}X,\pi_{1},\pi_{2}\right)$ is the pullback
of $f$ with itself.
\end{prop}

\begin{proof}
It is simple to check the above defines an adjunction. We now check
that all adjunctions have this form, up to isomorphism. To do this,
suppose we are given an adjunction of spans
\[
\vcenter{\hbox{\xymatrix@=1.5em{ & \bullet\ar[dl]_{s}\ar[rd]^{t}\\
\bullet &  & \bullet
}
}}\qquad\dashv\qquad\vcenter{\hbox{\xymatrix@=1.5em{ & \bullet\ar[dl]_{u}\ar[rd]^{v}\\
\bullet &  & \bullet
}
}}
\]
and denote the unit of this adjunction (actually a representation
of the unit using the universal property of pullback) by
\begin{equation}
\xymatrix@=1.5em{ &  & \bullet\ar[rd]|-{\beta}\ar[ld]|-{\alpha}\ar[dd]|-{h}\ar@/^{1.2pc}/[rrdd]|-{\textnormal{id}}\ar@/_{1.2pc}/[lldd]|-{\textnormal{id}}\\
 & \bullet\ar[dl]|-{s}\ar[rd]|-{t} &  & \bullet\ar[dl]|-{u}\ar[rd]|-{v}\\
\bullet &  & \bullet &  & \bullet
}
\label{spanmate}
\end{equation}
noting that $v\beta$ is the identity. We then factor this unit as
\[
\xymatrix{1\myar{}{r} & \left(s,t\right);\left(h,1\right)\myar{\textnormal{id};\beta}{r} & \left(s,t\right);\left(u,v\right)}
\]
where the first morphism is represented by
\[
\xymatrix@=1.5em{ &  & \bullet\ar[rd]|-{\textnormal{id}}\ar[ld]|-{\alpha}\ar[dd]|-{h}\ar@/^{1.2pc}/[rrdd]|-{\textnormal{id}}\ar@/_{1.2pc}/[lldd]|-{\textnormal{id}}\\
 & \bullet\ar[dl]|-{s}\ar[rd]|-{t} &  & \bullet\ar[dl]|-{h}\ar[rd]|-{\textnormal{id}}\\
\bullet &  & \bullet &  & \bullet
}
\]
and $\beta\colon\left(h,1\right)\to\left(u,v\right)$ is pictured
on the right in \eqref{spanmate}. Under the mates correspondence
this yields two morphisms
\[
\xymatrix{\left(u,v\right)\myar{}{r} & \left(h,1\right)\myar{\beta}{r} & \left(u,v\right)}
\]
which must compose to the identity. As the first morphism of spans
is necessarily $v$ we have also established $\beta v$ as the identity,
and hence $v$ as an isomorphism. This allows us to construct an isomorphism
of right adjoints $\left(u,v\right)\to\left(f,1\right)$ for an $f$
as in \eqref{spanadj}, corresponding to an isomorphism of left adjoints
$\left(1,f\right)\to\left(s,t\right)$ and hence showing $s$ is invertible
also.
\end{proof}
\begin{rem}
If we restrict ourselves to the bicategory $\mathbf{Span}_{\textnormal{iso}}\left(\mathcal{E}\right)$
then we only have adjunctions as above when $f$ is invertible (necessary
to construct the counit).
\end{rem}

In the case of polynomials there are more adjunctions to consider.
\begin{prop}
\label{polyadjprop} Up to isomorphism, every adjunction in $\mathbf{Poly}\left(\mathcal{E}\right)$
is a composite of adjunctions of the form
\[
\vcenter{\hbox{\xymatrix@=1.5em{ & X\ar[dl]_{1_{X}}\ar[r]^{1_{X}} & X\ar[rd]^{f}\\
X &  &  & Y
}
}}\qquad\dashv\qquad\vcenter{\hbox{\xymatrix@=1.5em{ & X\ar[dl]_{f}\ar[r]^{1_{X}} & X\ar[rd]^{1_{X}}\\
Y &  &  & X
}
}}
\]
with unit and counit
\[
\xymatrix@=1.5em{ & X\ar[dl]_{1_{X}}\ar[r]^{1_{X}} & X\ar[rd]^{1_{X}}\ar@{=}[d] &  &  &  &  & X\ar[dl]_{f}\ar[r]^{1_{X}} & X\ar[rd]^{f}\ar@{=}[d]\\
X & X\ar[d]_{\delta}\ar[u]^{1_{X}}\ar[r]^{1_{X}} & X\ar[d]^{\delta} & X &  &  & Y & X\ar[d]_{f}\ar[r]^{1_{X}}\ar[u]^{1_{X}} & X\ar[d]^{f} & Y\\
 & X\times_{Y}X\ar[lu]^{\pi_{1}}\ar[r]_{1_{X\times_{Y}X}} & X\times_{Y}X\ar[ru]_{\pi_{2}} &  &  &  &  & Y\ar[ul]^{1_{Y}}\ar[r]_{1_{Y}} & Y\ar[ru]_{1_{Y}}
}
\]
and 
\[
\vcenter{\hbox{\xymatrix@=1.5em{ & X\ar[dl]_{f}\ar[r]^{1_{X}} & X\ar[rd]^{1_{X}}\\
Y &  &  & X
}
}}\qquad\dashv\qquad\vcenter{\hbox{\xymatrix@=1.5em{ & X\ar[dl]_{1_{X}}\ar[r]^{f} & Y\ar[rd]^{1_{Y}}\\
X &  &  & Y
}
}}
\]
with unit and counit
\[
\xymatrix@=1.5em{ & Y\ar[dl]_{1_{Y}}\ar[r]^{1_{Y}} & Y\ar[rd]^{1_{Y}}\ar@{=}[d] &  &  &  &  & X\times_{Y}X\ar[ld]_{\pi_{1}}\myar{\pi_{2}}{r} & X\ar[rd]^{1_{X}}\ar@{=}[d]\\
Y & X\ar[r]^{f}\ar[d]_{1_{X}}\ar[u]^{f} & Y\ar[d]^{1_{Y}} & Y &  &  & X & X\ar[u]^{\delta}\ar[d]_{1_{X}}\ar[r]^{1_{X}} & X\ar[d]^{1_{X}} & X\\
 & X\ar[r]_{f}\ar[lu]^{f} & Y\ar[ru]_{1_{Y}} &  &  &  &  & X\ar[r]_{1_{X}}\ar[lu]^{1_{X}} & X\ar[ru]_{1_{X}}
}
\]
\end{prop}

\begin{proof}
It is simple to check that the above define adjunctions of polynomials,
indeed this is almost the same calculation as in the case of spans.
We now check that all adjunctions have this form, up to isomorphism.
To do this, suppose we are given an adjunction of polynomials
\[
\vcenter{\hbox{\xymatrix@=1.5em{ & \bullet\ar[dl]_{s}\ar[r]^{p} & \bullet\ar[rd]^{t}\\
\bullet &  &  & \bullet
}
}}\qquad\dashv\qquad\vcenter{\hbox{\xymatrix@=1.5em{ & \bullet\ar[dl]_{u}\ar[r]^{q} & \bullet\ar[rd]^{v}\\
\bullet &  &  & \bullet
}
}}
\]
and denote the unit of this adjunction by\footnote{Here the cartesian part of the morphism of polynomials is represented
using Proposition \ref{unipolycomp}.}
\[
\xymatrix@=1.5em{ &  & \bullet\ar@/_{1.2pc}/[llddd]|-{\textnormal{id}}\ar@/^{0.6pc}/[rrd]|-{\textnormal{id}}\\
 &  & \bullet\ar[ld]|-{\alpha_{1}}\ar[r]|-{p'}\ar[u]|-{e} & \bullet\ar[rd]|-{\beta_{1}}\ar[ld]|-{\alpha_{2}}\ar[r]|-{q'}\ar[dd]|-{h} & \bullet\ar[rd]|-{\beta_{2}}\ar@/^{1.2pc}/[rrdd]|-{\textnormal{id}}\\
 & \bullet\ar[dl]|-{s}\ar[r]|-{p} & \bullet\ar[rd]|-{t} &  & \bullet\ar[r]|-{q}\ar[dl]|-{u} & \bullet\ar[rd]|-{v}\\
\bullet &  &  & \bullet &  &  & \bullet
}
\]
noting that $v\beta_{2}$ is the identity. We then factor this unit
as, where $\left(\beta_{1},\beta_{2}\right)\colon\left(h,q',1\right)\to\left(u,q,v\right)$
is the cartesian morphism of polynomials pictured on the right above,
\[
\xymatrix{1\myar{}{rr} &  & \left(s,p,t\right);\left(h,q',1\right)\myar{\textnormal{id};\left(\beta_{1},\beta_{2}\right)}{rr} &  & \left(s,p,t\right);\left(u,q,v\right)}
\]
which under the mates correspondence yields two morphisms
\[
\xymatrix{\left(u,q,v\right)\myar{}{r} & \left(h,q',1\right)\myar{\left(\beta_{1},\beta_{2}\right)}{r} & \left(u,q,v\right)}
\]
which must compose to the identity; that is, a diagram below
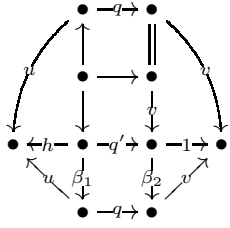
\[
\xymatrix@=1.5em{ & \bullet\ar@/_{0.6pc}/[ddl]|-{u}\ar[r]|-{q} & \bullet\ar@/^{0.6pc}/[rdd]|-{v}\ar@{=}[d]\\
 & \bullet\ar[r]\ar[u]\ar[d] & \bullet\ar[d]|-{v}\\
\bullet & \bullet\ar[r]|-{q'}\ar[l]|-{h}\ar[d]|-{\beta_{1}} & \bullet\ar[r]|-{1}\ar[d]|-{\beta_{2}} & \bullet\\
 & \bullet\ar[r]|-{q}\ar[ul]|-{u} & \bullet\ar[ru]|-{v}
}
\]
composing to the identity, showing $\beta_{2}v$ is the identity,
and hence that $v$ is invertible. This allows us to construct an
isomorphism of right adjoints $\left(u,q,v\right)\to\left(f,g,1\right)$
for some $f$ and $g$, corresponding to an isomorphism of left adjoints
$\left(g,1,f\right)\to\left(s,p,t\right)$ and hence showing $p$
is invertible also.
\end{proof}
\begin{rem}
If we restrict ourselves to the bicategory $\mathbf{Poly}_{c}\left(\mathcal{E}\right)$
then to have the second adjunction of Proposition \ref{polyadjprop}
we require $f$ to be invertible.
\end{rem}

\subsection{Basic properties of generic bicategories\label{generic}}

The bicategories of spans $\mathbf{Span}\left(\mathcal{E}\right)$
and bicategories of polynomials with cartesian 2-cells $\mathbf{Poly}_{c}\left(\mathcal{E}\right)$
defined above both satisfy a special property: they are examples of
a bicategory $\mathscr{A}$ which contains a special class of 2-cells
(which one may informally think of as the ``diagonal'' 2-cells\footnote{Formally, these diagonals are defined as the generic morphisms against
the composition functor. See \cite{WalkerGeneric} for details.}) such that any 2-cell into a composite of 1-cells $\alpha\colon c\to a;b$
factors uniquely as some diagonal 2-cell $\delta\colon c\to l;r$
pasted with 2-cells $\alpha_{1}\colon l\to a$ and $\alpha_{2}\colon r\to b$. 

For the reader familiar with generic morphisms \cite{diers,DiersDiag,WeberGeneric},
this property can be stated concisely by asking that each composition
functor 
\[
\circ_{X,Y,Z}\colon\mathscr{A}_{Y,Z}\times\mathscr{A}_{X,Y}\to\mathscr{A}_{X,Z}
\]
admits generic factorisations. A bicategory $\mathscr{A}$ with this
property is called \emph{generic.}

As shown in \cite{WalkerGeneric}, one of the main properties of generic
bicategories $\mathscr{A}$ is that oplax functors out of them admit
an alternative description, similar to the description of a comonad.
In particular, for a locally defined functor $L\colon\mathscr{A}\to\mathscr{C}$
one may define a bijection between coherent binary and nullary oplax
constraint cells
\[
\varphi_{a,b}\colon L\left(a;b\right)\to La;Lb,\qquad\lambda_{X}\colon L1_{X}\to1_{X}
\]
and ``coherent'' comultiplication and counit maps 
\[
\Phi_{\delta}\colon Lc\to Ll;Lr,\qquad\Lambda_{\epsilon}\colon Ln\to1_{X}
\]
indexed over diagonal maps $\delta\colon c\to l;r$ and augmentations
(2-cells into identity 1-cells) $\epsilon\colon n\to1_{X}$. Indeed,
given the data $\left(\varphi,\lambda\right)$ the comultiplication
maps $\Phi_{\delta}$ and counit maps $\Lambda_{\epsilon}$ are given
by the composites
\[
\xymatrix{Lc\myar{L\delta}{r} & L\left(l;r\right)\myar{\varphi_{l,r}}{r} & Ll;Lr &  & Ln\myar{L\epsilon}{r} & L1_{X}\myar{\lambda_{X}}{r} & 1_{X}}
\]
and conversely given the data $\left(\Phi,\Lambda\right)$ the oplax
constraints $\varphi_{a,b}$ and $\lambda_{X}$ are recovered by factoring
the identity 2-cell through a diagonal as on the left below and defining
the right diagram to commute.
\[
\xymatrix@=1em{ & l;r\ar[rd]^{s_{1};s_{2}} &  &  &  &  & Ll;Lr\ar[rd]^{Ls_{1};Ls_{2}}\\
a;b\myard{\textnormal{id}}{rr}\ar[ur]^{\delta} &  & a;b &  &  & L\left(a;b\right)\myard{\varphi_{a,b}}{rr}\ar[ur]^{\Phi_{\delta}} &  & La;Lb
}
\]
Trivially, we recover each unit $\lambda_{X}\colon L\left(1_{X}\right)\to1_{X}$
as the component of $\Lambda$ at $\textnormal{id}_{1_{X}}$.

For the full statement concerning the bicategories of spans and cartesian
polynomials, see Proposition \ref{spancoh} and Proposition \ref{polycoh}
respectively.

\section{Universal properties of spans\label{unispans}}

In this section, we give a complete proof of the universal properties
of spans \cite{unispans} using the properties of generic bicategories
\cite{WalkerGeneric}. This is to demonstrate our method in the simpler
case of spans before applying it to polynomials in Section \ref{unipolycart}.

\subsection{Stating the universal property}

Before stating the universal property we recall that we have two canonical
embeddings into the bicategory of spans given by the pseudofunctors
denoted
\[
\left(-\right)_{\Sigma}:\mathcal{E}\to\mathbf{Span}\left(\mathcal{E}\right),\qquad\left(-\right)_{\Delta}:\mathcal{E}^{\textnormal{op}}\to\mathbf{Span}\left(\mathcal{E}\right).
\]
 These are defined on objects by sending an object of $\mathcal{E}$
to itself, and are defined on each morphism in $\mathcal{E}$ by the
assignments
\[
\xymatrix@R=1em{\left(-\right)_{\Sigma}\colon & X\ar[r]^{f} & Y & \mapsto & X & X\ar[l]_{1_{X}}\ar[r]^{f} & Y\\
\left(-\right)_{\Delta}\colon & X\ar[r]^{f} & Y & \mapsto & Y & X\ar[l]_{f}\ar[r]^{1_{X}} & X
}
\]
\begin{rem}
Note that the embedding $\left(-\right)_{\Sigma}$ is an example of
a pseudofunctor which is both sinister and satisfies the Beck condition.
\end{rem}

The universal property of spans is then the following result, as given
by Hermida \cite{hermida} and Dawson, Paré, and Pronk \cite[Theorem 2.15]{unispans}.
\begin{thm}
[Universal Properties of Spans] \label{unispansthm} Given a category
$\mathcal{E}$ with chosen pullbacks, composition with the canonical
embedding $\left(-\right)_{\Sigma}:\mathcal{E}\to\mathbf{Span}\left(\mathcal{E}\right)$
defines the equivalence of categories 
\[
\mathbf{Greg}\left(\mathbf{Span}\left(\mathcal{E}\right),\mathscr{C}\right)\simeq\mathbf{Sin}\left(\mathcal{E},\mathscr{C}\right)
\]
which restricts to the equivalence
\[
\mathbf{Icon}\left(\mathbf{Span}\left(\mathcal{E}\right),\mathscr{C}\right)\simeq\mathbf{Beck}\left(\mathcal{E},\mathscr{C}\right)
\]
for any bicategory $\mathscr{C}$.
\end{thm}

\subsection{Proving the universal property}

Before proving Theorem \ref{unispansthm} we will need to show that
given a sinister pseudofunctor $\mathcal{E}\to\mathscr{C}$ one may
reconstruct an oplax functor $\mathbf{Span}\left(\mathcal{E}\right)\to\mathscr{C}$.
The following lemma and subsequent propositions describe this construction.
\begin{lem}
\label{spanlocally} Let $\mathcal{E}$ be a category with pullbacks
seen as a locally discrete 2-category, and let $\mathscr{C}$ be a
bicategory. Suppose $F\colon\mathcal{E}\to\mathscr{C}$ is a given
sinister pseudofunctor, and for each morphism $f\in\mathcal{E}$ define
$F_{\Sigma}f:=Ff$ and take $F_{\Delta}f$ to be a chosen right adjoint
of $Ff$ (choosing $F_{\Delta}$ to strictly preserve identities).
We may then define local functors
\[
L_{X,Y}:\mathbf{Span}\left(\mathcal{E}\right)_{X,Y}\to\mathscr{C}_{LX,LY},\qquad X,Y\in\mathcal{E}
\]
by the assignment $T\mapsto FT$ on objects, and
\[
\xymatrix@=1.5em{ & T\ar[dl]_{s}\ar[dd]|-{f}\ar[dr]^{t} &  &  &  & FT\ar[dr]^{F_{\Sigma}t}\\
X &  & Y & \mapsto & FX\ar[ur]^{F_{\Delta}s}\ar[dr]_{F_{\Delta}u} & \;\ar@{}[l]|-{\Downarrow\alpha}\ar@{}[r]|-{\Downarrow\gamma} & FY\\
 & S\ar[ul]^{u}\ar[ur]_{v} &  &  &  & FS\ar[uu]|-{F_{\Delta}f}\ar[ur]_{F_{\Sigma}v}
}
\]
on morphisms, where $\alpha$ is the mate of the isomorphism on the
left below
\[
\xymatrix{FT\ar[r]^{1_{FT}}\ar[d]_{F_{\Sigma}s}\ar@{}[dr]|-{\cong} & FT\ar[d]^{F_{\Sigma}u\cdot F_{\Sigma}f} &  &  & FT\ar[r]^{F_{\Sigma}t}\ar[d]_{F_{\Sigma}f}\ar@{}[dr]|-{\cong} & FY\ar[d]^{1_{FY}}\\
FX\ar[r]_{1_{FX}} & FX &  &  & FS\ar[r]_{F_{\Sigma}v} & FY
}
\]
under the adjunctions $F_{\Sigma}s\dashv F_{\Delta}s$ and $F_{\Sigma}u\cdot F_{\Sigma}f\dashv F_{\Delta}f\cdot F_{\Delta}u$,
and $\gamma$ is the mate of the isomorphism on the right above under
the adjunctions $F_{\Sigma}f\dashv F_{\Delta}f$ and $\textnormal{1}_{FY}\dashv\textnormal{1}_{FY}$. 
\end{lem}

\begin{proof}
Functoriality is clear from functoriality of mates and the associativity
condition and unitary conditions on $F$.
\end{proof}
To show that these local functors can be endowed with the structure
of an oplax functor it will be useful to recall the following reduced
description of such an oplax structure, obtained via the theory of
Subsection \ref{generic}.
\begin{prop}
\cite{WalkerGeneric}\label{spancoh} Let $\mathcal{E}$ be a category
with pullbacks and denote by $\mathbf{Span}\left(\mathcal{E}\right)$
the bicategory of spans in $\mathcal{E}$. Let $\mathscr{C}$ be a
bicategory. Then to give an oplax functor
\[
L\colon\mathbf{Span}\left(\mathcal{E}\right)\to\mathscr{C}
\]
 is to give a locally defined functor
\[
L_{X,Y}\colon\mathbf{Span}\left(\mathcal{E}\right)_{X,Y}\to\mathscr{C}_{LX,LY},\qquad X,Y\in\mathcal{E}
\]
with comultiplication and counit maps 
\[
\Phi_{s,h,t}\colon L\left(s,t\right)\to L\left(s,h\right);L\left(h,t\right),\qquad\Lambda_{h}\colon L\left(h,h\right)\to1_{LX}
\]
for every respective diagram in $\mathcal{E}$ 
\[
\xymatrix@=1.5em{ & T\ar[rd]^{t}\ar[ld]_{s}\ar[d]|-{h} &  &  &  & T\ar[d]^{h}\\
X & Y & Z &  &  & X
}
\]
such that:
\begin{enumerate}
\item for any triple of morphisms of spans as below
\[
\xymatrix@=1.5em{ & R\ar[dl]_{u}\ar[dr]^{v}\ar[dd]|-{f} &  &  & R\ar[dl]_{u}\ar[dr]^{k}\ar[dd]|-{f} &  &  & R\ar[dl]_{k}\ar[dr]^{v}\ar[dd]|-{f}\\
X &  & Z & X &  & Y & Y &  & Z\\
 & T\ar[ul]^{s}\ar[ru]_{t} &  &  & T\ar[ul]^{s}\ar[ru]_{h} &  &  & T\ar[ul]^{h}\ar[ru]_{t}
}
\]
we have the commuting diagram
\[
\xymatrix{L\left(u,v\right)\ar[d]_{Lf}\myar{\Phi_{u,k,v}}{rr} &  & L\left(u,k\right);L\left(k,v\right)\ar[d]^{Lf;Lf}\\
L\left(s,t\right)\myard{\Phi_{s,h,t}}{rr} &  & L\left(s,h\right);L\left(h,t\right)
}
\]
\item for any morphism of spans as on the left below
\[
\xymatrix@=1.5em{ & M\ar[dl]_{p}\ar[dr]^{p}\ar[d]|-{f} &  &  & L\left(p,p\right)\ar[rr]^{Lf}\ar[rd]_{\Lambda_{p}} &  & L\left(q,q\right)\ar[ld]^{\Lambda_{q}}\\
X & N\ar[l]^{q}\ar[r]_{q} & X &  &  & 1_{LX}
}
\]
the diagram on the right above commutes;
\item for all diagrams of the form
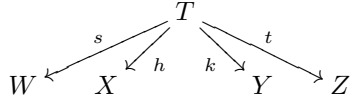
\[
\xymatrix@=1.5em{ &  & T\ar[dll]_{s}\ar[drr]^{t}\ar[dl]^{h}\ar[dr]_{k}\\
W & X &  & Y & Z
}
\]
in $\mathcal{E}$, we have the commuting diagram
\[
\xymatrix{L\left(s,t\right)\ar[d]_{\Phi_{s,h,t}}\ar@{=}[rr] &  & L\left(s,t\right)\ar[d]^{\Phi_{s,k,t}}\\
L\left(s,h\right);L\left(h,t\right)\myard{L\left(s;h\right);\Phi_{h,k,t}}{d} &  & L\left(s,k\right);L\left(k,t\right)\ar[d]^{\Phi_{s,h,k};L\left(k;t\right)}\\
L\left(s,h\right);\left(L\left(h,k\right);L\left(k,t\right)\right)\myard{\textnormal{assoc}}{rr} &  & \left(L\left(s,h\right);L\left(h,k\right)\right);L\left(k,t\right)
}
\]
\item for all spans $\left(s,t\right)$ we have the commuting diagrams
\[
\xymatrix@=0.5em{ & L\left(s,s\right);L\left(s,t\right)\ar[rdd]^{\Lambda_{s};L\left(s,t\right)} &  &  &  & L\left(s,t\right);L\left(t,t\right)\ar[rdd]^{L\left(s,t\right);\Lambda_{t}}\\
\\
L\left(s,t\right)\myard{\textnormal{unitor}}{rr}\ar[uur]^{\Phi_{s,s,t}} &  & 1_{LX};L\left(s,t\right) &  & L\left(s,t\right)\myard{\textnormal{unitor}}{rr}\ar[uur]^{\Phi_{s,t,t}} &  & L\left(s,t\right);1_{LY}
}
\]
\end{enumerate}
\end{prop}

We now prove that the locally defined functor $L$ above may be endowed
with an oplax structure.
\begin{prop}
\label{spanLoplax} Let $\mathcal{E}$ be a category with pullbacks
seen as a locally discrete 2-category, and let $\mathscr{C}$ be a
bicategory. Suppose $F\colon\mathcal{E}\to\mathscr{C}$ is a given
sinister pseudofunctor. Then the locally defined functor
\[
L_{X,Y}:\mathbf{Span}\left(\mathcal{E}\right)_{X,Y}\to\mathscr{C}_{LX,LY},\qquad X,Y\in\mathcal{E}
\]
as in Lemma \ref{spanlocally} canonically admits the structure of
an oplax functor.
\end{prop}

\begin{proof}
By Proposition \ref{spancoh}, to equip the locally defined functor
$L$ with an oplax structure is to give comultiplication maps $\Phi_{s,h,t}\colon L\left(s,t\right)\to L\left(s,h\right);L\left(h,t\right)$
and counit maps $\Lambda_{h}\colon L\left(h,h\right)\to1_{LX}$ for
diagrams of the respective forms
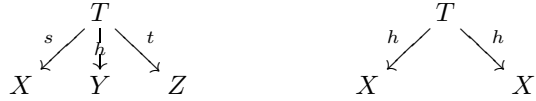
\[
\xymatrix@=1.5em{ & T\ar[dl]_{s}\ar[dr]^{t}\ar[d]|-{h} &  &  &  &  & T\ar[rd]^{h}\ar[ld]_{h}\\
X & Y & Z &  &  & X &  & X
}
\]
satisfying naturality, associativity, and unitary conditions. To do
this, we take each $\Phi_{s,h,t}$ and $\Lambda_{h}$ to be the respective
pastings
\begin{equation}
\xymatrix{ &  & \mbox{\;} &  &  &  & FT\ar[rd]^{F_{\Sigma}h} & \;\\
FX\ar[r]^{F_{\Delta}s} & FT\ar[r]^{F_{\Sigma}h}\ar@/^{2.9pc}/[rr]^{1_{FT}} & FY\ar[r]^{F_{\Delta}h}\ar@{}[u]|-{\Downarrow\eta_{Fh}} & FT\ar[r]^{F_{\Sigma}t} & FZ & FX\ar[rr]^{1_{FX}}\ar[ur]^{F_{\Delta}h} & \ar@{}[u]|-{\Downarrow\epsilon_{Fh}} & FX
}
\label{spanlax}
\end{equation}
Associativity of comultiplication is trivial; indeed, given a diagram
of the form
\[
\xymatrix@=1.5em{ &  & T\ar[dll]_{s}\ar[drr]^{t}\ar[dl]^{h}\ar[dr]_{k}\\
W & X &  & Y & Z
}
\]
the pasting
\[
\xymatrix{ &  & \mbox{\;} &  & \;\\
FW\ar[r]^{F_{\Delta}s} & FT\ar[r]^{F_{\Sigma}h}\ar@/^{2.5pc}/[rr]^{\textnormal{id}_{FT}} & FX\ar[r]^{F_{\Delta}h}\ar@{}[u]|-{\Downarrow\eta_{Fh}} & FT\ar[r]^{F_{\Sigma}k}\ar@/^{2.5pc}/[rr]^{\textnormal{id}_{FT}} & FY\ar[r]^{F_{\Delta}k}\ar@{}[u]|-{\Downarrow\eta_{Fk}} & FT\ar[r]^{F_{\Sigma}t} & FZ
}
\]
evaluates to the same 2-cell regardless if we paste the bottom path
with first $\eta_{Fh}$ and then $\eta_{Fk}$, or vice versa. The
unitary axioms are also trivial, an immediate consequence of the triangle
identities for an adjunction.

For the naturality condition, suppose we are given a triple of morphisms
of spans
\[
\xymatrix@=1.5em{ & R\ar[dl]_{u}\ar[dr]^{v}\ar[dd]|-{f} &  &  &  & R\ar[dl]_{u}\ar[dr]^{k}\ar[dd]|-{f} &  &  &  & R\ar[dl]_{k}\ar[dr]^{v}\ar[dd]|-{f}\\
X &  & Z &  & X &  & Y &  & Y &  & Z\\
 & T\ar[ul]^{s}\ar[ru]_{t} &  &  &  & T\ar[ul]^{s}\ar[ru]_{h} &  &  &  & T\ar[ul]^{h}\ar[ru]_{t}
}
\]
and note that we have the commuting diagram
\[
\xymatrix{L\left(u,v\right)\ar[d]_{Lf}\myar{\Phi_{u,k,v}}{rr} &  & L\left(u,k\right);L\left(k,v\right)\ar[d]^{Lf;Lf}\\
L\left(s,t\right)\myard{\Phi_{s,h,t}}{rr} &  & L\left(s,h\right);L\left(h,t\right)
}
\]
since the top composite is
\[
\xymatrix{ & FR\ar[rr]|-{1_{FR}}\ar[rd]|-{F_{\Sigma}k} & \;\ar@{}[d]|-{\Downarrow\eta_{Fk}} & FR\ar[rd]|-{F_{\Sigma}v}\\
FX\ar[ur]|-{F_{\Delta}u}\ar[rd]|-{F_{\Delta}s}\ar@{}[r]|-{\Downarrow} & \ar@{}[r]|-{\Downarrow} & FY\ar[rd]|-{F_{\Delta}h}\ar@{}[r]|-{\Downarrow}\ar[ur]|-{F_{\Delta}k} & \;\ar@{}[r]|-{\Downarrow} & FZ\\
 & FT\ar[ur]|-{F_{\Sigma}h}\ar[uu]|-{F_{\Delta}f} &  & FT\ar[ur]|-{F_{\Sigma}t}\ar[uu]|-{F_{\Delta}f}
}
\]
and the bottom composite is 

\[
\xymatrix{ & FR\ar[rr]|-{1_{FR}} &  & FR\ar[rd]|-{F_{\Sigma}v}\\
FX\ar[ur]|-{F_{\Delta}u}\ar[rd]|-{F_{\Delta}s}\ar@{}[r]|-{\Downarrow} & \ar@{}[rr]|-{=} &  & \;\ar@{}[r]|-{\Downarrow} & FZ\\
 & FT\ar[rr]|-{1_{FT}}\ar[uu]|-{F_{\Delta}f}\ar[dr]|-{F_{\Sigma}h} & \ar@{}[d]|-{\Downarrow\eta_{Fh}} & FT\ar[ur]|-{F_{\Sigma}t}\ar[uu]|-{F_{\Delta}f}\\
 &  & FY\ar[ur]|-{F_{\Delta}h}
}
\]
where the unlabeled 2-cells are as in Lemma \ref{spanlocally}. That
these pastings agree is a standard functoriality of mates calculation.
We omit the naturality of counits calculation, as it is a simpler
functoriality of mates calculation.
\end{proof}
\begin{rem}
It is trivial that each $\lambda_{X}$ given by $\Lambda$ at $1_{X}$
is invertible above.
\end{rem}

We now check that the structure given above has its oplax constraints
given by Beck 2-cells.
\begin{lem}
\label{oplaxBeck} Let the oplax functor $L\colon\mathbf{Span}\left(\mathcal{E}\right)\to\mathscr{C}$
be constructed as in Proposition \ref{spanLoplax}. Then the binary
oplax constraint cell on $L$, at a composite of spans constructed
as below
\begin{equation}
\xymatrix@=1.5em{ &  & M\ar[ld]_{c'}\ar[rd]^{b'}\ar@{}[dd]|-{\textnormal{pb}}\\
 & T\ar[ld]_{a}\ar[rd]^{b} &  & S\ar[ld]_{c}\ar[rd]^{d}\\
X &  & Y &  & Z
}
\label{spancomp}
\end{equation}
is given by the Beck 2-cell for the pullback appropriately whiskered
by $F_{\Delta}a$ and $F_{\Sigma}d$.
\end{lem}

\begin{proof}
Given composable spans $\left(a,b\right)$ and $\left(c,d\right)$
the composite is given by the diagram \eqref{spancomp}. We then have
an induced diagonal
\[
\delta_{ac',h,db'}\colon\left(ac',db'\right)\to\left(ac',h\right);\left(h,db'\right)
\]
and morphisms $c'\colon\left(ac',h\right)\to\left(a,b\right)$ and
$b'\colon\left(h,db'\right)\to\left(c,d\right)$ for which
\[
\xymatrix{\left(a,b\right);\left(c,d\right)\myar{\delta_{ac',h,db'}}{r} & \left(ac',h\right);\left(h,db'\right)\myar{c';b'}{r} & \left(a,b\right);\left(c,d\right)}
\]
is the identity on $\left(a,b\right);\left(c,d\right)$. Hence the
oplax constraint cell corresponding to the comultiplication maps $\Phi$,
namely
\[
\varphi_{\left(a,b\right),\left(c,d\right)}\colon L\left(\left(a,b\right);\left(c,d\right)\right)\to L\left(a,b\right);L\left(c,d\right)
\]
is given by the pasting 
\begin{equation}
\xymatrix{ & FM\ar[rr]|-{1_{FM}}\ar[rd]|-{F_{\Sigma}h} & \;\ar@{}[d]|-{\Downarrow\eta_{Fh}} & FM\ar[rd]|-{F_{\Sigma}db'}\\
FX\ar[ur]|-{F_{\Delta}ac'}\ar[rd]|-{F_{\Delta}a}\ar@{}[r]|-{\Downarrow} & \ar@{}[r]|-{\Downarrow} & FY\ar[rd]|-{F_{\Delta}c}\ar@{}[r]|-{\Downarrow}\ar[ur]|-{F_{\Delta}h} & \;\ar@{}[r]|-{\Downarrow} & FZ\\
 & FT\ar[ur]|-{F_{\Sigma}b}\ar[uu]|-{F_{\Delta}c'} &  & FS\ar[ur]|-{F_{\Sigma}d}\ar[uu]|-{F_{\Delta}b'}
}
\label{spancompconstraint}
\end{equation}
where $h=bc'=cb'$. It is an easy consequence of functoriality of
mates that this pasting is the usual Beck 2-cell for the pullback
with the appropriate whiskerings.
\end{proof}
Finally, we will need the following lemma, a consequence of Lemma
\ref{mateinv}, in order to complete the proof.
\begin{lem}
\label{spaniconinv} Suppose $L,K\colon\mathbf{Span}\left(\mathcal{E}\right)\to\mathscr{C}$
are given gregarious functors. Then any icon $\alpha\colon L\Rightarrow K$
is necessarily invertible.
\end{lem}

\begin{proof}
We take identities to be pullback stable for simplicity, so that we
have $\left(s,t\right)=\left(s,1\right);\left(1,t\right)$. Let us
consider the component of such an icon $\alpha$ at a general span
$s,t$. Since $\alpha$ is an icon, the diagram
\begin{equation}
\xymatrix{L\left(s,t\right)\myar{\varphi}{r}\ar[d]_{\alpha_{\left(s,t\right)}} & L\left(s,1\right);L\left(1,t\right)\ar[d]^{\alpha_{\left(s,1\right)};\alpha_{\left(1,t\right)}}\\
K\left(s,t\right)\myard{\psi}{r} & K\left(s,1\right);K\left(1,t\right)
}
\label{spaniconvfigure}
\end{equation}
commutes. By Lemma \ref{mateinv} we know $\alpha_{\left(1,t\right)}$
is invertible, and by its dual we know $\alpha_{\left(s,1\right)}$
is invertible. As $F$ and $G$ are gregarious $\varphi$ and $\psi$
are invertible above. Hence $\alpha_{\left(s,t\right)}$ is invertible.
\end{proof}
We now know enough for a complete proof of the universal properties
of the span construction as given by Dawson, Par{\'e}, Pronk and
Hermida.
\begin{proof}
[Proof of Theorem \ref{unispansthm}] We consider the assignment of
Theorem \ref{unispansthm}, i.e. composition with the embedding $\left(-\right)_{\Sigma}\colon\mathcal{E}\to\mathbf{Span}\left(\mathcal{E}\right)$
written as the assignment
\[
\xymatrix{\mathbf{Span}\left(\mathcal{E}\right)\ar@/_{1pc}/[rr]_{G}\ar@/^{1pc}/[rr]^{F} & \Downarrow\alpha & \mathscr{C} & \mapsto & \mathcal{E}\ar@/_{1pc}/[rr]_{G_{\Sigma}}\ar@/^{1pc}/[rr]^{F_{\Sigma}} & \Downarrow\alpha_{\Sigma} & \mathscr{C}}
\]

We start by proving the first universal property.

\emph{\noun{Well defined.}} This is clear by Corollary \ref{spaniconinv}.

\noun{Fully faithful.} That the assignment $\alpha\mapsto\alpha_{\Sigma}$
is bijective follows from the condition $\left(\alpha_{\Sigma_{f}}^{-1}\right)^{*}=\alpha_{\Delta_{f}}$
forced by Lemma \ref{mateinv}, and the commutativity of \eqref{spaniconvfigure}.
One need only check that any collection
\[
\alpha_{s,t}\colon F\left(s,t\right)\to G\left(s,t\right)
\]
 satisfying these two properties necessarily defines an icon. Indeed,
that such an $\alpha$ is locally natural is a simple consequence
of functoriality of mates and $\alpha_{\Sigma}$ being an icon. To
see that such an $\alpha$ then defines an icon, note that each $\Phi_{s,h,t}$
may be decomposed as the commuting diagram
\[
\xymatrix{F\left(s,t\right)\myar{\Phi_{s,h,t}}{rrr}\ar[d]_{\Phi_{s,1,t}} &  &  & F\left(s,h\right);F\left(h,t\right)\\
F\left(s,1\right);F\left(1,t\right)\myard{F\left(s,1\right);\Phi_{1,h,1};F\left(1,t\right)}{rrr} &  &  & F\left(s,1\right);F\left(1,h\right);F\left(h,1\right);F\left(1,t\right)\ar[u]_{\Phi_{s,1,h}^{-1};\Phi_{h,1,t}^{-1}}
}
\]
and so the commutativity of the diagram\footnote{This diagram is equivalent to the binary coherence condition on such
an icon.}
\[
\xymatrix{F\left(s,t\right)\myar{\Phi_{s,h,t}}{rr}\ar[d]_{\alpha_{s,t}} &  & F\left(s,h\right);F\left(h,t\right)\ar[d]^{\alpha_{s,h};\alpha_{h,t}}\\
G\left(s,t\right)\myard{\Psi_{s,h,t}}{rr} &  & G\left(s,h\right);G\left(h,t\right)
}
\]
amounts to asking that the pastings
\[
\xymatrix{ &  & \ar@{}[d]|-{\Downarrow\eta_{Gh}}\\
\bullet\ar@/_{0.5pc}/[r]_{G_{\Delta}s}\ar@/^{0.5pc}/[r]^{F_{\Delta}s}\ar@{}[r]|-{\Downarrow} & \bullet\ar@/_{0.5pc}/[r]_{G_{\Sigma}h}\ar@/^{0.5pc}/[r]^{F_{\Sigma}h}\ar@{}[r]|-{\Downarrow}\ar@/^{2.5pc}/[rr]^{\textnormal{id}} & \bullet\ar@/_{0.5pc}/[r]_{G_{\Delta}h}\ar@/^{0.5pc}/[r]^{F_{\Delta}h}\ar@{}[r]|-{\Downarrow} & \bullet\ar@/_{0.5pc}/[r]_{G_{\Sigma}t}\ar@/^{0.5pc}/[r]^{F_{\Sigma}t}\ar@{}[r]|-{\Downarrow} & \bullet
}
\]
and
\[
\xymatrix{ &  & \ar@{}[d]|-{\Downarrow\eta_{Gh}}\\
\bullet\ar@/_{0.5pc}/[r]_{G_{\Delta}s}\ar@/^{0.5pc}/[r]^{F_{\Delta}s}\ar@{}[r]|-{\Downarrow} & \bullet\ar@/_{0pc}/[r]_{G_{\Sigma}h}\ar@/^{2.5pc}/[rr]^{\textnormal{id}} & \bullet\ar@/_{0pc}/[r]_{G_{\Delta}h} & \bullet\ar@/_{0.5pc}/[r]_{G_{\Sigma}t}\ar@/^{0.5pc}/[r]^{F_{\Sigma}t}\ar@{}[r]|-{\Downarrow} & \bullet
}
\]
agree; which is easily seen by expanding $\alpha_{\Delta_{h}}$ in
terms of $\alpha_{\Sigma_{h}}^{-1}$ and using the triangle identities.
The nullary icon condition is trivial. This shows that $\alpha$ indeed
admits the structure of an icon.

\noun{Essentially surjective}. Given any sinister pseudofunctor $F\colon\mathcal{E}\to\mathscr{C}$
we take the gregarious functor $L\colon\mathbf{Span}\left(\mathcal{E}\right)\to\mathscr{C}$
from Proposition \ref{spanLoplax} and note that $L_{\Sigma}=F$. 

We now verify the second universal property.

\noun{Restrictions}. The second property is a restriction of the first.
Indeed, given a pseudofunctor $L\colon\mathbf{Span}\left(\mathcal{E}\right)\to\mathscr{C}$
the corresponding pseudofunctor $L_{\Sigma}\colon\mathcal{E}\to\mathscr{C}$
satisfies the Beck condition, since the embedding $\left(-\right)_{\Sigma}:\mathcal{E}\to\mathbf{Span}\left(\mathcal{E}\right)$
satisfies the Beck condition. Moreover, given a sinister pseudofunctor
$F\colon\mathcal{E}\to\mathscr{C}$ which satisfies the Beck condition,
the corresponding map $\mathbf{Span}\left(\mathcal{E}\right)\to\mathscr{C}$
is pseudo since the oplax constraint cells of this functor are Beck
2-cells by Lemma \ref{oplaxBeck}.
\end{proof}

\section{Universal properties of spans with invertible 2-cells\label{unispaniso}}

In this section we derive the universal property of the bicategory
of spans with invertible 2-cells, denoted $\mathbf{Span}_{\textnormal{iso}}\left(\mathcal{E}\right)$.
Indeed, an understanding of this universal property will be required
for stating the universal property of polynomials with cartesian 2-cells
$\mathbf{Poly}_{c}\left(\mathcal{E}\right)$ described in the next
section.

\subsection{Stating the universal property}

The embeddings $\left(-\right)_{\Sigma}$ and $\left(-\right)_{\Delta}$
into $\mathbf{Span}_{\textnormal{iso}}\left(\mathcal{E}\right)$ are
defined the same as in the case of spans with the usual 2-cells. The
difference here is that we no longer have adjunctions $f_{\Sigma}\dashv f_{\Delta}$
in general, a fact which we will emphasize by replacing the symbol
$\Sigma$ with $\otimes$. Consequently the universal property is
more complicated to state, and so we will need some definitions.
\begin{defn}
\label{Beckpair} Given a category $\mathcal{E}$ with chosen pullbacks,
we may define the category of \emph{lax Beck pairs} on $\mathcal{E}$,
denoted $\mathbf{LaxBeckPair}\left(\mathcal{E},\mathscr{C}\right)$.
This category has objects given by pairs of pseudofunctors 
\[
F_{\otimes}\colon\mathcal{E}\to\mathscr{C},\qquad\qquad F_{\Delta}\colon\mathcal{E}^{\textnormal{op}}\to\mathscr{C}
\]
which agree on objects, equipped with, for each pullback square 
\[
\xymatrix{\bullet\ar[d]_{g'}\myar{f'}{r} & \bullet\ar[d]^{g} &  &  & \bullet\myar{F_{\otimes}f'}{r}\ar@{}[dr]|-{\Downarrow\mathfrak{b}_{f,g}^{f',g'}} & \bullet\\
\bullet\ar[r]_{f} & \bullet &  &  & \bullet\ar[r]_{F_{\otimes}f}\ar[u]^{F_{\Delta}g'} & \bullet\ar[u]_{F_{\Delta}g}
}
\]
in $\mathcal{E}$ as on the left, a 2-cell as on the right (which
we call a Beck 2-cell). The collection of these Beck 2-cells comprise
the ``Beck data'' denoted $^{F}\mathfrak{b}$ (or just $\mathfrak{b}$),
and are required to satisfy the following coherence conditions:
\begin{enumerate}
\item (horizontal double pullback condition) for any double pullback
\begin{equation}
\xymatrix@=1em{\bullet\ar[rr]^{f_{2}'}\ar[dd]_{g''} &  & \bullet\ar[dd]|-{g'}\myar{f_{1}'}{rr} &  & \bullet\ar[dd]^{g}\\
\\
\bullet\ar[rr]_{f_{2}} &  & \bullet\ar[rr]_{f_{1}} &  & \bullet
}
\label{pullbackcohhor}
\end{equation}
we have
\[
\xymatrix{ &  & \ar@{}[d]|-{\overset{\;}{\cong}}\\
\bullet\ar[rr]^{F_{\otimes}f_{2}'}\ar@/^{2pc}/[rrrr]^{F_{\otimes}\left(f_{1}'f_{2}'\right)} &  & \bullet\myar{F_{\otimes}f_{1}'}{rr}\ar@{}[dll]|-{\Downarrow\mathfrak{b}_{f_{2},g'}^{f'',g''}}\ar@{}[drr]|-{\Downarrow\mathfrak{b}_{f_{1},g}^{f_{1}',g'}} &  & \bullet & \ar@{}[d]|-{=} & \bullet\ar[rr]^{F_{\otimes}\left(f_{1}'f_{2}'\right)}\ar@{}[drr]|-{\Downarrow\mathfrak{b}_{f_{1}f_{2},g}^{f_{1}'f_{2}',g''}} &  & \bullet\\
\bullet\ar[rr]_{F_{\otimes}f_{2}}\ar[u]^{F_{\Delta}g''}\ar@/_{2pc}/[rrrr]_{F_{\otimes}\left(f_{1}'f_{2}'\right)} &  & \bullet\ar[rr]_{F_{\otimes}f_{1}}\ar[u]|-{F_{\Delta}g'} &  & \bullet\ar[u]_{F_{\Delta}g} & \; & \bullet\ar[u]^{F_{\Delta}g''}\ar[rr]_{F_{\otimes}\left(f_{1}f_{2}\right)} &  & \bullet\ar[u]_{F_{\Delta}g}\\
 &  & \ar@{}[u]|-{\underset{\;}{\cong}}
}
\]
\item (vertical double pullback condition) for any double pullback
\begin{equation}
\xymatrix@=1em{\bullet\ar[rr]^{f''}\ar[dd]_{g_{2}'} &  & \bullet\ar[dd]^{g_{2}}\\
\\
\bullet\ar[rr]|-{f'}\ar[dd]_{g_{1}'} &  & \bullet\ar[dd]^{g_{1}}\\
\\
\bullet\ar[rr]_{f} &  & \bullet
}
\label{pullbackcohvert}
\end{equation}
we have 
\[
\xymatrix{ & \bullet\ar[rr]^{F_{\otimes}f''}\ar@{}[drr]|-{\Downarrow\mathfrak{b}_{f,'g_{2}}^{f'',g_{2}'}} &  & \bullet &  &  &  & \bullet\ar[rr]^{F_{\otimes}f''}\ar@{}[ddrr]|-{\Downarrow\mathfrak{b}_{f,g_{1}g_{2}}^{f'',g_{1}'g_{2}'}} &  & \bullet\\
\ar@{}[r]|-{\;\;\;\;\cong} & \bullet\ar[u]|-{F_{\Delta}g_{2}'}\ar[rr]|-{F_{\otimes}f'}\ar@{}[drr]|-{\Downarrow\mathfrak{b}_{f,g_{1}}^{f',g_{1}'}} &  & \bullet\ar[u]|-{F_{\Delta}g_{2}} & \ar@{}[l]|-{\cong\;\;\;\;} & \ar@{}[r]|-{=\;} & \;\\
 & \bullet\ar[rr]_{F_{\otimes}f}\ar[u]|-{F_{\Delta}g_{1}'}\ar@/^{2pc}/[uu]^{F_{\Delta}\left(g_{1}'g_{2}'\right)} &  & \bullet\ar[u]|-{F_{\Delta}g_{1}}\ar@/_{2pc}/[uu]_{F_{\Delta}\left(g_{1}g_{2}\right)} &  &  &  & \bullet\ar[rr]_{F_{\otimes}f}\ar[uu]^{F_{\Delta}\left(g_{1}'g_{2}'\right)} &  & \bullet\ar[uu]_{F_{\Delta}\left(g_{1}g_{2}\right)}
}
\]
\item (horizontal nullary pullback condition) for any nullary pullback as
on the left below, the right pasting below is the identity
\[
\xymatrix{\bullet\ar[d]_{\textnormal{id}}\myar{f}{r} & \bullet\ar[d]^{\textnormal{id}} &  &  & \ar@{}[dr]|-{\cong\;} & \bullet\ar[rr]^{F_{\otimes}\left(f\right)}\ar@{}[drr]|-{\Downarrow\mathfrak{b}_{f,\textnormal{id}}^{f,\textnormal{id}}} &  & \bullet & \ar@{}[dl]|-{\;\cong}\\
\bullet\ar[r]_{f} & \bullet &  &  &  & \bullet\ar[rr]_{F_{\otimes}\left(f\right)}\ar[u]|-{F_{\Delta}\left(\textnormal{id}\right)}\ar@/^{2.5pc}/[u]^{\textnormal{id}} &  & \bullet\ar[u]|-{F_{\Delta}\left(\textnormal{id}\right)}\ar@/_{2.5pc}/[u]_{\textnormal{id}}
}
\]
\item (vertical nullary pullback condition) for any nullary pullback as
on the left below, the right pasting below is the identity
\[
\xymatrix{ &  &  &  &  &  & \ar@{}[d]|-{\overset{\;}{\cong}}\\
\bullet\ar[d]_{g}\myar{\textnormal{id}}{r} & \bullet\ar[d]^{g} &  &  &  & \bullet\ar[rr]|-{F_{\otimes}\left(\textnormal{id}\right)}\ar@{}[drr]|-{\Downarrow\mathfrak{b}_{\textnormal{id},g}^{\textnormal{id},g}}\ar@/^{2pc}/[rr]^{\textnormal{id}} &  & \bullet\\
\bullet\ar[r]_{\textnormal{id}} & \bullet &  &  &  & \bullet\ar[u]^{F_{\Delta}g}\ar[rr]|-{F_{\otimes}\left(\textnormal{id}\right)}\ar@/_{2pc}/[rr]_{\textnormal{id}} &  & \bullet\ar[u]_{F_{\Delta}g}\\
 &  &  &  &  &  & \ar@{}[u]|-{\underset{\;}{\cong}}
}
\]
\end{enumerate}
We refer to these conditions as the \emph{Beck\textendash Chevalley
coherence conditions}. A morphism in this category $\left(F_{\otimes},F_{\Delta},^{F}\mathfrak{b}\right)\to\left(G_{\otimes},G_{\Delta},^{G}\mathfrak{b}\right)$
is a pair of icons $\alpha\colon F_{\otimes}\Rightarrow G_{\otimes}$
and $\beta\colon F_{\Delta}\Rightarrow G_{\Delta}$ such that for
each pullback square as on the left below
\begin{equation}
\xymatrix{\bullet\ar[d]_{g'}\myar{f'}{r} & \bullet\ar[d]^{g} &  & F_{\otimes}f'\cdot F_{\Delta}g'\ar[d]_{^{F}\mathfrak{b}_{f,g}^{f',g'}}\ar[rr]^{\alpha_{f'}\ast\beta_{g'}} &  & G_{\otimes}f'\cdot G_{\Delta}g'\ar[d]^{^{G}\mathfrak{b}_{f,g}^{f',g'}}\\
\bullet\ar[r]_{f} & \bullet &  & F_{\Delta}g\cdot F_{\otimes}f\ar[rr]_{\beta_{g}\ast\alpha_{f}} &  & G_{\Delta}g\cdot G_{\otimes}f
}
\label{conditionBeck}
\end{equation}
the right diagram commutes. The category $\mathbf{BeckPair}\left(\mathcal{E},\mathscr{C}\right)$
is the subcategory of $\mathbf{LaxBeckPair}\left(\mathcal{E},\mathscr{C}\right)$
containing objects $\left(F_{\otimes},F_{\Delta},^{F}\mathfrak{b}\right)$
such that every Beck 2-cell in $^{F}\mathfrak{b}$ is invertible.
\end{defn}

Before we can state the universal property, we will need to describe
how lax Beck pairs arise from suitable functors out of $\mathbf{Span}_{\textnormal{iso}}\left(\mathcal{E}\right)$.
\begin{defn}
Let $\mathcal{E}$ be a category with pullbacks (chosen such that
identities pullback to identities) and let $\mathscr{C}$ be a bicategory.
Then the category 
\[
\mathbf{Greg}_{\otimes,\Delta}\left(\mathbf{Span}_{\textnormal{iso}}\left(\mathcal{E}\right),\mathscr{C}\right)
\]
 has objects given by those gregarious functors of bicategories $\mathbf{Span}_{\textnormal{iso}}\left(\mathcal{E}\right)\to\mathscr{C}$
which restrict to pseudofunctors when composed with the canonical
embeddings $\left(-\right)_{\otimes}\colon\mathcal{E}\to\mathbf{Span}_{\textnormal{iso}}\left(\mathcal{E}\right)$
and $\left(-\right)_{\Delta}\colon\mathcal{E}\to\mathbf{Span}_{\textnormal{iso}}\left(\mathcal{E}\right)$.
Moreover, we require that each oplax constraint
\begin{equation}
F\left(\vcenter{\hbox{\xymatrix@=1em{ & \bullet\ar[rd]^{t}\ar[ld]_{s}\\
\bullet &  & \bullet
}
}}\right)\to F\left(\vcenter{\hbox{\xymatrix@=1em{ & \bullet\ar[rd]^{\textnormal{id}}\ar[ld]_{s}\\
\bullet &  & \bullet
}
}}\right);F\left(\vcenter{\hbox{\xymatrix@=1em{ & \bullet\ar[rd]^{t}\ar[ld]_{\textnormal{id}}\\
\bullet &  & \bullet
}
}}\right)\label{invconstraint}
\end{equation}
be invertible. The morphisms of this category are icons.
\end{defn}

\begin{prop}
\label{spanisoemb} Let $\mathcal{E}$ be a category with pullbacks
(chosen such that identities pullback to identities) and let $\mathscr{C}$
be a bicategory. We then have a functor
\[
\left(-\right)_{\otimes\Delta}\colon\mathbf{Greg}_{\otimes,\Delta}\left(\mathbf{Span}_{\textnormal{iso}}\left(\mathcal{E}\right),\mathscr{C}\right)\to\mathbf{LaxBeckPair}\left(\mathcal{E},\mathscr{C}\right)
\]
defined by the assignment taking such a gregarious functor $F\colon\mathbf{Span}_{\textnormal{iso}}\left(\mathcal{E}\right)\to\mathscr{C}$
to the pair of pseudofunctors
\[
F_{\otimes}\colon\mathcal{E}\to\mathscr{C},\qquad F_{\Delta}\colon\mathcal{E}^{\textnormal{op}}\to\mathscr{C}
\]
equipped with Beck data $^{F}\mathfrak{b}$ given by, for each pullback
square as on the left below (with the chosen pullback on the right
below)
\[
\xymatrix{\bullet\ar[d]_{g'}\myar{f'}{r} & \bullet\ar[d]^{g} &  &  & \bullet\ar[d]_{\widetilde{g}}\myar{\widetilde{f}}{r} & \bullet\ar[d]^{g}\\
\bullet\ar[r]_{f} & \bullet &  &  & \bullet\ar[r]_{f} & \bullet
}
\]
the composite of: 
\begin{enumerate}
\item the inverse of an oplax constraint cell 
\[
F\left(\vcenter{\hbox{\xymatrix@=1em{ & \bullet\ar[rd]^{\textnormal{id}}\ar[ld]_{g'}\\
\bullet &  & \bullet
}
}}\right);F\left(\vcenter{\hbox{\xymatrix@=1em{ & \bullet\ar[rd]^{f'}\ar[ld]_{\textnormal{id}}\\
\bullet &  & \bullet
}
}}\right)\to F\left(\vcenter{\hbox{\xymatrix@=1em{ & \bullet\ar[rd]^{f'}\ar[ld]_{g'}\\
\bullet &  & \bullet
}
}}\right)
\]
\item the application of $F$ to the induced isomorphism of pullbacks
\[
F\left(\vcenter{\hbox{\xymatrix@=1em{ & \bullet\ar[rd]^{f'}\ar[ld]_{g'}\\
\bullet &  & \bullet
}
}}\right)\to F\left(\vcenter{\hbox{\xymatrix@=1em{ & \bullet\ar[rd]^{\widetilde{f}}\ar[ld]_{\widetilde{g}}\\
\bullet &  & \bullet
}
}}\right)
\]
\item the oplax constraint cell 
\[
F\left(\vcenter{\hbox{\xymatrix@=1em{ & \bullet\ar[rd]^{\widetilde{f}}\ar[ld]_{\widetilde{g}}\\
\bullet &  & \bullet
}
}}\right)\to F\left(\vcenter{\hbox{\xymatrix@=1em{ & \bullet\ar[rd]^{f}\ar[ld]_{\textnormal{id}}\\
\bullet &  & \bullet
}
}}\right);F\left(\vcenter{\hbox{\xymatrix@=1em{ & \bullet\ar[rd]^{\textnormal{id}}\ar[ld]_{g}\\
\bullet &  & \bullet
}
}}\right)
\]
\end{enumerate}
\end{prop}

\begin{proof}
We must check the Beck 2-cells defined as above satisfy the required
coherence conditions. The nullary conditions on the Beck 2-cells are
trivially equivalent to the nullary conditions on the constraints
of $F$. To see the ``horizontal double pullback condition'' holds,
we note that since $F\colon\mathbf{Span}_{\textnormal{iso}}\left(\mathcal{E}\right)\to\mathscr{C}$
is normal oplax, we have a resulting natural transformation
\[
N\left(F\right)\colon N\left(\mathbf{Span}_{\textnormal{iso}}\left(\mathcal{E}\right)\right)\to N\left(\mathscr{C}\right)
\]
where the functor $N\colon\mathbf{Bicat}\to\left[\Delta^{\textnormal{op}},\mathbf{Set}\right]$
is given by the geometric nerve \cite{streetnerve}. In particular
(as in \cite{nerve2cat}), on 2-simplices the assignment
\[
\xymatrix@=1em{ &  & \bullet\ar[rrdd]^{b} &  &  &  &  &  & \bullet\ar[rrdd]^{Fb}\\
 &  &  &  &  & \mapsto\\
\bullet\ar[rrrr]_{c}\ar[rruu]^{a} &  & \;\utwocell[0.4]{uu}{\alpha} &  & \bullet &  & \bullet\ar[rrrr]_{Fc}\ar[rruu]^{Fa} &  & \;\utwocell[0.4]{uu}{\overline{F\mathfrak{\alpha}}} &  & \bullet
}
\]
(where $\overline{F\mathfrak{\alpha}}$ is $F\alpha$ composed with
the appropriate oplax constraint cell) satisfies the condition that
\[
\xymatrix@=1em{\bullet\ar[rrrr] &  &  &  & \bullet\ar[dd] &  & \bullet\ar[rrrr]\ar[rrrrdd] &  &  & \; & \bullet\ar[dd]\\
 &  & \utwocell[0.4]{dr}{\mathfrak{\alpha}}\ultwocell{ull}{\mathfrak{\beta}} &  &  & = &  &  & \utwocell[0.4]{dll}{\mathfrak{\gamma}}\urtwocell[0.4]{urr}{\mathfrak{\delta}}\\
\bullet\ar[rrrr]\ar[uu]\ar[uurrrr] &  & \; & \; & \bullet &  & \bullet\ar[rrrr]\ar[uu] &  & \; &  & \bullet
}
\]
implies that 
\[
\xymatrix@=1em{\bullet\ar[rrrr] &  &  &  & \bullet\ar[dd] &  & \bullet\ar[rrrr]\ar[rrrrdd] &  &  & \; & \bullet\ar[dd]\\
 &  & \utwocell[0.4]{dr}{\overline{F\mathfrak{\alpha}}}\ultwocell{ull}{\overline{F\mathfrak{\beta}}} &  &  & = &  &  & \utwocell[0.4]{dll}{\overline{F\mathfrak{\gamma}}}\urtwocell[0.4]{urr}{\overline{F\mathfrak{\delta}}}\\
\bullet\ar[rrrr]\ar[uu]\ar[uurrrr] &  & \; &  & \bullet &  & \bullet\ar[rrrr]\ar[uu] &  & \; &  & \bullet
}
\]

Now consider the three spans
\[
\xymatrix@=1em{ & \bullet\ar[rd]^{f_{2}}\ar[ld]_{\textnormal{id}} &  &  &  & \bullet\ar[rd]^{f_{1}}\ar[ld]_{\textnormal{id}} &  &  &  & \bullet\ar[rd]^{\textnormal{id}}\ar[ld]_{g}\\
\bullet &  & \bullet &  & \bullet &  & \bullet &  & \bullet &  & \bullet
}
\]
which we denote by shorthand as $\left(1,f_{2}\right)$, $\left(1,f_{1}\right)$
and $\left(g,1\right)$ respectively (where $f_{1}$, $f_{2}$ and
$g$ are as in \eqref{pullbackcohhor}). Applying the above implication
to the equality below, where each of the four regions contains a canonical
isomorphism or equality of spans
\[
\xymatrix@=1em{\bullet\ar[rrrr]^{\left(1,f_{1}\right)} &  &  &  & \bullet\ar[dd]^{\left(g,1\right)} &  &  &  & \bullet\ar[rrrr]^{\left(1,f_{1}\right)}\ar[rrrrdd]|-{\left(g',f_{1}'\right)} &  &  & \; & \bullet\ar[dd]^{\left(g,1\right)}\\
 &  &  &  &  &  & =\\
\bullet\ar[rrrr]_{\left(g'',f_{1}'f_{2}'\right)}\ar[uu]^{\left(1,f_{2}\right)}\ar[uurrrr]|-{\left(1,f_{1}f_{2}\right)} &  & \; &  & \bullet &  &  &  & \bullet\ar[rrrr]_{\left(g'',f_{1}'f_{2}'\right)}\ar[uu]^{\left(1,f_{2}\right)} &  & \; &  & \bullet
}
\]
then gives the horizontal double pullback condition (after composing
with the appropriate pseudofunctoriality constraints of $F_{\Sigma}$
and constraints of the form \eqref{invconstraint}). The proof of
the vertical condition is similar. Finally, it is clear the canonical
assignation on morphisms is well defined, and the assignment given
by composing with the canonical embeddings is trivially functorial. 
\end{proof}
We can now state the universal property of $\mathbf{Span}_{\textnormal{iso}}\left(\mathcal{E}\right)$.
\begin{thm}
\label{unispanisothm} Given a category $\mathcal{E}$ with chosen
pullbacks (chosen such that identities pullback to identities), the
functor $\left(-\right)_{\otimes\Delta}$ of Proposition \ref{spanisoemb}
defines the equivalence of categories
\[
\mathbf{Greg}_{\otimes,\Delta}\left(\mathbf{Span}_{\textnormal{iso}}\left(\mathcal{E}\right),\mathscr{C}\right)\simeq\mathbf{LaxBeckPair}\left(\mathcal{E},\mathscr{C}\right)
\]
which restricts to the equivalence
\[
\mathbf{Icon}\left(\mathbf{Span}_{\textnormal{iso}}\left(\mathcal{E}\right),\mathscr{C}\right)\simeq\mathbf{BeckPair}\left(\mathcal{E},\mathscr{C}\right)
\]
for any bicategory $\mathscr{C}$.
\end{thm}

\subsection{Proving the universal property}

We prove Theorem \ref{unispanisothm} directly, as the properties
of generic bicategories cannot be used here. Also, for simplicity
we assume without loss of generality that $\mathscr{C}$ is a 2-category
and that the gregarious functors in question strictly preserve identities.
This is justified since every bicategory is equivalent to a 2-category
and every normal oplax functor is isomorphic to one which preserves
identity 1-cells strictly.
\begin{proof}
[Proof of Theorem \ref{unispanisothm}] We start by proving the first
universal property. We must prove that the functor
\[
\left(-\right)_{\otimes\Delta}\colon\mathbf{Greg}_{\otimes,\Delta}\left(\mathbf{Span}_{\textnormal{iso}}\left(\mathcal{E}\right),\mathscr{C}\right)\to\mathbf{LaxBeckPair}\left(\mathcal{E},\mathscr{C}\right)
\]
defines an equivalence of categories.

\noun{Essentially Surjective}. Given such a pair $F_{\Sigma}$ and
$F_{\Delta}$ with Beck data $\mathfrak{b}$ we may define local functors
\[
L_{X,Y}:\mathbf{Span}_{\textnormal{iso}}\left(\mathcal{E}\right)_{X,Y}\to\mathscr{C}_{X,Y},\qquad X,Y\in\mathcal{E}
\]
by the assignment (suppressing pseudofunctoriality of $F_{\Sigma}$
and $F_{\Delta}$) 
\[
\xymatrix@=1.5em{ & E\ar[dl]_{s}\ar[dd]|-{f}\ar[dr]^{t} &  &  &  &  & FE\ar[dr]^{F_{\otimes}f}\\
X &  & Y & \mapsto & FX\ar[r]^{F_{\Delta}u} & FM\ar[ur]^{F_{\Delta}f}\ar[rd]_{F_{\otimes}1} & \;\ar@{}[]|-{\Downarrow\mathfrak{b}_{1,1}^{f,f}} & FM\ar[r]^{F_{\otimes}v} & FY\\
 & M\ar[ul]^{u}\ar[ur]_{v} &  &  &  &  & FM\ar[ru]_{F_{\Delta}1}
}
\]
which is functorial by the Beck coherence conditions. An oplax constraint
cell 
\[
L\left(\vcenter{\hbox{\xymatrix@=1em{ & \bullet\ar[dr]^{qv'}\ar[dl]_{up'}\\
\bullet &  & \bullet
}
}}\right)\to L\left(\vcenter{\hbox{\xymatrix@=1em{ & \bullet\ar[dr]^{v}\ar[dl]_{u}\\
\bullet &  & \bullet
}
}}\right);L\left(\vcenter{\hbox{\xymatrix@=1em{ & \bullet\ar[dr]^{q}\ar[dl]_{p}\\
\bullet &  & \bullet
}
}}\right)
\]
is given by (suppressing pseudofunctoriality of $F_{\otimes}$ and
$F_{\Delta}$) 
\[
\xymatrix@=1.5em{ &  & \bullet\ar[dr]^{F_{\otimes}v'}\\
\bullet\ar[r]^{F_{\Delta}u} & \bullet\ar[ur]^{F_{\Delta}p'}\ar[dr]_{F_{\Sigma}v} & \;\ar@{}[]|-{\Downarrow\mathfrak{b}_{v,p}^{v',p'}} & \bullet\ar[r]^{F_{\otimes}q} & \bullet\\
 &  & \bullet\ar[ur]_{F_{\Delta}p}
}
\]

That these constraints satisfy the identity conditions trivially follows
from the unit condition on the Beck 2-cells. For the associativity
condition, suppose we are given diagrams of chosen pullbacks as below,
with $p$ the induced isomorphism of generalized pullbacks, that is
the associator for the triple $\left(a,b\right),\left(c,d\right),\left(e,f\right)$),
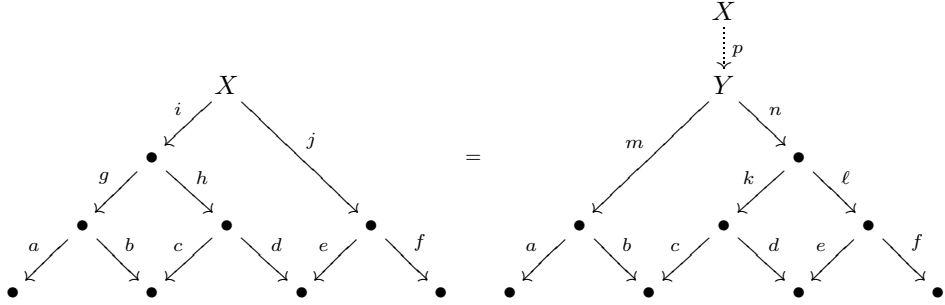
\[
\xymatrix@=1.5em{ &  &  &  &  &  &  &  &  &  & X\ar@{..>}[d]^{p}\\
 &  &  & X\ar[ld]_{i}\ar[rdrd]^{j} &  &  &  &  &  &  & Y\ar[ldld]_{m}\ar[rd]^{n}\\
 &  & \bullet\ar[ld]_{g}\ar[rd]^{h} &  &  &  & \ar@{}[r]|-{=} & \; &  &  &  & \bullet\ar[ld]_{k}\ar[rd]^{\ell}\\
 & \bullet\ar[rd]^{b}\ar[ld]_{a} &  & \bullet\ar[rd]^{d}\ar[ld]_{c} &  & \bullet\ar[rd]^{f}\ar[ld]_{e} &  &  & \bullet\ar[rd]^{b}\ar[ld]_{a} &  & \bullet\ar[rd]^{d}\ar[ld]_{c} &  & \bullet\ar[rd]^{f}\ar[ld]_{e}\\
\bullet &  & \bullet &  & \bullet &  & \bullet & \bullet &  & \bullet &  & \bullet &  & \bullet
}
\]
Then we must check that
\[
\xymatrix{L\left(agi,fj\right)\ar[d]_{Lp}\myar{\varphi}{r} & L\left(ag,dh\right);L\left(e,f\right)\myar{\varphi}{r} & \left(L\left(a,b\right);L\left(c,d\right)\right);L\left(e,f\right)\ar@{=}[d]\\
L\left(am,f\ell n\right)\myard{\varphi}{r} & L\left(a,b\right);L\left(ck,f\ell\right)\myard{\varphi}{r} & L\left(a,b\right);\left(L\left(c,d\right);L\left(e,f\right)\right)
}
\]
 commutes. The top path is a pasting of Beck 2-cells corresponding\footnote{By ``corresponding'' we mean that one assigns each pullback square
in \eqref{spanisocohdiag} to the Beck data for that square, as in
Definition \ref{Beckpair}.} to the left diagram in \eqref{spanisocohdiag} below, and the bottom
path is the pasting of Beck 2-cells corresponding to the right diagram
in \eqref{spanisocohdiag} below (suppressing pseudofunctoriality
constraints of $F_{\Sigma}$ and $F_{\Delta}$) which are equal by
the Beck coherence conditions.
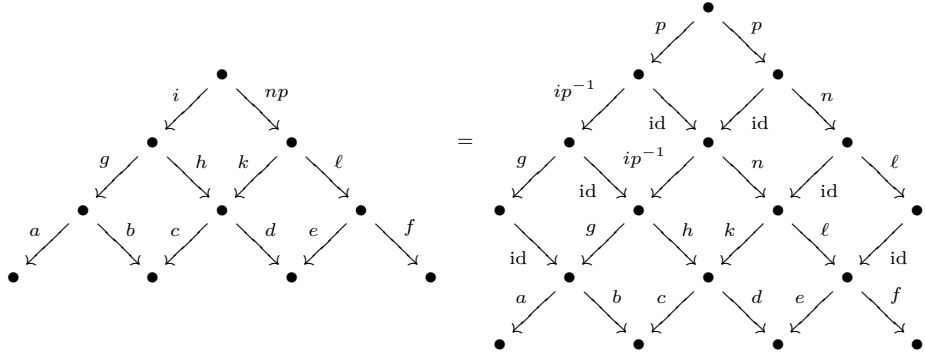
\begin{equation}
\xymatrix@=1.5em{ &  &  &  &  &  &  &  &  &  & \bullet\ar[ld]_{p}\ar[rd]^{p}\\
 &  &  & \bullet\ar[ld]_{i}\ar[rd]^{np} &  &  &  &  &  & \bullet\ar[dr]_{\textnormal{id}}\ar[ld]_{ip^{-1}} &  & \bullet\ar[dl]^{\textnormal{id}}\ar[rd]^{n}\\
 &  & \bullet\ar[ld]_{g}\ar[rd]^{h} &  & \bullet\ar[rd]^{\ell}\ar[ld]_{k} &  & \; & \;\ar@{}[l]|-{=} & \bullet\ar[ld]_{g}\ar[dr]_{\textnormal{id}} &  & \bullet\ar[ld]_{ip^{-1}}\ar[rd]^{n} &  & \bullet\ar[rd]^{\ell}\ar[dl]^{\textnormal{id}}\\
 & \bullet\ar[rd]^{b}\ar[ld]_{a} &  & \bullet\ar[rd]^{d}\ar[ld]_{c} &  & \bullet\ar[rd]^{f}\ar[ld]_{e} &  & \bullet\ar[dr]_{\textnormal{id}} &  & \bullet\ar[ld]_{g}\ar[rd]^{h} &  & \bullet\ar[ld]_{k}\ar[rd]^{\ell} &  & \bullet\ar[dl]^{\textnormal{id}}\\
\bullet &  & \bullet &  & \bullet &  & \bullet &  & \bullet\ar[rd]^{b}\ar[ld]_{a} &  & \bullet\ar[rd]^{d}\ar[ld]_{c} &  & \bullet\ar[rd]^{f}\ar[ld]_{e}\\
 &  &  &  &  &  &  & \bullet &  & \bullet &  & \bullet &  & \bullet
}
\label{spanisocohdiag}
\end{equation}
For checking the oplax constraint cells are natural, consider a pair
of morphisms of spans
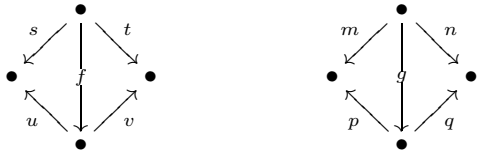
\[
\xymatrix@=1.5em{ & \bullet\ar[dl]_{s}\ar[dd]|-{f}\ar[dr]^{t} &  &  &  &  & \bullet\ar[dl]_{m}\ar[dd]|-{g}\ar[dr]^{n}\\
\bullet &  & \bullet &  &  & \bullet &  & \bullet\\
 & \bullet\ar[ul]^{u}\ar[ur]_{v} &  &  &  &  & \bullet\ar[ul]^{p}\ar[ur]_{q}
}
\]
We must check the commutativity of
\[
\xymatrix{L\left(sm',nt'\right)\ar[d]_{L\left(f;g\right)}\myar{\varphi}{r} & L\left(s,t\right);L\left(m,n\right)\ar[d]^{Lf;Lg}\\
L\left(up',qv'\right)\myard{\varphi}{r} & L\left(u,v\right);L\left(p,q\right)
}
\]
The top path of this diagram corresponds to a pasting of Beck 2-cells
for the left diagram below, and the bottom path corresponds to the
pasting of Beck 2-cells for the right diagram below. Hence the commutativity
of this diagram amounts to applying the Beck coherence conditions
to the diagrams of pullbacks (which compose to the same pullback)
\[
\xymatrix@=0.5em{\bullet\ar[rrrr]^{t'}\ar[dddd]_{m'} &  &  &  & \bullet\ar[rr]^{g}\ar[dd]_{g} &  & \bullet\ar[dd]^{1}\\
 &  &  &  &  &  &  &  &  &  &  &  & \bullet\ar[rr]^{h}\ar[dd]_{h} &  & \bullet\ar[rr]^{v'}\ar[dd]^{1} &  & \bullet\ar[dd]^{1}\\
 &  &  &  & \bullet\ar[dd]_{p}\ar[rr]^{1} &  & \bullet\ar[dd]^{p}\\
 &  &  &  &  &  &  &  &  & = &  &  & \bullet\ar[rr]_{1}\ar[dd]_{p'} &  & \bullet\ar[rr]^{v'}\ar[dd]_{p'} &  & \bullet\ar[dd]^{p}\\
\bullet\ar[dd]_{f}\ar[rr]^{f} &  & \bullet\ar[rr]^{v}\ar[dd]_{1} &  & \bullet\ar[rr]^{1}\ar[dd]_{1} &  & \bullet\ar[dd]^{1}\\
 &  &  &  &  &  &  &  &  &  &  &  & \bullet\ar[rr]_{1} &  & \bullet\ar[rr]_{v} &  & \bullet\\
\bullet\ar[rr]_{1} &  & \bullet\ar[rr]_{v} &  & \bullet\ar[rr]_{1} &  & \bullet
}
\]
where $h$ is the morphism of spans arising from horizontally composing
the morphisms of spans $f$ and $g$.

Inherited from Proposition \ref{spanadjprop} is the fact that all
adjunctions in $\mathbf{Span}_{\textnormal{iso}}\left(\mathcal{E}\right)$
are of the form $\left(1,f\right)\dashv\left(f,1\right)$ up to isomorphism,
where $f$ must be invertible. To see $F$ is gregarious, meaning
that this gives an adjunction $F_{\Sigma}f\dashv F_{\Delta}f$ in
$\mathscr{C}$, note that we may construct the unit and counit as
the Beck 2-cells arising from the pullback squares
\[
\xymatrix{\bullet\ar[r]^{1}\ar[d]_{1} & \bullet\ar[d]^{f} &  &  & \bullet\ar[r]^{f}\ar[d]_{f} & \bullet\ar[d]^{1}\\
\bullet\ar[r]_{f} & \bullet &  &  & \bullet\ar[r]_{1} & \bullet
}
\]

\noun{Fully Faithful}. Suppose we are given two gregarious functors
$F,G\colon\mathbf{Span}_{\textnormal{iso}}\left(\mathcal{E}\right)\to\mathscr{C}$
along with their restrictions $F_{\otimes},G_{\otimes}$ and $F_{\Delta},G_{\Delta}$
and families of Beck 2-cells $^{F}\mathfrak{b}$ and $^{G}\mathfrak{b}$.

We first check the assignment of icons is surjective. Suppose we are
given icons $\alpha\colon F_{\otimes}\to G_{\otimes}$ and $\beta\colon F_{\Delta}\to G_{\Delta}$
such that \eqref{conditionBeck} holds. Then we may define an icon
$\gamma\colon F\to G$ on each span $\left(s,t\right)$ by
\begin{equation}
\xymatrix{F\left(s,t\right)\ar[rr]^{\gamma_{s,t}}\ar[d]_{\varphi} &  & G\left(s,t\right)\ar[d]^{\psi}\\
F_{\otimes}t\cdot F_{\Delta}s\ar[rr]_{\alpha_{t}\ast\beta_{s}} &  & G_{\otimes}t\cdot G_{\Delta}s
}
\label{applies}
\end{equation}
where $\varphi$ and $\psi$ are the appropriate oplax constraint
cells (necessarily invertible above). Now \eqref{conditionBeck} forces
$\gamma$ to be locally natural, as it suffices to check naturality
on generating 2-cells, that is diagrams such as 
\[
\xymatrix@=1em{ & \bullet\ar[rd]^{f}\ar[dd]^{f}\ar[dl]_{f}\\
\bullet &  & \bullet\\
 & \bullet\ar[ur]_{1}\ar[ul]^{1}
}
\]
with $f$ invertible (this only needs trivial pullbacks corresponding
to $\mathfrak{b}_{1,1}^{f,f}$). For checking $\gamma$ is an icon,
the identity condition on $\gamma$ is from that of $\alpha$ and
$\beta$. The composition condition is precisely \eqref{conditionBeck}. 

We now check that the assignment of icons is injective. Suppose two
given icons $\sigma,\delta$ both restrict to icons $\alpha$ and
$\beta$. Then since the icons $\sigma$ and $\delta$ respect the
composite of the spans
\[
\xymatrix@=1em{ & \bullet\ar[rd]^{1}\ar[ld]_{s} &  &  &  & \bullet\ar[rd]^{t}\ar[ld]_{1}\\
\bullet &  & \bullet &  & \bullet &  & \bullet
}
\]
both $\sigma$ and $\delta$ must satisfy \eqref{applies} (in place
of $\gamma$) and so are equal.

\noun{Restrictions. }It is clear from the above that the oplax constraints
are invertible precisely when the Beck data is invertible.
\end{proof}

\section{Universal properties of polynomials with cartesian 2-cells\label{unipolycart}}

In this section we prove the universal property of the bicategory
of polynomials with cartesian 2-cells, denoted $\mathbf{Poly}_{c}\left(\mathcal{E}\right)$.
We will keep the proof as analogous to the case of spans as possible,
though it still becomes somewhat more complicated. 

\subsection{Stating the universal property}

This universal property of $\mathbf{Poly}_{c}\left(\mathcal{E}\right)$
turns out to be an amalgamation of that of $\mathbf{Span}\left(\mathcal{E}\right)$
and $\mathbf{Span}_{\textnormal{iso}}\left(\mathcal{E}\right)$; in
particular to give a pseudofunctor $\mathbf{Poly}_{c}\left(\mathcal{E}\right)\to\mathscr{C}$
is to give a pair of pseudofunctors 
\[
\mathbf{Span}\left(\mathcal{E}\right)\to\mathscr{C},\qquad\mathbf{Span}_{\textnormal{iso}}\left(\mathcal{E}\right)\to\mathscr{C}
\]
which ``$\Delta$-agree'' , that is coincide on objects and on spans
of the form
\[
\xymatrix{Y & X\ar[r]^{1_{X}}\ar[l]_{f} & X}
\]
with an additional condition asking that certain ``distributivity
morphisms'' be invertible. For the purposes of the proof we will
give a slightly different but equivalent description, for which we
will need the following definitions.
\begin{defn}
\label{distpair} Given a category $\mathcal{E}$ with chosen pullbacks,
we may define the category of \emph{lax Beck triples} from $\mathcal{E}$
to a bicategory $\mathscr{C}$, denoted $\mathbf{LaxBeckTriple}\left(\mathcal{E},\mathscr{C}\right)$.
An object consists of a triple of pseudofunctors which agree on objects
\[
F_{\Sigma}\colon\mathcal{E}\to\mathscr{C},\qquad F_{\Delta}\colon\mathcal{E}^{\textnormal{op}}\to\mathscr{C},\qquad F_{\otimes}\colon\mathcal{E}\to\mathscr{C}
\]
such that $F_{\Sigma}f\dashv F_{\Delta}f$ for all morphisms $f$
in $\mathcal{E}$, along with ``Beck data'' denoted by $^{F}\mathfrak{b}$
and consisting of for each pullback square
\begin{equation}
\xymatrix{\bullet\ar[d]_{g'}\myar{f'}{r} & \bullet\ar[d]^{g} &  &  & \bullet\myar{F_{\otimes}f'}{r}\ar@{}[dr]|-{\Downarrow\mathfrak{b}_{f,g}^{f',g'}} & \bullet\\
\bullet\ar[r]_{f} & \bullet &  &  & \bullet\ar[r]_{F_{\otimes}f}\ar[u]^{F_{\Delta}g'} & \bullet\ar[u]_{F_{\Delta}g}
}
\label{Beckmorphismscart}
\end{equation}
in $\mathcal{E}$ as on the left, a 2-cell as on the right subject
to the binary and nullary Beck coherence conditions as in Definition
\ref{Beckpair}. 

A morphism $\left(F_{\Sigma},F_{\Delta},F_{\otimes},{}^{F}\mathfrak{b}\right)\to\left(G_{\Sigma},G_{\Delta},G_{\otimes},{}^{G}\mathfrak{b}\right)$
in this category consists of an invertible icon $\beta\colon F_{\Delta}\Rightarrow G_{\Delta}$
and icon $\gamma\colon F_{\otimes}\Rightarrow G_{\otimes}$ such that
for each pullback square in $\mathcal{E}$ as above, the diagram
\begin{equation}
\xymatrix{F_{\otimes}f'\cdot F_{\Delta}g'\ar[r]^{\gamma_{f'}\ast\beta_{g'}}\ar[d]_{^{F}\mathfrak{b}_{f,g}^{f',g'}} & G_{\otimes}f'\cdot G_{\Delta}g'\ar[d]^{^{G}\mathfrak{b}_{f,g}^{f',g'}}\\
F_{\Delta}g\cdot F_{\otimes}f\ar[r]_{\beta_{g}\ast\gamma_{f}} & G_{\Delta}g\cdot G_{\otimes}f
}
\label{polyBeckcoh}
\end{equation}
commutes.
\end{defn}

There are a number of conditions which may be imposed on a lax Beck
triple; these are defined as follows.
\begin{defn}
\label{distcondition} We say a\emph{ }lax Beck triple $\left(F_{\Sigma},F_{\Delta},F_{\otimes},{}^{F}\mathfrak{b}\right)$
from $\mathcal{E}$ to a bicategory $\mathscr{C}$ is a \emph{Beck
triple }if both:
\begin{enumerate}
\item the $\Delta\otimes$ condition holds; meaning each component of the
Beck data $\mathfrak{b}_{f,g}^{f',g'}$ is invertible;
\item the $\Sigma\Delta$ condition holds; meaning each component of the
$F_{\Sigma}$-$F_{\Delta}$ Beck data is invertible\footnote{This is equivalent to asking the gregarious functor $\mathbf{Span}\left(\mathcal{E}\right)\to\mathscr{C}$
resulting from $F_{\Sigma}$ be a pseudofunctor.};
\end{enumerate}
Furthermore, we say such a Beck triple is a \emph{distributive Beck
triple} if in addition:

\begin{enumerate}\setcounter{enumi}{2}\item{ the $\Sigma\otimes$
condition (distributivity condition) holds; meaning that for any distributivity
pullback in $\mathcal{E}$ as on the left below
\begin{equation}
\begin{aligned}\xymatrix@=3em{\bullet\ar[r]^{q}\ar[d]_{p} & \bullet\ar[dd]^{r} &  &  & \bullet\ar[r]^{F_{\otimes}q}\ar@{}[drd]|-{\Downarrow\mathfrak{b}_{f,r}^{q,up'}} & \bullet\ar[rd]^{F_{\Sigma}r}\\
\bullet\ar[d]_{u} &  &  & \bullet\ar[rd]_{F_{\Sigma}u}^{\!\!\!\!\Downarrow\eta_{Fu}}\ar[r]^{\textnormal{id}}\ar[ur]_{=}^{F_{\Delta}p} & \bullet\ar[u]|-{F_{\Delta p}} & \ar@{}[r]|-{\Downarrow\epsilon_{Fr}} & \bullet\\
\bullet\ar[r]_{f} & \bullet &  &  & \bullet\ar[r]_{F_{\otimes}f}\ar[u]|-{F_{\Delta}u} & \bullet\ar[uu]|-{F_{\Delta}r}\ar[ru]_{\textnormal{id}}
}
\end{aligned}
\label{distpbexample}
\end{equation}
the corresponding ``distributivity morphism'' (defined as the pasting
on the right above) is invertible. }\end{enumerate}.

In particular, we define a \emph{Beck triple }to be a lax Beck triple
such that both conditions $\left(1\right)$ and $\left(2\right)$
hold, and a \emph{DistBeck triple} to be a Beck triple also satisfying
$\left(3\right)$. We denote the corresponding subcategories of $\mathbf{LaxBeckTriple}\left(\mathcal{E},\mathscr{C}\right)$
as $\mathbf{BeckTriple}\left(\mathcal{E},\mathscr{C}\right)$ and
$\mathbf{DistBeckTriple}\left(\mathcal{E},\mathscr{C}\right)$ respectively.
\end{defn}

There are a number of canonical embeddings into $\mathbf{Poly}_{c}\left(\mathcal{E}\right)$
to mention; the most obvious being the embeddings
\[
\left(-\right)_{\Sigma}:\mathcal{E}\to\mathbf{Poly}_{c}\left(\mathcal{E}\right),\quad\left(-\right)_{\Delta}:\mathcal{E}^{\textnormal{op}}\to\mathbf{Poly}_{c}\left(\mathcal{E}\right),\quad\left(-\right)_{\otimes}:\mathcal{E}\to\mathbf{Poly}_{c}\left(\mathcal{E}\right)
\]
 which are defined on objects by sending an object of $\mathcal{E}$
to itself, and are defined on each morphism in $\mathcal{E}$ by the
assignments
\[
\xymatrix@R=1em{\left(-\right)_{\Sigma}\colon & X\ar[r]^{f} & Y & \mapsto & X & X\ar[l]_{1_{X}}\ar[r]^{1_{X}} & X\ar[r]^{f} & Y\\
\left(-\right)_{\Delta}\colon & X\ar[r]^{f} & Y & \mapsto & Y & X\ar[l]_{f}\ar[r]^{1_{X}} & X\ar[r]^{1_{X}} & X\\
\left(-\right)_{\otimes}\colon & X\ar[r]^{f} & Y & \mapsto & Y & X\ar[r]^{f}\ar[l]_{1_{X}} & Y\ar[r]^{1_{Y}} & Y
}
\]
We also have the inclusion $\left(-\right)_{\Sigma\Delta}\colon\mathbf{Span}\left(\mathcal{E}\right)\to\mathbf{Poly}_{c}\left(\mathcal{E}\right)$
of spans into polynomials given by the assignment 
\[
\xymatrix{ & \bullet\ar[rd]^{t}\ar[ld]_{s}\ar[dd]|-{f} &  &  &  & \bullet\ar[dl]_{s}\ar[r]^{\textnormal{id}}\ar[dd]_{f} & \bullet\ar[rd]^{t}\ar[dd]^{f}\\
\bullet &  & \bullet & \mapsto & \bullet &  &  & \bullet\\
 & \bullet\ar[ru]_{v}\ar[lu]^{u} &  &  &  & \bullet\ar[ul]^{u}\ar[r]_{\textnormal{id}} & \bullet\ar[ur]_{v}
}
\]
The less obvious embedding $\left(-\right)_{\Delta\otimes}\colon\mathbf{Span}_{\textnormal{iso}}\left(\mathcal{E}\right)\to\mathbf{Poly}_{c}\left(\mathcal{E}\right)$
is the canonical embedding of spans with invertible 2-cells into polynomials,
given by the assignment
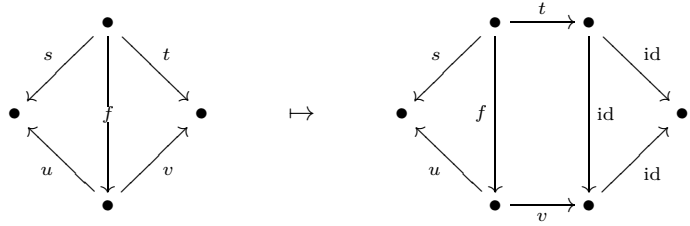
\[
\xymatrix{ & \bullet\ar[rd]^{t}\ar[ld]_{s}\ar[dd]|-{f} &  &  &  & \bullet\ar[dl]_{s}\ar[r]^{t}\ar[dd]_{f} & \bullet\ar[rd]^{\textnormal{id}}\ar[dd]^{\textnormal{id}}\\
\bullet &  & \bullet & \mapsto & \bullet &  &  & \bullet\\
 & \bullet\ar[ru]_{v}\ar[lu]^{u} &  &  &  & \bullet\ar[ul]^{u}\ar[r]_{v} & \bullet\ar[ru]_{\textnormal{id}}
}
\]
where one must note the appropriate square is a pullback since $f$
is invertible.

We will need to consider gregarious functors which restrict to pseudofunctors
on the embeddings we have just defined, and so we make the following
definition.
\begin{defn}
Let $\mathcal{E}$ be a locally cartesian closed category, let $\mathscr{C}$
be a bicategory and form the category $\mathbf{Greg}\left(\mathbf{Poly}_{c}\left(\mathcal{E}\right),\mathscr{C}\right)$.
We define $\mathbf{Greg}_{\otimes}\left(\mathbf{Poly}_{c}\left(\mathcal{E}\right),\mathscr{C}\right)$
as the subcategory of gregarious functors $F\colon\mathbf{Poly}_{c}\left(\mathcal{E}\right)\to\mathscr{C}$
such that the restriction $F_{\otimes}\colon\mathcal{E}\to\mathscr{C}$
is pseudo. Define $\mathbf{Greg}_{\Sigma\Delta,\Delta\otimes}\left(\mathbf{Poly}_{c}\left(\mathcal{E}\right),\mathscr{C}\right)$
as the subcategory of gregarious functors for which both restrictions
$F_{\Sigma\Delta}\colon\mathbf{Span}\left(\mathcal{E}\right)\to\mathscr{C}$
and $F_{\Delta\otimes}\colon\mathbf{Span}_{\textnormal{iso}}\left(\mathcal{E}\right)\to\mathscr{C}$
are pseudo.
\end{defn}

\begin{rem}
Note that a gregarious functor $F\colon\mathbf{Poly}_{c}\left(\mathcal{E}\right)\to\mathscr{C}$
automatically restricts to pseudofunctors $F_{\Sigma}$ and $F_{\Delta}$.
This is why we have omitted these conditions. Also note that oplax
constraints of the form 
\[
F\left(\vcenter{\hbox{\xymatrix@=1em{ & \bullet\ar[ld]_{s}\ar[r]^{t} & \bullet\ar[rd]^{\textnormal{id}}\\
\bullet &  &  & \bullet
}
}}\right)\to F\left(\vcenter{\hbox{\xymatrix@=1em{ & \bullet\ar[ld]_{s}\ar[r]^{\textnormal{id}} & \bullet\ar[rd]^{\textnormal{id}}\\
\bullet &  &  & \bullet
}
}}\right);F\left(\vcenter{\hbox{\xymatrix@=1em{ & \bullet\ar[ld]_{\textnormal{id}}\ar[r]^{t} & \bullet\ar[rd]^{\textnormal{id}}\\
\bullet &  &  & \bullet
}
}}\right)
\]
are automatically invertible by gregariousness.
\end{rem}

\begin{defn}
\label{xidef} Given a category $\mathcal{E}$ with pullbacks, we
define 
\[
\mathbf{Greg}\left(\mathbf{Span}\left(\mathcal{E}\right),\mathscr{C}\right)\times_{\textnormal{\ensuremath{\Delta}}}\mathbf{Greg}_{\otimes,\Delta}\left(\mathbf{Span}_{\textnormal{iso}}\left(\mathcal{E}\right),\mathscr{C}\right)
\]
 to be the full subcategory of 
\[
\mathbf{Greg}\left(\mathbf{Span}\left(\mathcal{E}\right),\mathscr{C}\right)\times\mathbf{Greg}_{\otimes,\Delta}\left(\mathbf{Span}_{\textnormal{iso}}\left(\mathcal{E}\right),\mathscr{C}\right)
\]
consisting of pairs $H\colon\mathbf{Span}\left(\mathcal{E}\right)\to\mathscr{C}$
and $K\colon\mathbf{Span}_{\textnormal{iso}}\left(\mathcal{E}\right)\to\mathscr{C}$
which coincide on objects and on spans of the form 
\[
\xymatrix{Y & X\ar[r]^{1_{X}}\ar[l]_{f} & X}
\]
Noting that this forces $H_{\Sigma}f\dashv H_{\Delta}f=K_{\Delta}f$
for all morphisms $f$ in $\mathcal{E}$, we denote by $\Xi$ the
assignment of such a $H$ and $K$ to the lax Beck triple
\[
H_{\Sigma}\colon\mathcal{E}\to\mathscr{C},\qquad K_{\Delta}\colon\mathcal{E}^{\textnormal{op}}\to\mathscr{C},\qquad K_{\otimes}\colon\mathcal{E}\to\mathscr{C}
\]
with $K_{\Delta}$-$K_{\otimes}$ Beck data $^{K}\mathfrak{b}$ given
by Proposition \ref{spanisoemb}.
\end{defn}

We now have enough to state the universal property of polynomials.
\begin{thm}
[Universal Properties of Polynomials: Cartesian Setting] \label{unipolycartthm}
Given a locally cartesian closed category $\mathcal{E}$ with chosen
pullbacks and distributivity pullbacks, denote by $\varUpsilon$ the
composite operation
\[
\xymatrix@=0.8em{\mathbf{Greg}_{\otimes}\left(\mathbf{Poly}_{c}\left(\mathcal{E}\right),\mathscr{C}\right)\ar[d]\\
\mathbf{Greg}\left(\mathbf{Span}\left(\mathcal{E}\right),\mathscr{C}\right)\times_{\Delta}\mathbf{Greg}_{\otimes,\Delta}\left(\mathbf{Span}_{\textnormal{iso}}\left(\mathcal{E}\right),\mathscr{C}\right)\ar[d]\\
\mathbf{Lax}\mathbf{Beck}\mathbf{Triple}\left(\mathcal{E},\mathscr{C}\right)
}
\]
where the first operation is composition with the embeddings $\left(-\right)_{\Sigma\Delta}$
and $\left(-\right)_{\Delta\otimes}$, and the second operation is
$\Xi$ from Definition \ref{xidef}. Then $\varUpsilon$ defines the
equivalence of categories
\[
\mathbf{Greg}_{\otimes}\left(\mathbf{Poly}_{c}\left(\mathcal{E}\right),\mathscr{C}\right)\simeq\mathbf{Lax}\mathbf{Beck}\mathbf{Triple}\left(\mathcal{E},\mathscr{C}\right)
\]
which restricts to the equivalence
\[
\mathbf{Greg}_{\Sigma\Delta,\Delta\otimes}\left(\mathbf{Poly}_{c}\left(\mathcal{E}\right),\mathscr{C}\right)\simeq\mathbf{Beck}\mathbf{Triple}\left(\mathcal{E},\mathscr{C}\right)
\]
and further restricts to the equivalence
\[
\mathbf{Icon}\left(\mathbf{Poly}_{c}\left(\mathcal{E}\right),\mathscr{C}\right)\simeq\mathbf{Dist}\mathbf{Beck}\mathbf{Triple}\left(\mathcal{E},\mathscr{C}\right)
\]
for any bicategory $\mathscr{C}$.
\end{thm}

\begin{rem}
There are five other equivalences of categories since each of the
three independent conditions $\Sigma\Delta$, $\Delta\otimes$ and
$\Sigma\otimes$ of Definition \ref{distcondition} may or may not
be enforced (giving a total of eight conditions). However, as the
three above appear to be the most useful, we will not mention the
others.
\end{rem}

\subsection{Proving the universal property}

Before proving Theorem \ref{unipolycartthm} we will need to show
that given a lax Beck triple $\mathcal{E}\to\mathscr{C}$ one may
reconstruct an oplax functor $\mathbf{Poly}_{c}\left(\mathcal{E}\right)\to\mathscr{C}$.
The following lemma and subsequent propositions describe this construction.
Also note that we are keeping the proof as similar as possible to
the case of spans, starting with the below lemma which is the analogue
of Lemma \ref{spanlocally}.
\begin{lem}
\label{polyclocally} Let $\mathcal{E}$ be a locally cartesian closed
category seen as a locally discrete 2-category, and let $\mathscr{C}$
be a bicategory. Suppose we are given a lax Beck triple consisting
of pseudofunctors
\[
F_{\Sigma}\colon\mathcal{E}\to\mathscr{C},\qquad F_{\Delta}\colon\mathcal{E}^{\textnormal{op}}\to\mathscr{C},\qquad F_{\otimes}\colon\mathcal{E}\to\mathscr{C}
\]
and Beck 2-cells $\mathfrak{b}$. We may then define local functors
\[
L_{X,Y}:\mathbf{Poly}_{c}\left(\mathcal{E}\right)_{X,Y}\to\mathscr{C}_{LX,LY},\qquad X,Y\in\mathcal{E}
\]
by the assignment $T\to FT$ on objects, and
\[
\xymatrix@=1.5em{ & E\ar[dl]_{s}\ar[r]^{p}\ar[dd]_{f}\ar@{}[rdd]|-{\textnormal{pb}} & B\ar[dr]^{t}\ar[dd]^{g} &  &  &  & E\ar[r]^{F_{\otimes}p}\ar@{}[ddr]|-{\Downarrow\mathfrak{b}} & B\ar[dr]^{F_{\Sigma}t}\\
X &  &  & Y & \mapsto & X\ar[ur]^{F_{\Delta}s}\ar[dr]_{F_{\Delta}u} & \;\ar@{}[l]|-{\Downarrow\alpha} & \;\ar@{}[r]|-{\Downarrow\gamma} & Y\\
 & M\ar[r]_{q}\ar[ul]^{u} & N\ar[ur]_{v} &  &  &  & M\ar[r]_{F_{\otimes}q}\ar[uu]|-{F_{\Delta}f} & N\ar[ur]_{F_{\Sigma}v}\ar[uu]|-{F_{\Delta}g}
}
\]
on morphisms, where $\alpha$ is the mate of the isomorphism on the
left below
\[
\xymatrix{FE\ar[r]^{1_{FE}}\ar[d]_{F_{\Sigma}s}\ar@{}[dr]|-{\cong} & FE\ar[d]^{F_{\Sigma}u\cdot F_{\Sigma}f} &  &  & FB\ar[r]^{F_{\Sigma}t}\ar[d]_{F_{\Sigma}g}\ar@{}[dr]|-{\cong} & FY\ar[d]^{1_{FY}}\\
FX\ar[r]_{1_{FX}} & FX &  &  & FN\ar[r]_{F_{\Sigma}v} & FY
}
\]
under the adjunctions $F_{\Sigma}s\dashv F_{\Delta}s$ and $F_{\Sigma}u\cdot F_{\Sigma}f\dashv F_{\Delta}f\cdot F_{\Delta}u$,
$\gamma$ is the mate of the isomorphism on the right above under
the adjunctions $F_{\Sigma}g\dashv F_{\Delta}g$ and $\textnormal{1}_{FY}\dashv\textnormal{1}_{FY}$,
and $\mathfrak{b}_{q,g}^{p,f}$ (simply denoted $\mathfrak{b}$ for
convenience) is the component of the Beck data at the given pullback.
\end{lem}

\begin{proof}
The local functor $L_{X,Y}$ sends the components of the composite
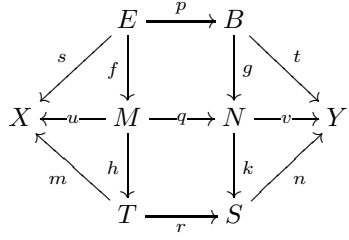
\[
\xymatrix{ & E\ar[dl]_{s}\ar[r]^{p}\ar[d]_{f} & B\ar[dr]^{t}\ar[d]^{g}\\
X & M\ar[r]|-{q}\ar[l]|-{u}\ar[d]_{h} & N\ar[r]|-{v}\ar[d]^{k} & Y\\
 & T\ar[r]_{r}\ar[ul]^{m} & S\ar[ru]_{n}
}
\]
to the top and bottom halves of the pasting diagram below:
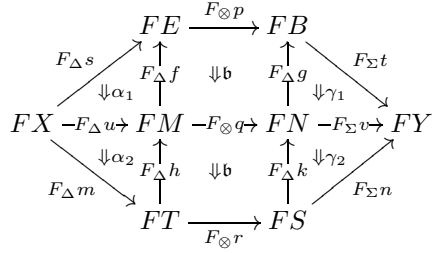
\[
\xymatrix{ & FE\ar[r]^{F_{\otimes}p}\ar@{}[dr]|-{\Downarrow\mathfrak{b}} & FB\ar[rd]_{\Downarrow\gamma_{1}}^{F_{\Sigma}t}\\
FX\ar[ur]_{\Downarrow\alpha_{1}}^{F_{\Delta}s}\ar[r]|-{F_{\Delta}u}\ar[rd]_{F_{\Delta}m}^{\Downarrow\alpha_{2}} & FM\ar[u]|-{F_{\Delta}f}\ar[r]|-{F_{\otimes}q}\ar@{}[dr]|-{\Downarrow\mathfrak{b}} & FN\ar[u]|-{F_{\Delta}g}\ar[r]|-{F_{\Sigma}v} & FY\\
 & FT\ar[r]_{F_{\otimes}r}\ar[u]|-{F_{\Delta}h} & FS\ar[u]|-{F_{\Delta}k}\ar[ru]_{F_{\Sigma}n}^{\Downarrow\gamma_{2}}
}
\]
To see this is functorial, we insert an unlabeled constraint of $F_{\Delta}$
and its inverse on both the left and right side of the above diagram,
giving the pasting below
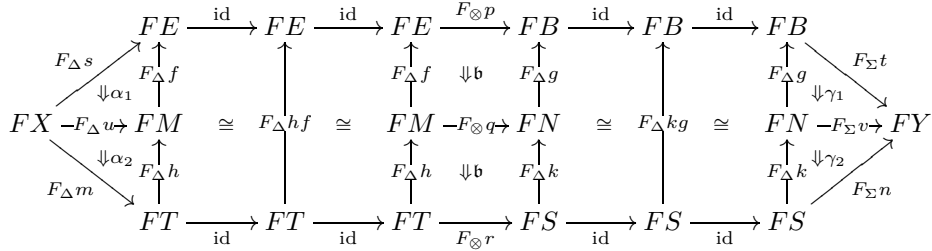
\[
\xymatrix{ & FE\ar[r]^{\textnormal{id}} & FE\ar[r]^{\textnormal{id}} & FE\ar[r]^{F_{\otimes}p}\ar@{}[dr]|-{\Downarrow\mathfrak{b}} & FB\ar[r]^{\textnormal{id}} & FB\ar[r]^{\textnormal{id}} & FB\ar[rd]_{\Downarrow\gamma_{1}}^{F_{\Sigma}t}\\
FX\ar[ur]_{\Downarrow\alpha_{1}}^{F_{\Delta}s}\ar[r]|-{F_{\Delta}u}\ar[rd]_{F_{\Delta}m}^{\Downarrow\alpha_{2}} & FM\ar[u]|-{F_{\Delta}f}\ar@{}[r]|-{\cong\;\;} & \;\ar@{}[r]|-{\;\;\cong} & FM\ar[u]|-{F_{\Delta}f}\ar[r]|-{F_{\otimes}q}\ar@{}[dr]|-{\Downarrow\mathfrak{b}} & FN\ar[u]|-{F_{\Delta}g}\ar@{}[r]|-{\cong\;\;} & \;\ar@{}[r]|-{\;\;\cong} & FN\ar[u]|-{F_{\Delta}g}\ar[r]|-{F_{\Sigma}v} & FY\\
 & FT\ar[r]_{\textnormal{id}}\ar[u]|-{F_{\Delta}h} & FT\ar[r]_{\textnormal{id}}\ar[uu]|-{F_{\Delta}hf} & FT\ar[r]_{F_{\otimes}r}\ar[u]|-{F_{\Delta}h} & FS\ar[u]|-{F_{\Delta}k}\ar[r]_{\textnormal{id}} & FS\ar[uu]|-{F_{\Delta}kg}\ar[r]_{\textnormal{id}} & FS\ar[u]|-{F_{\Delta}k}\ar[ru]_{F_{\Sigma}n}^{\Downarrow\gamma_{2}}
}
\]
and then apply the vertical double pullback condition on Beck data
and use functoriality of mates. This shows that the above diagram
is $L_{X,Y}$ applied to the composite. That the identity maps are
preserved is similar to the case of spans, but using the vertical
nullary pullback condition on Beck 2-cells $\mathfrak{b}$.
\end{proof}
As in the case of spans, it will be helpful to recall the reduced
description of an oplax structure on local functors out of the bicategory
of polynomials.
\begin{prop}
\cite{WalkerGeneric}\label{polycoh} Let $\mathcal{E}$ be a locally
cartesian closed category and denote by $\mathbf{Poly}_{c}\left(\mathcal{E}\right)$
the bicategory of polynomials in $\mathcal{E}$ with cartesian 2-cells.
Let $\mathscr{C}$ be a bicategory. Then to give an oplax functor
\[
L\colon\mathbf{Poly}_{c}\left(\mathcal{E}\right)\to\mathscr{C}
\]
 is to give a locally defined functor
\[
L_{X,Y}\colon\mathbf{Poly}_{c}\left(\mathcal{E}\right)_{X,Y}\to\mathscr{C}_{LX,LY},\qquad X,Y\in\mathcal{E}
\]
with comultiplication and counit maps 
\[
\Phi_{s,p_{1},h,p_{2},t}\colon L\left(s,p,t\right)\to L\left(s,p_{1},h\right);L\left(h,p_{2},t\right),\qquad\Lambda_{h}\colon L\left(h,1,h\right)\to1_{LX}
\]
for every respective diagram in $\mathcal{E}$ 
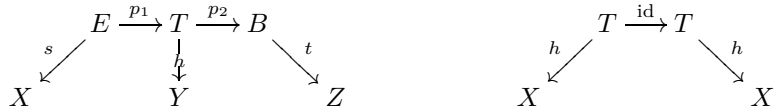
\[
\xymatrix@=1.5em{ &  & E\ar[ld]_{s}\ar[r]^{p_{1}} & T\ar[d]|-{h}\ar[r]^{p_{2}} & B\ar[rd]^{t} &  &  &  &  & T\ar[ld]_{h}\ar[r]^{\textnormal{id}} & T\ar[rd]^{h}\\
 & X &  & Y &  & Z &  &  & X &  &  & X
}
\]
where we assert $p=p_{1};p_{2}$ on the left, such that:
\begin{enumerate}
\item for any morphisms of polynomials as below
\[
\xymatrix@=1.5em{ & R\ar[dl]_{u}\ar[dd]|-{f}\ar[r]^{q} & I\ar[dr]^{v}\ar[dd]|-{g} &  &  & R\ar[dl]_{u}\ar[dd]|-{f}\ar[r]^{q_{1}} & S\ar[dr]^{k}\ar[dd]|-{c} &  &  & S\ar[dl]_{k}\ar[dd]|-{c}\ar[r]^{q_{2}} & I\ar[dr]^{v}\ar[dd]|-{g}\\
X &  &  & Z & X &  &  & Y & Y &  &  & Z\\
 & E\ar[ul]^{s}\ar[r]_{p} & B\ar[ru]_{t} &  &  & E\ar[ul]^{s}\ar[r]_{p_{1}} & T\ar[ru]_{h} &  &  & T\ar[ul]^{h}\ar[r]_{p_{2}} & B\ar[ru]_{t}
}
\]
we have the commuting diagram
\[
\xymatrix{L\left(u,q,v\right)\ar[d]_{L\left(f,g\right)}\myar{\Phi_{u,q_{1},k,q_{2},v}}{rr} &  & L\left(u,q_{1},k\right);L\left(k,q_{2},v\right)\ar[d]^{L\left(f,c\right);L\left(c,g\right)}\\
L\left(s,p,t\right)\myard{\Phi_{s,p_{1},h,p_{2},t}}{rr} &  & L\left(s,p_{1},h\right);L\left(h,p_{2},t\right)
}
\]
\item for any morphism of polynomials as on the left below
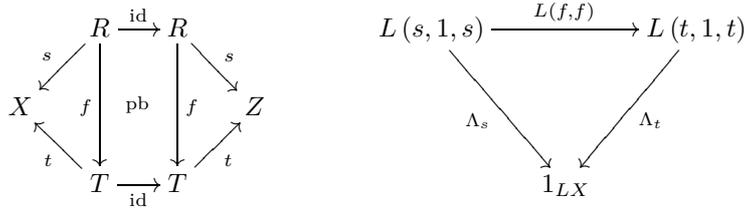
\[
\xymatrix@=1.5em{ & R\ar[dl]_{s}\ar[dd]_{f}\ar[r]^{\textnormal{id}}\ar@{}[rdd]|-{\textnormal{pb}} & R\ar[dr]^{s}\ar[dd]^{f} &  &  & L\left(s,1,s\right)\ar[rr]^{L\left(f,f\right)}\ar[rdd]_{\Lambda_{s}} &  & L\left(t,1,t\right)\ar[ldd]^{\Lambda_{t}}\\
X &  &  & Z\\
 & T\ar[ul]^{t}\ar[r]_{\textnormal{id}} & T\ar[ru]_{t} &  &  &  & 1_{LX}
}
\]
the diagram on the right above commutes;
\item for all diagrams of the form
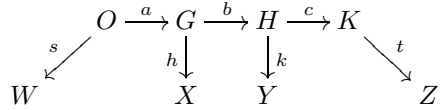
\[
\xymatrix@=1.5em{ & O\ar[dl]_{s}\ar[r]^{a} & G\ar[r]^{b}\ar[d]_{h} & H\ar[r]^{c}\ar[d]^{k} & K\ar[rd]^{t}\\
W &  & X & Y &  & Z
}
\]
in $\mathcal{E}$, we have the commuting diagram
\[
\xymatrix{L\left(s,a;b;c,t\right)\ar[d]_{\Phi_{s,a,h,b;c,t}}\ar@{=}[rr] &  & L\left(s,a;b;c,t\right)\ar[d]^{\Phi_{s,a;b,k,c,t}}\\
L\left(s,a,h\right);L\left(h,b;c,t\right)\myard{L\left(s,a,h\right);\Phi_{h,b,k,c,t}}{d} &  & L\left(s,a;b,k\right);L\left(k,c,t\right)\ar[d]^{\Phi_{s,a,h,b,k};L\left(k,c,t\right)}\\
L\left(s,a,h\right);\left(L\left(h,b,k\right);L\left(k,c,t\right)\right)\myard{\textnormal{assoc}}{rr} &  & \left(L\left(s,a,h\right);L\left(h,b,k\right)\right);L\left(k,c,t\right)
}
\]
\item for all polynomials $\left(s,p,t\right)$ the diagrams
\[
\xymatrix@=0.5em{ & L\left(s,1,s\right);L\left(s,p,t\right)\ar[rdd]^{\Lambda_{s};L\left(s,p,t\right)}\\
\\
L\left(s,p,t\right)\myard{\textnormal{unitor}}{rr}\ar[uur]^{\Phi_{s,1,s,p,t}} &  & 1_{LX};L\left(s,p,t\right)\\
\\
 & L\left(s,p,t\right);L\left(t,1,t\right)\ar[rdd]^{L\left(s,p,t\right);\Lambda_{t}}\\
\\
L\left(s,p,t\right)\myard{\textnormal{unitor}}{rr}\ar[uur]^{\Phi_{s,p,t,1,t}} &  & L\left(s,p,t\right);1_{LY}
}
\]
commute.
\end{enumerate}
\end{prop}

We now prove that the locally defined functor $L$ above may be endowed
with an oplax structure.
\begin{lem}
\label{polycartLoplax} Let $\mathcal{E}$ be a locally cartesian
closed category seen as a locally discrete 2-category, and let $\mathscr{C}$
be a bicategory. Suppose we are given a lax Beck triple 
\[
F_{\Sigma}\colon\mathcal{E}\to\mathscr{C},\qquad F_{\Delta}\colon\mathcal{E}^{\textnormal{op}}\to\mathscr{C},\qquad F_{\otimes}\colon\mathcal{E}\to\mathscr{C}
\]
with Beck 2-cells $\mathfrak{b}$. Then the locally defined functor
\[
L_{X,Y}\colon\mathbf{Poly}_{c}\left(\mathcal{E}\right)_{X,Y}\to\mathscr{C}_{LX,LY},\qquad X,Y\in\mathcal{E}
\]
as in Lemma \ref{polyclocally} canonically admits the structure of
an oplax functor.
\end{lem}

\begin{proof}
By Proposition \ref{polycoh}, to equip the locally defined functor
$L$ with an oplax structure is to give comultiplication maps $\Phi_{s,p_{1},h,p_{2},t}\colon L\left(s,p,t\right)\to L\left(s,p_{1},h\right);L\left(h,p_{2},t\right)$
and counit maps $\Lambda_{h}\colon L\left(h,1,h\right)\to1_{LX}$
for all diagrams of the respective forms, where $p=p_{2}p_{1}$,
\[
\xymatrix@=1.5em{ & E\ar[ld]_{s}\ar[r]^{p_{1}} & T\ar[d]|-{h}\ar[r]^{p_{2}} & B\ar[rd]^{t} &  &  &  &  & T\ar[ld]_{h}\ar[r]^{\textnormal{id}} & T\ar[rd]^{h}\\
X &  & Y &  & Z &  &  & X &  &  & X
}
\]
satisfying naturality, associativity, and unitary conditions. To do
this, we take each $\Phi_{s,p_{1},h,p_{2},t}$ to be the pasting
\[
\xymatrix{ &  &  & \mbox{\;}\\
 &  &  & \;\ar@{}[u]|-{\cong}\\
FX\ar[r]^{F_{\Delta}s} & FE\ar[r]^{F_{\otimes}p_{1}}\ar@/^{4.8pc}/[rrrr]^{F_{\otimes}p} & FT\ar[r]^{F_{\Sigma}h}\ar@/^{2.9pc}/[rr]|-{1_{FT}} & FY\ar[r]^{F_{\Delta}h}\ar@{}[u]|-{\Downarrow\eta_{Fh}} & FT\ar[r]^{F_{\otimes}p_{2}} & FB\ar[r]^{F_{\Sigma}t} & FZ
}
\]
and each $\Lambda_{h}$ to be the pasting
\[
\xymatrix{ & FT\ar@/^{0.8pc}/[r]^{F_{\otimes}1_{T}}\ar@/_{0.8pc}/[r]|-{1_{FT}}\ar@{}[r]|-{\cong} & FT\ar[rd]^{F_{\Sigma}h} & \;\\
FX\ar[rrr]|-{1_{FX}}\ar[ur]^{F_{\Delta}h} & \ar@{}[ur]|-{\overset{\;}{\Downarrow\epsilon_{Fh}}} &  & FX
}
\]
Associativity of comultiplication is almost trivial; indeed, given
a diagram of the form
\[
\xymatrix@=1.5em{ & O\ar[dl]_{s}\ar[r]^{a} & G\ar[r]^{b}\ar[d]_{h} & H\ar[r]^{c}\ar[d]^{k} & K\ar[rd]^{t}\\
W &  & X & Y &  & Z
}
\]
both paths in the associativity of comultiplication condition compose
to
\[
\xymatrix@=1.5em{\\
 &  &  & \; & \ar@{}[r]|-{\cong} & \; & \;\\
FW\ar[r]^{F_{\Delta}s} & FO\ar[r]^{F_{\otimes}a}\ar@/^{4pc}/[rrrrrrr]^{F_{\otimes}cba} & FG\ar[r]^{F_{\Sigma}h}\ar@/^{2pc}/[rr]|-{1_{FG}} & FX\ar[r]^{F_{\Delta}h}\ar@{}[u]|-{\Downarrow\eta_{Fh}} & FG\ar[r]^{F_{\otimes}b} & FH\ar[r]^{F_{\Sigma}k}\ar@/^{2pc}/[rr]|-{1_{FH}} & FY\ar[r]^{F_{\Delta}k}\ar@{}[u]|-{\Downarrow\eta_{Fk}} & FH\ar[r]^{F_{\otimes}c} & FK\ar[r]^{F_{\Sigma}t} & FZ
}
\]
by associativity of the constraints of $F_{\otimes}$. The unitary
axioms are also almost trivial, a consequence of the triangle identities
for an adjunction and the unitary axioms on $F_{\otimes}$.

For the naturality condition, suppose we are given a triple of cartesian
morphisms of polynomials
\[
\xymatrix@=1.5em{ & R\ar[dl]_{u}\ar[dd]|-{f}\ar[r]^{q} & I\ar[dr]^{v}\ar[dd]|-{g} &  &  & R\ar[dl]_{u}\ar[dd]|-{f}\ar[r]^{q_{1}} & S\ar[dr]^{k}\ar[dd]|-{c} &  &  & S\ar[dl]_{k}\ar[dd]|-{c}\ar[r]^{q_{2}} & I\ar[dr]^{v}\ar[dd]|-{g}\\
X &  &  & Z & X &  &  & Y & Y &  &  & Z\\
 & E\ar[ul]^{s}\ar[r]_{p} & B\ar[ru]_{t} &  &  & E\ar[ul]^{s}\ar[r]_{p_{1}} & T\ar[ru]_{h} &  &  & T\ar[ul]^{h}\ar[r]_{p_{2}} & B\ar[ru]_{t}
}
\]
and consider the diagram
\[
\xymatrix{L\left(u,q,v\right)\ar[d]_{L\left(f,g\right)}\myar{\Phi_{u,q_{1},k,q_{2},v}}{rr} &  & L\left(u,q_{1},k\right);L\left(k,q_{2},v\right)\ar[d]^{L\left(f,c\right);L\left(c,g\right)}\\
L\left(s,p,t\right)\myard{\Phi_{s,p_{1},h,p_{2},t}}{rr} &  & L\left(s,p_{1},h\right);L\left(h,p_{2},t\right)
}
\]
Now the top composite is
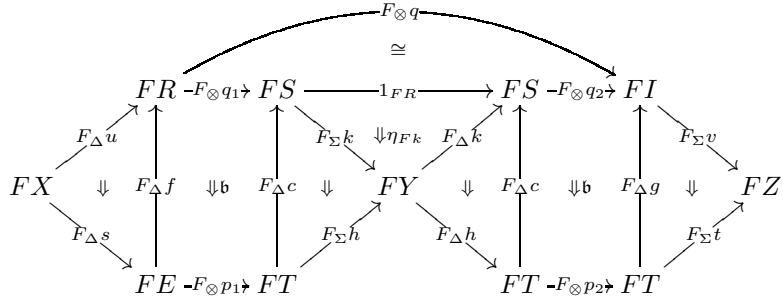
\[
\xymatrix{ &  &  & \;\ar@{}[d]|-{\cong}\\
 & FR\ar[r]|-{F_{\otimes}q_{1}}\ar@/^{2.5pc}/[rrrr]|-{F_{\otimes}q} & FS\ar[rr]|-{1_{FR}}\ar[rd]|-{F_{\Sigma}k} & \;\ar@{}[d]|-{\Downarrow\eta_{Fk}} & FS\ar[r]|-{F_{\otimes}q_{2}} & FI\ar[rd]|-{F_{\Sigma}v}\\
FX\ar[ur]|-{F_{\Delta}u}\ar[rd]|-{F_{\Delta}s}\ar@{}[r]|-{\Downarrow} & \;\ar@{}[r]|-{\Downarrow\mathfrak{b}} & \ar@{}[r]|-{\Downarrow} & FY\ar[rd]|-{F_{\Delta}h}\ar@{}[r]|-{\Downarrow}\ar[ur]|-{F_{\Delta}k} & \;\ar@{}[r]|-{\Downarrow\mathfrak{b}} & \;\ar@{}[r]|-{\Downarrow} & FZ\\
 & FE\ar[r]|-{F_{\otimes}p_{1}}\ar[uu]|-{F_{\Delta}f} & FT\ar[ur]|-{F_{\Sigma}h}\ar[uu]|-{F_{\Delta}c} &  & FT\ar[uu]|-{F_{\Delta}c}\ar[r]|-{F_{\otimes}p_{2}} & FT\ar[ur]|-{F_{\Sigma}t}\ar[uu]|-{F_{\Delta}g}
}
\]
where the unlabeled 2-cells are as in Lemma \ref{polyclocally}, and
one may rewrite the pasting of the three middle triangles above as
an ``identity square'' and pasting with $\eta_{Fh}$. It follows
that this is equal to the bottom composite given by the pasting 

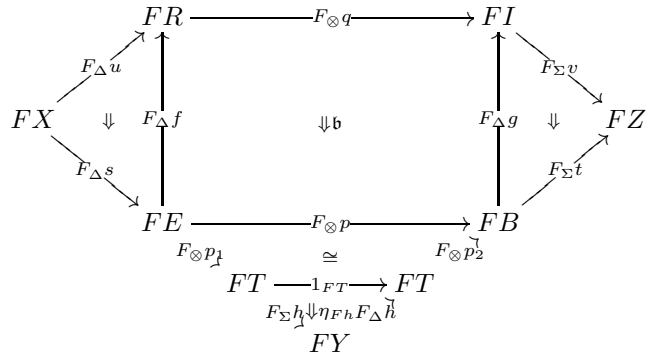
\[
\xymatrix@=1em{ &  & FR\ar[rrrr]|-{F_{\otimes}q} &  & \ar@{}[dddd]|-{\Downarrow\mathfrak{b}} &  & FI\ar[drdr]|-{F_{\Sigma}v}\\
\\
FX\ar[rdrd]|-{F_{\Delta}s}\ar[ruru]|-{F_{\Delta}u}\ar@{}[rr]|-{\Downarrow} &  &  &  &  &  & \;\ar@{}[rr]|-{\Downarrow} &  & FZ\\
\\
 &  & FE\ar[rd]|-{F_{\otimes}p_{1}}\ar[rrrr]|-{F_{\otimes}p}\ar[uuuu]|-{F_{\Delta}f} &  & \;\ar@{}[d]|-{\cong} &  & FB\ar[uuuu]|-{F_{\Delta}g}\ar[uurr]|-{F_{\Sigma}t}\\
 &  &  & FT\ar[rr]|-{1_{FT}}\ar[dr]|-{F_{\Sigma}h} & \ar@{}[d]|-{\Downarrow\eta_{Fh}} & FT\ar[ur]|-{F_{\otimes}p_{2}}\\
 &  &  &  & FY\ar[ur]|-{F_{\Delta}h}
}
\]
using the horizontal binary axiom on elements of $\mathfrak{b}$;
thus showing naturality of comultiplication. Naturality of counits
is similar to the case of spans (except that one must use the horizontal
nullary axiom on elements of $\mathfrak{b}$) and so will be omitted.
\end{proof}
It will be useful to have a description of the oplax constraint cells
$\varphi$ corresponding to our comultiplication maps $\Phi$. This
is described by the following lemma.
\begin{lem}
\label{oplaxBeckdist} Let the oplax functor $L\colon\mathbf{Poly}_{c}\left(\mathcal{E}\right)\to\mathscr{C}$
be constructed as in Proposition \ref{polycartLoplax}. Then the binary
oplax constraint cell on $L$ at a composite of polynomials constructed
as below
\begin{equation}
\xymatrix@=1.5em{ &  & H\ar[r]^{p_{1}}\ar[ld]_{w}\ar@{}[d]|-{\textnormal{pb}} & M\ar[ld]^{x}\ar[rd]_{y}\ar[r]^{p_{2}} & K\ar[rd]^{z}\ar@{}[d]|-{\textnormal{pb}}\\
 & A\ar[ld]_{a}\ar[r]^{m} & B\ar[rd]^{b} &  & C\ar[ld]_{c}\ar[r]^{n} & D\ar[rd]^{d}\\
X &  &  & Y &  &  & Z
}
\label{polycomp}
\end{equation}
is given by the pasting
\begin{equation}
\xymatrix{ &  &  & \;\ar@{}[d]|-{\cong}\\
 & FH\ar[r]|-{F_{\otimes}p_{1}}\ar@/^{2.5pc}/[rrrr]|-{F_{\otimes}p} & FM\ar[rr]|-{1_{FM}}\ar[rd]|-{F_{\Sigma}h} & \;\ar@{}[d]|-{\Downarrow\eta_{Fh}} & FM\ar[r]|-{F_{\otimes}p_{2}} & FK\ar[rd]|-{F_{\Sigma}dz}\\
FX\ar[ur]|-{F_{\Delta}aw}\ar[rd]|-{F_{\Delta}a}\ar@{}[r]|-{\Downarrow} & \;\ar@{}[r]|-{\Downarrow\mathfrak{b}} & \ar@{}[r]|-{\Downarrow} & FY\ar[rd]|-{F_{\Delta}c}\ar@{}[r]|-{\Downarrow}\ar[ur]|-{F_{\Delta}h} & \;\ar@{}[r]|-{\Downarrow\mathfrak{b}} & \;\ar@{}[r]|-{\Downarrow} & FZ\\
 & FA\ar[r]|-{F_{\otimes}m}\ar[uu]|-{F_{\Delta}w} & FB\ar[ur]|-{F_{\Sigma}b}\ar[uu]|-{F_{\Delta}x} &  & FC\ar[uu]|-{F_{\Delta}y}\ar[r]|-{F_{\otimes}n} & FD\ar[uu]|-{F_{\Delta}z}\ar[ur]|-{F_{\Sigma}d}
}
\label{polycompconstraint}
\end{equation}
where $p=p_{2}p_{1}$ and $h=bx=cy$.
\end{lem}

\begin{proof}
Given composable polynomials $\left(a,m,b\right)$ and $\left(c,n,d\right)$
the composite is given by the terminal diagram as in \eqref{polycomp}
. We then have an induced diagonal
\[
\delta_{ac',h,db'}\colon\left(aw,p,dz\right)\to\left(aw,p_{1},h\right);\left(h,p_{2},dz\right)
\]
and morphisms $\left(w,x\right)\colon\left(aw,p_{1},h\right)\to\left(a,m,b\right)$
and $\left(y,z\right)\colon\left(h,p_{2},dz\right)\to\left(c,n,d\right)$
for which
\[
\xymatrix{\left(a,m,b\right);\left(c,n,d\right)\myar{\delta_{ac',h,db'}}{r} & \left(aw,p_{1},h\right);\left(h,p_{2},dz\right)\myar{\left(w,x\right);\left(y,z\right)}{r} & \left(a,m,b\right);\left(c,n,d\right)}
\]
is the identity on $\left(a,m,b\right);\left(c,n,d\right)$. It follows
that the binary oplax constraint cell corresponding to the comultiplication
maps $\Phi$, namely
\[
\varphi_{\left(a,m,b\right),\left(c,n,d\right)}\colon L\left(\left(a,m,b\right);\left(c,n,d\right)\right)\to L\left(a,m,b\right);L\left(c,n,d\right)
\]
is given by \eqref{polycompconstraint}.
\end{proof}
We now know enough for a complete proof of the universal properties
of the polynomials with cartesian 2-cells.
\begin{proof}
[Proof of Theorem \ref{unipolycartthm}] We consider the assignment
\[
\varUpsilon\colon\mathbf{Greg}_{\otimes}\left(\mathbf{Poly}_{c}\left(\mathcal{E}\right),\mathscr{C}\right)\to\mathbf{LaxBeckTriple}\left(\mathcal{E},\mathscr{C}\right)
\]
of Theorem \ref{unipolycartthm}, i.e. given a gregarious functor
$F\colon\mathbf{Poly}_{c}\left(\mathcal{E}\right)\to\mathscr{C}$
which restricts to a pseudofunctor when composed with $\left(-\right)_{\otimes}$,
we extract, via Theorem \ref{unispanisothm}, the pseudofunctors $F_{\Sigma}$,
$F_{\Delta}$ and $F_{\otimes}$ equipped with the Beck data $\mathfrak{b}$.
This data defines a lax Beck triple $\mathcal{E}\to\mathscr{C}$.

\emph{\noun{Well defined.}} Given an icon $\alpha\colon F\Rightarrow G\colon\mathbf{Poly}_{c}\left(\mathcal{E}\right)\to\mathscr{C}$
we know $\alpha_{\Delta}\colon F_{\Delta}\Rightarrow G_{\Delta}$
is invertible, as it is a restriction of an icon $\alpha_{\Sigma\Delta}\colon F_{\Sigma\Delta}\Rightarrow G_{\Sigma\Delta}\colon\mathbf{Span}\left(\mathcal{E}\right)\to\mathscr{C}$
which is necessarily invertible by Lemma \ref{spaniconinv}.

We start by proving the first universal property.

\noun{Fully faithful.} That the assignment $\alpha\mapsto\left(\alpha_{\Delta},\alpha_{\otimes}\right)$
is bijective follows from the fact the assignment $\alpha_{\Sigma\Delta}\mapsto\alpha_{\Delta}$
is bijective, and the necessary commutativity of 
\begin{equation}
\xymatrix{F\left(s,p,t\right)\myar{\varphi}{r}\ar[d]_{\alpha_{\left(s,p,t\right)}} & F\left(s,1,1\right);F\left(1,p,1\right);F\left(1,1,t\right)\ar[d]^{\alpha_{\left(s,1,1\right)};\alpha_{\left(1,p,1\right)};\alpha_{\left(1,1,t\right)}}\\
G\left(s,p,t\right)\myard{\psi}{r} & G\left(s,1,1\right);G\left(1,p,1\right);G\left(1,1,t\right)
}
\label{polydiagram}
\end{equation}
where $\varphi$ and $\psi$ must be invertible constraints since
$F$ and $G$ are gregarious.

Again, that $\left(\alpha_{\Sigma_{f}}^{-1}\right)^{*}=\alpha_{\Delta_{f}}$
is forced by Lemma \ref{mateinv} and one need only check that any
collection
\[
\alpha_{s,p,t}\colon F\left(s,p,t\right)\to G\left(s,p,t\right)
\]
 satisfying this property and \eqref{polydiagram} necessarily defines
an icon. 

We omit the calculation showing $\alpha$ is locally natural. Indeed,
this calculation is almost the same as in the proof of Theorem \ref{unispansthm},
except we must interchange a Beck 2-cell with the components $\alpha_{\Delta}$
and $\alpha_{\otimes}$ using the condition \eqref{polyBeckcoh}.

To see that such an $\alpha$ then defines an icon, note that each
$\Phi_{s,p_{1},h,p_{2},t}$ may be decomposed as the commuting diagram
\[
\xymatrix{F\left(s,p,t\right)\myar{\Phi_{s,p_{1},h,p_{2},t}}{rrr}\ar[d]_{\Phi_{s,p_{1},1,p_{2},t}} &  &  & F\left(s,p_{1},h\right);F\left(h,p_{2},t\right)\\
F\left(s,p_{1},1\right);F\left(1,p_{2},t\right)\myard{F\left(s,p_{1},1\right);\Phi_{1,1,h,1,1};F\left(1,p_{2},t\right)}{rrr} &  &  & F\left(s,p_{1},1\right);F\left(1,1,h\right);F\left(h,1,1\right);F\left(1,p_{2},t\right)\ar[u]_{\Phi_{s,p_{1},1,1,h}^{-1};\Phi_{h,1,1,p_{2},t}^{-1}}
}
\]
and so the commutativity of the diagram
\[
\xymatrix{F\left(s,p,t\right)\myar{\Phi_{s,p_{1},h,p_{2},t}}{rr}\ar[d]_{\alpha_{s,p,t}} &  & F\left(s,p_{1},h\right);F\left(h,p_{2},t\right)\ar[d]^{\alpha_{s,p_{1},h};\alpha_{h,p_{2},t}}\\
G\left(s,p,t\right)\myard{\Psi_{s,p_{1},h,p_{2},t}}{rr} &  & G\left(s,p_{1},h\right);G\left(h,p_{2},t\right)
}
\]
amounts to checking that the pastings
\[
\xymatrix{ &  &  & \ar@{}[d]|-{\Downarrow\varphi}\\
 &  &  & \ar@{}[d]|-{\Downarrow\eta_{Gh}}\\
\bullet\ar@/_{0.5pc}/[r]_{G_{\Delta}s}\ar@/^{0.5pc}/[r]^{F_{\Delta}s}\ar@{}[r]|-{\Downarrow} & \bullet\ar@/_{0.5pc}/[r]_{G_{\otimes}p_{1}}\ar@/^{0.5pc}/[r]^{F_{\otimes}p_{1}}\ar@{}[r]|-{\Downarrow}\ar@/^{4.5pc}/[rrrr]^{F_{\otimes}p} & \bullet\ar@/_{0.5pc}/[r]_{G_{\Sigma}h}\ar@/^{0.5pc}/[r]^{F_{\Sigma}h}\ar@{}[r]|-{\Downarrow}\ar@/^{2.5pc}/[rr]^{\textnormal{id}} & \bullet\ar@/_{0.5pc}/[r]_{G_{\Delta}h}\ar@/^{0.5pc}/[r]^{F_{\Delta}h}\ar@{}[r]|-{\Downarrow} & \bullet\ar@/_{0.5pc}/[r]_{G_{\otimes}p_{2}}\ar@/^{0.5pc}/[r]^{F_{\otimes}p_{2}}\ar@{}[r]|-{\Downarrow} & \bullet\ar@/_{0.5pc}/[r]_{G_{\Sigma}t}\ar@/^{0.5pc}/[r]^{F_{\Sigma}t}\ar@{}[r]|-{\Downarrow} & \bullet
}
\]
and
\[
\xymatrix{ &  &  & \ar@{}[d]|-{\Downarrow\psi}\ar@{}[]|-{\Downarrow\alpha_{\otimes}}\\
 &  &  & \ar@{}[d]|-{\Downarrow\eta_{Gh}}\\
\bullet\ar@/_{0.5pc}/[r]_{G_{\Delta}s}\ar@/^{0.5pc}/[r]^{F_{\Delta}s}\ar@{}[r]|-{\Downarrow} & \bullet\ar[r]_{G_{\otimes}p_{1}}\ar@/^{4.5pc}/[rrrr]|-{G_{\otimes}p}\ar@/^{6pc}/[rrrr]|-{F_{\otimes}p} & \bullet\ar@/_{0pc}/[r]_{G_{\Sigma}h}\ar@/^{2.5pc}/[rr]^{\textnormal{id}} & \bullet\ar@/_{0pc}/[r]_{G_{\Delta}h} & \bullet\ar[r]_{G_{\otimes}p_{2}} & \bullet\ar@/_{0.5pc}/[r]_{G_{\Sigma}t}\ar@/^{0.5pc}/[r]^{F_{\Sigma}t}\ar@{}[r]|-{\Downarrow} & \bullet
}
\]
agree. This is almost the same calculation as in spans except here
we must use that $\alpha_{\otimes}$ is an icon.

\noun{Essentially surjective}. Suppose we are given a lax Beck triple
$\left(F_{\Sigma},F_{\Delta},F_{\otimes},\mathfrak{b}\right)$. Then
by Proposition \ref{polycartLoplax}, we get a normal oplax functor
$F\colon\mathbf{Poly}_{c}\left(\mathcal{E}\right)\to\mathscr{C}$
(which is gregarious as a consequence Proposition \ref{polyadjprop},
and clearly restricts to a pseudofunctor on $\otimes$), and this
constructed $F$ clearly restricts to the same lax Beck triple when
$\varUpsilon$ is applied.

We now prove the remaining two universal properties, seen as restrictions
of the first.

\noun{Restrictions}. It is clear that for any $\textnormal{Greg}_{\otimes}$
functor $F\colon\mathbf{Poly}_{c}\left(\mathcal{E}\right)\to\mathscr{C}$
we may write $F\cong\widetilde{F}$ where $\widetilde{F}$ is given
by sending $F$ to its lax Beck triple and recovering a map $\widetilde{F}\colon\mathbf{Poly}_{c}\left(\mathcal{E}\right)\to\mathscr{C}$
under the above equivalence.

Also it is clear that $F$ (or equivalently $\widetilde{F}$) restricts
to pseudofunctors on $\Sigma\Delta$ and $\Delta\otimes$ precisely
when this lax Beck triple is a Beck triple. This is seen by using
the general expression for an oplax constraint cell \eqref{polycompconstraint}
on composites of polynomials $\left(s,1,t\right);\left(u,1,v\right)$
and $\left(s,t,1\right);\left(u,v,1\right)$.

Now as each oplax constraint cell may be constructed from ``Beck
composites'' as above and ``distributivity composites'' of the
form $\left(1,1,u\right);\left(1,f,1\right)$ (by the proof of \cite[Prop. 1.12]{gambinokock}),
it follows that asking $F$ be pseudo corresponds to asking that,
in addition, the oplax constraint cells for composites $\left(1,1,u\right);\left(1,f,1\right)$
be invertible. But this is precisely the $\Sigma\otimes$ distributivity
condition.
\end{proof}

\section{Universal properties of polynomials with general 2-cells\label{unipoly}}

In this section we prove the universal property of the bicategory
of polynomials with general 2-cells, denoted $\mathbf{Poly}\left(\mathcal{E}\right)$.
As this bicategory is not generic, the methods of the previous section
do not directly apply. However, as composition in $\mathbf{Poly}_{c}\left(\mathcal{E}\right)$
and $\mathbf{Poly}\left(\mathcal{E}\right)$ is the same we can still
apply some results of the previous section to help prove this universal
property.

\subsection{Stating the universal property}

The universal property of $\mathbf{Poly}\left(\mathcal{E}\right)$
ends up being simpler to state than that of $\mathbf{Poly}_{c}\left(\mathcal{E}\right)$
due to the existence of more adjunctions. To state this property we
will first require a strengthening of the notions of ``sinister''
and ``Beck'' pseudofunctors as described in Definition \ref{defSinBeck}.
For the following definition, recall that the categories of such pseudofunctors
and invertible icons are denoted $\mathbf{Sin}\left(\mathcal{E},\mathscr{C}\right)$
and $\mathbf{Beck}\left(\mathcal{E},\mathscr{C}\right)$ respectively.
\begin{defn}
Let $\mathcal{E}$ be a category with pullbacks, and let $\mathscr{C}$
be a bicategory. We denote by $\mathbf{2Sin}\left(\mathcal{E},\mathscr{C}\right)$
the subcategory of $\mathbf{Sin}\left(\mathcal{E},\mathscr{C}\right)$
consisting of pseudofunctors $F\colon\mathcal{E}\to\mathscr{C}$ for
which $Ff$ has two successive right adjoints for every morphism $f\in\mathcal{E}$.
We denote by $\mathbf{2Beck}\left(\mathcal{E},\mathscr{C}\right)$
the subcategory of $\mathbf{2Sin}\left(\mathcal{E},\mathscr{C}\right)$
consisting of those pseudofunctors which in addition satisfy the Beck
condition. 
\end{defn}

\begin{rem}
The above Beck condition is on the pair $F_{\Sigma}$-$F_{\Delta}$,
but one could also ask a Beck condition on the pair $F_{\Delta}$-$F_{\Pi}$.
The reason for not using the latter is that the Beck 2-cells (arising
from adjunctions $F_{\Delta}f\dashv F_{\Pi}f$) are not in the direction
required for constructing a lax Beck triple, and are invertible if
and only if the former Beck 2-cells are invertible.
\end{rem}

The following lemma will be needed to describe a distributivity condition
which may be imposed on such pseudofunctors.
\begin{lem}
\label{recoverdistpair} Let $\mathcal{E}$ be a locally cartesian
closed category seen as a locally discrete 2-category, and let $\mathscr{C}$
be a bicategory. Suppose $F\colon\mathcal{E}\to\mathscr{C}$ is a
given 2-Beck pseudofunctor, and for each morphism $f\in\mathcal{E}$
define $F_{\Sigma}f:=Ff$, take $F_{\Delta}f$ to be a chosen right
adjoint of $Ff$ (choosing $F_{\Delta}$ to strictly preserve identities),
and take $F_{\Pi}f$ to be a chosen right adjoint of $F_{\Delta}f$
(again choosing $F_{\Pi}$ to strictly preserve identities). We may
then define a Beck triple with underlying pseudofunctors 
\[
F_{\Sigma}\colon\mathcal{E}\to\mathscr{C},\qquad\qquad F_{\Delta}\colon\mathcal{E}^{\textnormal{op}}\to\mathscr{C},\qquad\qquad F_{\Pi}\colon\mathcal{E}\to\mathscr{C}
\]
and for each pullback as on the left below, the mate of the middle
isomorphism below whose existence is asserted by the Beck condition
\begin{equation}
\xymatrix{\bullet\ar[d]_{g'}\myar{f'}{r} & \bullet\ar[d]^{g} &  & \bullet\ar[d]_{F_{\Sigma}g'}\ar@{}[dr]|-{\cong} & \bullet\ar[l]_{F_{\Delta}f'}\ar[d]^{F_{\Sigma}g} &  & \bullet\myar{F_{\Pi}f'}{r}\ar@{}[dr]|-{\Downarrow\mathfrak{b}_{f,g}^{f',g'}} & \bullet\\
\bullet\ar[r]_{f} & \bullet &  & \bullet & \bullet\ar[l]^{F_{\Delta}f} &  & \bullet\ar[r]_{F_{\Pi}f}\ar[u]^{F_{\Delta}g'} & \bullet\ar[u]_{F_{\Delta}g}
}
\label{Beckmates}
\end{equation}
defining the Beck data as on the right above.
\end{lem}

\begin{proof}
One needs to check that the defined Beck data satisfies the necessary
coherence conditions, but this trivially follows from functoriality
of mates. Also, every component of the Beck data $\mathfrak{b}_{f,g}^{f',g'}$
defined as in the above lemma must be invertible. This is since an
isomorphism of left adjoints must correspond to an isomorphism of
right adjoints under the mates correspondence.
\end{proof}
It will be useful to give the Beck triples arising this way a name,
and so we make the following definition.
\begin{defn}
We call a Beck triple $\mathcal{E}\to\mathscr{C}$ \emph{cartesian
}if for every morphism $f\in\mathcal{E}$ there exists adjunctions
$F_{\Sigma}f\dashv F_{\Delta}f\dashv F_{\Pi}f$ and the $\Delta\Pi$
Beck data corresponds to the $\Sigma\Delta$ data via the mates correspondence
as in \eqref{Beckmates}.
\end{defn}

We may also ask that a cartesian Beck triple (or the corresponding
2-Beck functor) satisfies a distributivity condition.
\begin{defn}
Given the assumptions and data of Lemma \ref{recoverdistpair}, we
say a 2-Beck pseudofunctor $F\colon\mathcal{E}\to\mathscr{C}$ satisfies
the\emph{ distributivity condition} if the cartesian Beck triple recovered
from Lemma \ref{recoverdistpair} satisfies the distributivity condition
of Definition \ref{distcondition} (meaning this cartesian Beck triple
is a DistBeck triple).
\end{defn}

Similar to the case of $\mathbf{Poly}_{c}\left(\mathcal{E}\right)$,
we again have embeddings
\[
\left(-\right)_{\Sigma}:\mathcal{E}\to\mathbf{Poly}\left(\mathcal{E}\right),\quad\left(-\right)_{\Delta}:\mathcal{E}^{\textnormal{op}}\to\mathbf{Poly}\left(\mathcal{E}\right),\quad\left(-\right)_{\Pi}:\mathcal{E}\to\mathbf{Poly}\left(\mathcal{E}\right)
\]
The main difference here is that with these embeddings we have triples
of adjunctions $f_{\Sigma}\dashv f_{\Delta}\dashv f_{\Pi}$ for every
morphism $f\in\mathcal{E}$.

Trivially we have the inclusion $\left(-\right)_{\Sigma\Delta}\colon\mathbf{Span}\left(\mathcal{E}\right)\to\mathbf{Poly}\left(\mathcal{E}\right)$
of spans into polynomials given by the assignment 
\[
\xymatrix{ & \bullet\ar[rd]^{t}\ar[ld]_{s}\ar[dd]|-{f} &  &  &  & \bullet\ar[dl]_{s}\ar[r]^{\textnormal{id}} & \bullet\ar[rd]^{t}\ar@{=}[d]\\
\bullet &  & \bullet & \mapsto & \bullet & \bullet\ar[d]_{f}\ar[r]|-{\textnormal{id}}\ar[u]^{\textnormal{id}} & \bullet\ar[d]^{f} & \bullet\\
 & \bullet\ar[ru]_{v}\ar[lu]^{u} &  &  &  & \bullet\ar[ul]^{u}\ar[r]_{\textnormal{id}} & \bullet\ar[ur]_{v}
}
\]
The less obvious embedding $\left(-\right)_{\Delta\Pi}\colon\mathbf{Span}\left(\mathcal{E}\right)^{\textnormal{co}}\to\mathbf{Poly}\left(\mathcal{E}\right)$
is the canonical embedding of spans with reversed 2-cells into polynomials,
given by the assignment
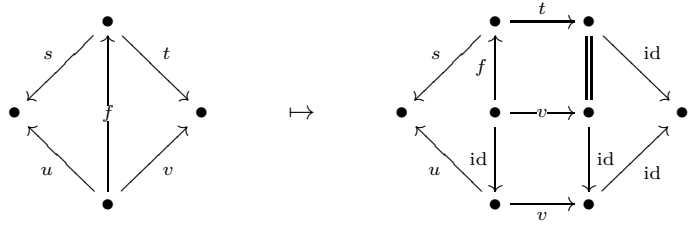
\[
\xymatrix{ & \bullet\ar[rd]^{t}\ar[ld]_{s} &  &  &  & \bullet\ar[dl]_{s}\ar[r]^{t} & \bullet\ar[rd]^{\textnormal{id}}\ar@{=}[d]\\
\bullet &  & \bullet & \mapsto & \bullet & \bullet\ar[d]_{\textnormal{id}}\ar[r]|-{v}\ar[u]^{f} & \bullet\ar[d]^{\textnormal{id}} & \bullet\\
 & \bullet\ar[ru]_{v}\ar[lu]^{u}\ar[uu]|-{f} &  &  &  & \bullet\ar[ul]^{u}\ar[r]_{v} & \bullet\ar[ru]_{\textnormal{id}}
}
\]

We now have enough to state the universal property of polynomials.
\begin{thm}
[Universal Properties of Polynomials: General Setting] \label{unipolythm}
Given a locally cartesian closed category $\mathcal{E}$ with chosen
pullbacks and distributivity pullbacks, composition with the canonical
embedding $\left(-\right)_{\Sigma}\colon\mathcal{E}\to\mathbf{Poly}\left(\mathcal{E}\right)$
defines the equivalence of categories
\[
\mathbf{Greg}\left(\mathbf{Poly}\left(\mathcal{E}\right),\mathscr{C}\right)\simeq\mathbf{2Beck}\left(\mathcal{E},\mathscr{C}\right)
\]
which restricts to the equivalence
\[
\mathbf{Icon}\left(\mathbf{Poly}\left(\mathcal{E}\right),\mathscr{C}\right)\simeq\mathbf{DistBeck}\left(\mathcal{E},\mathscr{C}\right)
\]
for any bicategory $\mathscr{C}$. 
\end{thm}

\begin{rem}
One might ask if there is a universal property without the Beck condition
being required. The problem is that if the restrictions to $\mathbf{Span}\left(\mathcal{E}\right)$
and $\mathbf{Span}\left(\mathcal{E}\right)^{\textnormal{co}}$ are
only required gregarious, but not pseudo, we do not have a canonical
way to construct the necessary $\Delta\Pi$ Beck data $\mathfrak{b}$,
and so such a universal property would be unnatural.
\end{rem}

\subsection{Proving the universal property}

Before proving Theorem \ref{unipolythm} we will need to show how
to reconstruct a gregarious functor $\mathbf{Poly}\left(\mathcal{E}\right)\to\mathscr{C}$
from a 2-Beck pseudofunctor $\mathcal{E}\to\mathscr{C}$. The following
proposition describes this construction.
\begin{prop}
\label{polyLoplax} Let $\mathcal{E}$ be a locally cartesian closed
category seen as a locally discrete 2-category, and let $\mathscr{C}$
be a bicategory. Suppose $F\colon\mathcal{E}\to\mathscr{C}$ is a
given 2-Beck pseudofunctor, and for each morphism $f\in\mathcal{E}$
define $F_{\Sigma}f:=Ff$, take $F_{\Delta}f$ to be a chosen right
adjoint of $Ff$ (choosing $F_{\Delta}$ to strictly preserve identities),
and take $F_{\Pi}f$ to be a chosen right adjoint of $F_{\Delta}f$
(again choosing $F_{\Pi}$ to strictly preserve identities). We may
then:
\begin{enumerate}
\item define a lax Beck triple as in Lemma \ref{recoverdistpair};
\item define a gregarious functor $\overline{L}\colon\mathbf{Poly}_{c}\left(\mathcal{E}\right)\to\mathscr{C}$
satisfying the $\Sigma\Delta$ and $\Delta\Pi$ Beck conditions;
\item define local functors
\[
L\colon\mathbf{Poly}\left(\mathcal{E}\right)_{X,Y}\to\mathscr{C}_{LX,LY},\qquad X,Y\in\mathcal{E}
\]
assigning each general morphism of polynomials
\[
\xymatrix{ & E\ar[dl]_{s}\ar[r]^{p} & B\ar[dr]^{t}\ar@{=}[d]\\
X & S\ar[d]_{f}\ar[r]|-{pe}\ar[u]^{e} & B\ar[d]^{g} & Y\\
 & M\ar[r]_{q}\ar[ul]^{u} & N\ar[ur]_{v}
}
\]
to the pasting
\[
\xymatrix@=1em{ &  & FE\ar[rr]^{F_{\Pi}p}\ar[dd]|-{F_{\Delta}e}\ar@{}[rddr]|-{\Downarrow\mathfrak{m}} &  & FB\ar[ddrr]^{F_{\Sigma}t}\ar@{=}[dd]\\
\\
FX\ar[urur]^{F_{\Delta}s}\ar[rrdd]_{F_{\Delta}u} & \ar@{}[]|-{\Downarrow\alpha} & FS\ar[rr]|-{F_{\Pi}pe}\ar@{}[rddr]|-{\Downarrow\mathfrak{b}} &  & FB & \ar@{}[]|-{\Downarrow\gamma\;\;} & FY\\
\\
 &  & FM\ar[uu]|-{F_{\Delta}f}\ar[rr]_{F_{\Pi}q} &  & FN\ar[uu]|-{F_{\Delta}g}\ar[rruu]_{F_{\Sigma}v}
}
\]
where of the diagrams
\[
\xymatrix{FS\ar[r]^{1_{FS}}\ar[d]|-{F_{\Sigma}s\cdot F_{\Sigma}e}\ar@{}[dr]|-{\cong} & FS\ar[d]|-{F_{\Sigma}u\cdot F_{\Sigma}f} & FE\ar[r]^{F_{\Pi}p}\ar@{}[rd]|-{\cong} & FB & FB\ar[r]^{F_{\Sigma}t}\ar[d]|-{F_{\Sigma}g}\ar@{}[dr]|-{\cong} & FY\ar[d]|-{1_{FY}}\\
FX\ar[r]_{1_{FX}} & FX & FS\ar[r]|-{F_{\Pi}pe}\ar[u]|-{F_{\Pi}e} & FB\ar[u]|-{1_{FB}} & FN\ar[r]_{F_{\Sigma}v} & FY
}
\]
\begin{enumerate}
\item $\alpha$ is constructed as the mate of the left diagram under the
adjunctions $F_{\Sigma}s\cdot F_{\Sigma}e\dashv F_{\Delta}e\cdot F_{\Delta}s$
and $F_{\Sigma}u\cdot F_{\Sigma}f\dashv F_{\Delta}f\cdot F_{\Delta}u$;
\item $\mathfrak{m}$ is constructed as the mate of the middle diagram under
the adjunctions $F_{\Delta}e\dashv F_{\Pi}e$ and $1_{FB}\dashv1_{FB}$;
\item $\mathfrak{b}$ is the component of the Beck data at the given pullback;
\item $\gamma$ is the mate of the isomorphism on the right above under
the adjunctions $F_{\Sigma}g\dashv F_{\Delta}g$ and $\textnormal{1}_{FY}\dashv\textnormal{1}_{FY}$.
\end{enumerate}
\item define a gregarious functor $L\colon\mathbf{Poly}\left(\mathcal{E}\right)\to\mathscr{C}$.
\end{enumerate}
\end{prop}

\begin{proof}
We prove the different parts of the statement separately.

\noun{Part 1. }See\noun{ }Lemma \ref{recoverdistpair}.

\noun{Part 2. }It then follows from Theorem \ref{unipolycartthm}
that this cartesian Beck triple gives rise to a gregarious functor
$\overline{L}\colon\mathbf{Poly}_{c}\left(\mathcal{E}\right)\to\mathscr{C}$.
The $\Sigma\Delta$ invertibility condition translates to an $\Delta\Pi$
invertibility condition via the mates correspondence; an isomorphism
of left adjoints must correspond to an isomorphism of right adjoints.
Therefore each component of the Beck data $\mathfrak{b}$ is invertible.

\noun{Part 3.} The goal here is to show that we have local functors
\[
L\colon\mathbf{Poly}\left(\mathcal{E}\right)_{X,Y}\to\mathscr{C}_{LX,LY},\qquad X,Y\in\mathcal{E}
\]
We first note, for well definedness, that given two general morphisms
of polynomials as below
\[
\xymatrix{ & E\ar[dl]_{s}\ar[r]^{p} & B\ar[dr]^{t}\ar@{=}[d] &  &  &  & E\ar[dl]_{s}\ar[r]^{p} & B\ar[dr]^{t}\ar@{=}[d]\\
X & S_{1}\ar[d]_{f_{1}}\ar[r]|-{pe_{1}}\ar[u]^{e_{1}} & B\ar[d]^{g} & Y & \sim & X & S_{2}\ar[d]_{f_{2}}\ar[r]|-{pe_{2}}\ar[u]^{e_{2}} & B\ar[d]^{g} & Y\\
 & M\ar[r]_{q}\ar[ul]^{u} & N\ar[ur]_{v} &  &  &  & M\ar[r]_{q}\ar[ul]^{u} & N\ar[ur]_{v}
}
\]
equivalent in that there exists an isomorphism $\nu\colon S_{1}\to S_{2}$
such that $f_{2}\nu=f_{1}$ and $e_{2}\nu=e_{1}$, it follows from
a straightforward functoriality of mates calculation that $L_{X,Y}$
assigns both morphisms of polynomials to equal pastings.

As local functoriality with respect to cartesian morphisms was shown
in Lemma \ref{polyclocally}, local functoriality with respect to
``triangle morphisms'' is a straightforward functoriality of mates
calculation, and the case of a triangle morphism followed by a cartesian
morphism is almost by definition, it suffices to consider the case
of a cartesian morphism followed by a triangle morphism (the only
non trivial case to consider). 

Suppose we are given a composite of polynomial morphisms as on the
left below
\[
\xymatrix{ & E\ar[dl]_{s}\ar[r]^{p}\ar[d]_{f}\ar@{}[rd]|-{\textnormal{pb}} & B\ar[dr]^{t}\ar[d]^{g} &  &  & E\ar[dl]_{s}\ar[rr]^{p} &  & B\ar[dr]^{t}\ar@{=}[d]\\
X & M\ar[r]|-{q}\ar[l]|-{u} & N\ar@{=}[d]\ar[r]|-{v} & Y\ar@{}[r]|-{=} & X & P\ar[d]_{f'}\ar[r]^{e'}\ar[u]^{e'}\ar@{}[rd]|-{\textnormal{pb}} & E\ar[d]|-{f}\ar[r]^{p}\ar@{}[rd]|-{\textnormal{pb}} & B\ar[d]^{g} & Y\\
 & J\ar[r]_{r}\ar[ul]^{a}\ar[u]^{e} & N\ar[ru]_{b} &  &  & J\ar[r]_{e}\ar[ul]^{a} & M\ar[r]_{q} & N\ar[ur]_{b}
}
\]
evaluated as the diagram on the right above. We must check that
\begin{equation}
\xymatrix{ & FE\ar[r]^{F_{\Pi}p}\ar@{}[rd]|-{\Downarrow\mathfrak{b}} & FB &  &  & FE\ar[r]^{F_{\Pi}p}\ar@{}[rd]|-{\Downarrow\mathfrak{m}}\ar[d]|-{F_{\Delta}e'} & FB\\
FX\ar@/^{0.7pc}/[ur]_{\;\;\cong}^{F_{\Delta}s}\ar@/_{0.7pc}/[rd]_{F_{\Delta}a}^{\;\;\cong}\ar[r]|-{F_{\Delta}u} & FM\ar[r]|-{F_{\Pi}q}\ar[u]|-{F_{\Delta}f}\ar@{}[rd]|-{\Downarrow\mathfrak{m}}\ar[d]|-{F_{\Delta}e} & FN\ar[u]_{F_{\Delta}g} & = & FX\ar@/^{0.7pc}/[ur]^{F_{\Delta}s}\ar@/_{0.7pc}/[rd]_{F_{\Delta}a}\ar@{}[r]|-{\Downarrow\alpha} & FP\ar@{}[rd]|-{\Downarrow\mathfrak{b}}\ar[r]|-{F_{\Pi}pe'} & FB\ar@{=}[u]\\
 & FJ\ar[r]_{F_{\Pi}r} & FN\ar@{=}[u] &  &  & FJ\ar[r]_{F_{\Pi}r}\ar[u]|-{F_{\Delta}f'} & FN\ar[u]_{F_{\Delta}g}
}
\label{mbinterchange}
\end{equation}
To see this, we paste both sides with the inverse of the $\mathfrak{b}$
appearing on the right above, and check that have an equality. Starting
with the observation that the left side pasted with this inverse is
the left diagram below, we see
\[
\xymatrix@=1em{ &  & FE\ar[rrrr]^{F_{\Pi}p}\ar@{}[rrd]|-{\Downarrow\mathfrak{b}} &  &  &  & FB &  &  &  & FE\ar[rrr]^{F_{\Pi}p}\ar@{}[rrdrd]|-{=} &  &  & FB\\
 & \ar@{}[]|-{\;\;\cong} &  &  & FN\ar[rru]|-{F_{\Delta}g}\ar@{}[rdr]|-{\Downarrow\mathfrak{b}^{-1}} &  &  &  &  & \ar@{}[]|-{\;\;\cong}\\
FX\ar@/^{0.7pc}/[urur]^{F_{\Delta}s}\ar@/_{0.7pc}/[rrdd]_{F_{\Delta}a}\ar[rr]|-{F_{\Delta}u} &  & FM\ar[rru]|-{F_{\Pi}q}\ar[uu]|-{F_{\Delta}f}\ar[dd]|-{F_{\Delta}e}\ar@{}[rdr]|-{\Downarrow\mathfrak{m}} &  &  &  & FE\ar[uu]_{F_{\Pi}p} & = & FX\ar@/^{0.7pc}/[urur]^{F_{\Delta}s}\ar@/_{0.7pc}/[rrdd]_{F_{\Delta}a}\ar[rr]|-{F_{\Delta}u} &  & FM\ar[uu]|-{F_{\Delta}f}\ar[dd]|-{F_{\Delta}e}\ar[r]_{\Downarrow\eta\;\;}^{\textnormal{id}} & FM\ar[rr]_{F_{\Delta}f}\ar@{}[rdrd]|-{\Downarrow\mathfrak{b}^{-1}} &  & FE\ar[uu]_{F_{\Pi}p}\\
 & \ar@{}[]|-{\;\;\cong} &  &  & FM\ar[uu]|-{F_{\Pi}q}\ar[urr]|-{F_{\Delta}f}\ar@{}[rdr]|-{\Downarrow\mathfrak{b}^{-1}} &  &  &  &  & \ar@{}[]|-{\;\;\cong}\\
 &  & FJ\ar[rrrr]_{F_{\Delta}f'}\ar[urr]|-{F_{\Pi}e} &  &  &  & FP\ar[uu]_{F_{\Pi}e'} &  &  &  & FJ\ar[rrr]_{F_{\Delta}f'}\ar[uur]|-{F_{\Pi}e} &  &  & FP\ar[uu]_{F_{\Pi}e'}
}
\]
upon realizing $\mathfrak{m}$ as a whiskering of a unit and canceling
the $\mathfrak{b}$. Transferring the unit along the mates correspondence
gives the left diagram below
\[
\xymatrix@=1em{ &  & FE\ar[rr]^{F_{\Pi}p}\ar@{}[rrdd]|-{=} &  & FB &  &  &  &  &  & FE\ar[rr]^{F_{\Pi}p}\ar@{}[rrdd]|-{=} &  & FB\\
 & \ar@{}[]|-{\;\;\cong} &  &  & \; & \ar@{}[l]|-{\;\;\;\;=} &  &  &  & \ar@{}[]|-{\;\;\cong} &  &  & \;\\
FX\ar@/^{0.7pc}/[urur]^{F_{\Delta}s}\ar@/_{0.7pc}/[rrdd]_{F_{\Delta}a}\ar[rr]|-{F_{\Delta}u} &  & FM\ar[uu]|-{F_{\Delta}f}\ar[dd]|-{F_{\Delta}e}\ar[rr]|-{F_{\Delta}f}\ar@{}[rrdd]|-{\cong} &  & FE\ar[dd]|-{F_{\Delta}e'}\ar[rr]|-{\textnormal{id}}\ar[uu]|-{F_{\Pi}p} &  & FE\ar@/_{0.7pc}/[uull]_{F_{\Pi}p} & = & FX\ar@/^{0.7pc}/[urur]^{F_{\Delta}s}\ar@/_{0.7pc}/[rrdd]_{F_{\Delta}a}\ar[rr]|-{F_{\Delta}u} &  & FM\ar[uu]|-{F_{\Delta}f}\ar[dd]|-{F_{\Delta}e}\ar[rr]|-{F_{\Delta}f}\ar@{}[rrdd]|-{\cong} &  & FE\ar[dd]|-{F_{\Delta}e'}\ar[uu]|-{F_{\Pi}p}\ar@{}[rr]|-{\Downarrow\mathfrak{m}\;\;\;\;} &  & \;\\
 & \ar@{}[]|-{\;\;\cong} &  &  & \; & \ar@{}[l]|-{\;\;\;\;\Downarrow\eta} &  &  &  & \ar@{}[]|-{\;\;\cong} &  &  & \;\\
 &  & FJ\ar[rr]_{F_{\Delta}f'} &  & FP\ar@/_{0.7pc}/[uurr]_{F_{\Pi}e'} &  &  &  &  &  & FJ\ar[rr]_{F_{\Delta}f'} &  & FP\ar@/_{3pc}/[uuuu]_{F_{\Pi}pe'}
}
\]
which is seen as the right diagram after writing the whiskering of
the unit back in terms of $\mathfrak{m}$. This is clearly the right
side of \eqref{mbinterchange} with the pasting of the 2-cell $\mathfrak{b}$
having been undone.

\noun{Part 4. }The goal here is to show that this now defines a gregarious
functor 
\[
L\colon\mathbf{Poly}\left(\mathcal{E}\right)\to\mathscr{C}
\]

Now, as we already have a gregarious functor $\overline{L}\colon\mathbf{Poly}_{c}\left(\mathcal{E}\right)\to\mathscr{C}$,
given by the restriction of $L$ to the cartesian setting, and composition
of 1-cells $\mathbf{Poly}_{c}\left(\mathcal{E}\right)$ and $\mathbf{Poly}\left(\mathcal{E}\right)$
is defined the same way, it suffices to check that the oplax constraint
data $\varphi,\lambda$ of $\overline{L}$ defines oplax constraint
data on $L$. Indeed, $\varphi$ and $\lambda$ are already known
to satisfy the nullary and associativity axioms and so we need only
check naturality of the constraint data with respect to our larger
class of 2-cells. 

Taking $\theta\colon P\to P''$ and $\phi\colon Q\to Q''$ to be general
morphisms of polynomials, canonically decomposed into triangle parts
$\theta_{t},\phi_{t}$ and cartesian parts $\theta_{c},\phi_{c}$,
we note that to see that the left diagram commutes below
\[
\xymatrix@R=1.5em{L\left(P;Q\right)\ar[r]^{\varphi_{P,Q}}\ar[dd]_{L\left(\theta;\phi\right)} & LP;LQ\ar[dd]^{L\theta;L\phi} &  & L\left(P;Q\right)\ar[r]^{\varphi_{P,Q}}\ar[d]_{L\left(\theta_{t};\phi_{t}\right)} & LP;LQ\ar[d]^{L\theta_{t};L\phi_{t}}\\
 &  &  & L\left(P';Q'\right)\ar[d]_{L\left(\theta_{c};\phi_{c}\right)}\ar[r]_{\varphi_{P',Q'}} & LP';LQ'\ar[d]^{L\theta_{c};L\phi_{c}}\\
L\left(P'';Q'\right)\ar[r]_{\varphi_{P'',Q''}} & LP'';LQ'' &  & L\left(P'';Q'\right)\ar[r]_{\varphi_{P'',Q''}} & LP'';LQ''
}
\]
it suffices to check that the top square in the right diagram commutes.
This is since the bottom square on the right commutes by naturality
of the constraint data $\varphi$ with respect to $\overline{L}$.
To prove the commutativity of this square, it will be helpful to decompose
further into
\[
\xymatrix@R=1.5em{L\left(P;Q\right)\ar[d]_{L\left(\theta_{t};Q\right)}\ar[rr]^{\varphi_{P,Q}} &  & LP;LQ\ar[d]^{L\theta_{t};LQ}\\
L\left(P';Q\right)\ar[d]_{L\left(P';\phi_{t}\right)}\ar[rr]|-{\varphi_{P',Q}} &  & LP';LQ\ar[d]^{LP';L\phi_{t}}\\
L\left(P';Q'\right)\ar[rr]_{\varphi_{P',Q'}} &  & LP';LQ'
}
\]
and so we need only prove naturality for whiskerings of triangle morphisms.

We first check the naturality condition for right whiskerings of triangle
morphisms. The whiskering of such an $x$ is constructed as the induced
map $x'$ into the pullback as in the diagram below
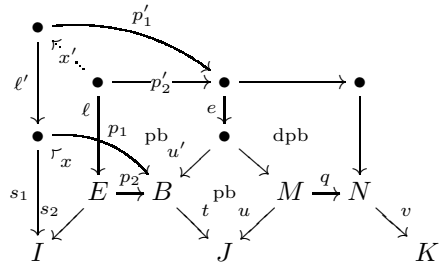
\[
\xymatrix@=1em{\bullet\ar[dd]_{\ell'}\ar@/^{0.7pc}/[drrr]^{p_{1}'}\\
 & \bullet\ar@{-}[d]_{\ell}\ar[dd]\ar[rr]|-{p_{2}'}\ar@{..>}[ul]|-{x'} &  & \bullet\ar[d]_{e}\ar[rr] &  & \bullet\ar[dd]\\
\bullet\ar[dd]_{s_{1}}\ar@/^{0.7pc}/[drr]^{p_{1}} &  & \ar@{}[]|-{\textnormal{pb}\;\;} & \bullet\ar[rd]\ar[ld]_{u'} & \ar@{}[]|-{\textnormal{dpb}}\\
 & E\ar[r]^{p_{2}}\ar[ld]_{s_{2}}\ar[ul]|-{x} & B\ar[rd]^{t} & \ar@{}[]|-{\textnormal{pb}} & M\ar[r]^{q}\ar[ld]_{u} & N\ar[rd]^{v}\\
I &  &  & J &  &  & K
}
\]
The naturality condition then amounts to checking that
\[
\xymatrix{ & \bullet\ar[r]^{F_{\Pi}p_{1}'}\ar@{}[rd]|-{\Downarrow\mathfrak{b}} & \bullet &  &  & \bullet\ar[r]^{F_{\Pi}p_{1}'}\ar@{}[rd]|-{\Downarrow\mathfrak{m}}\ar[d]|-{F_{\Delta}x'} & \bullet\\
\bullet\ar@/^{0.7pc}/[ur]^{F_{\Delta}s_{1}\ell'}\ar@/_{0.7pc}/[rd]_{F_{\Delta}s_{2}}\ar[r]|-{F_{\Delta}s_{1}} & \bullet\ar[r]|-{F_{\Pi}p_{1}}\ar[u]|-{F_{\Delta}\ell'}\ar@{}[rd]|-{\Downarrow\mathfrak{m}}\ar[d]|-{F_{\Delta}x} & \bullet\ar[u]_{F_{\Delta}u'e} & = & \bullet\ar@/^{0.7pc}/[ur]^{F_{\Delta}s_{1}\ell'}\ar@/_{0.7pc}/[rd]_{F_{\Delta}s_{2}}\ar@{}[r]|-{\Downarrow\alpha} & \bullet\ar@{}[rd]|-{\Downarrow\mathfrak{b}}\ar[r]|-{F_{\Pi}p_{2}'} & \bullet\ar@{=}[u]\\
 & \bullet\ar[r]_{F_{\Pi}p_{2}} & \bullet\ar@{=}[u] &  &  & \bullet\ar[r]_{F_{\Pi}p_{2}}\ar[u]|-{F_{\Delta}\ell} & \bullet\ar[u]_{F_{\Delta}u'e}
}
\]
which is similar to the calculation in \eqref{mbinterchange}.

We now check left whiskerings of triangle morphisms, which is significantly
more complicated than the above situation. To simplify this calculation,
we consider only simpler triangle morphisms of the form $x\colon\left(u_{1},1,1\right)\to\left(u_{2},x,1\right)$.
It will turn out that it suffices to consider only these simpler triangle
morphisms.

To construct the left whiskering of the triangle morphism $x$ by
a polynomial we first construct the two relevant composites of polynomials
as below
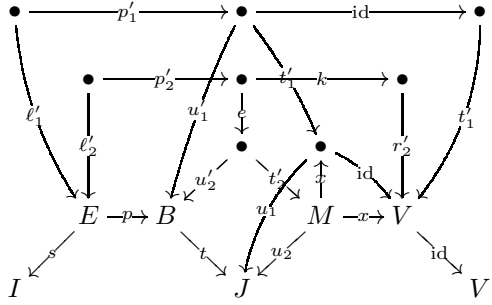
\[
\xymatrix@=1.5em{\bullet\ar[rrr]|-{p_{1}'}\ar@/_{0.5pc}/[dddr]|-{\ell'_{1}} &  &  & \bullet\ar@/^{0.2pc}/[rdd]|-{t_{1}'}\ar@/_{0.2pc}/[lddd]|-{u_{1}'}\ar[rrr]|-{\textnormal{id}} &  &  & \bullet\ar@/^{0.7pc}/[dddl]|-{\;t_{1}'}\\
 & \bullet\ar[dd]|-{\ell'_{2}}\ar[rr]|-{p_{2}'} &  & \bullet\ar[d]|-{e}\ar[rr]|-{k} &  & \bullet\ar[dd]|-{r_{2}'}\\
 &  &  & \bullet\ar[rd]|-{t_{2}'}\ar[ld]|-{u_{2}'} & \bullet\ar@/_{0.5pc}/[ldd]|-{u_{1}}\ar@/^{0.3pc}/[rd]|-{\textnormal{id}}\\
 & E\ar[r]|-{p}\ar[ld]|-{s} & B\ar[rd]|-{t} &  & M\ar[r]|-{x}\ar[ld]|-{u_{2}}\ar[u]|-{x} & V\ar[rd]|-{\textnormal{id}}\\
I &  &  & J &  &  & V
}
\]
Now since we have a factorization of pullbacks as below
\[
\xymatrix{\bullet\ar[r]|-{y}\ar[d]|-{t_{2}'} & \bullet\ar[d]|-{t_{1}'}\ar[r]|-{u_{1}'} & \bullet\ar[d]|-{t}\ar@{}[rrd]|-{=} &  & \bullet\ar[rr]|-{u_{2}'}\ar[d]|-{t_{2}'} &  & \bullet\ar[d]|-{t}\\
\bullet\ar[r]|-{x} & \bullet\ar[r]|-{u_{1}} & \bullet &  & \bullet\ar[rr]|-{u_{2}} &  & \bullet
}
\]
it follows that we have an induced cartesian morphism of polynomials
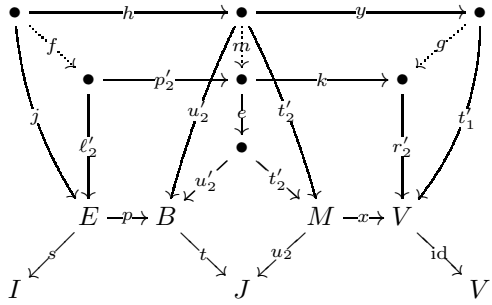
\[
\xymatrix@=1.5em{\bullet\ar[rrr]|-{h}\ar@/_{0.5pc}/[dddr]|-{j}\ar@{..>}[rd]|-{f} &  &  & \bullet\ar@/^{0.2pc}/[rddd]|-{t_{2}'}\ar@/_{0.2pc}/[lddd]|-{u_{2}'}\ar[rrr]|-{y}\ar@{..>}[d]|-{m} &  &  & \bullet\ar@{..>}[ld]|-{g}\ar@/^{0.7pc}/[dddl]|-{\;t_{1}'}\\
 & \bullet\ar[dd]|-{\ell'_{2}}\ar[rr]|-{p_{2}'} &  & \bullet\ar[d]|-{e}\ar[rr]|-{k} &  & \bullet\ar[dd]|-{r_{2}'}\\
 &  &  & \bullet\ar[rd]|-{t_{2}'}\ar[ld]|-{u_{2}'}\\
 & E\ar[r]|-{p}\ar[ld]|-{s} & B\ar[rd]|-{t} &  & M\ar[r]|-{x}\ar[ld]|-{u_{2}} & V\ar[rd]|-{\textnormal{id}}\\
I &  &  & J &  &  & V
}
\]
where $h$ and $j$ are the pullback of $p$ with $u_{2}'$. We then
give the factorization of pullbacks
\[
\xymatrix{\bullet\ar[d]|-{h}\ar[r]|-{\alpha} & \bullet\ar[d]|-{p_{1}'}\ar[r]|-{\ell_{1}'} & \bullet\ar@{}[rrd]|-{=}\ar[d]|-{p} &  & \bullet\ar[d]|-{h}\ar[rr]|-{j} &  & \bullet\ar[d]|-{p}\\
\bullet\ar[r]|-{y} & \bullet\ar[r]|-{u_{1}'} & \bullet &  & \bullet\ar[rr]|-{u_{2}'} &  & \bullet
}
\]
and see the morphism of polynomials resulting from this whiskering
is given by
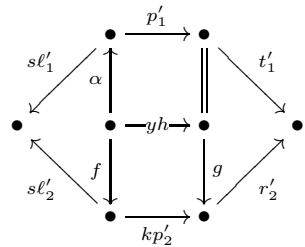
\[
\xymatrix{ & \bullet\ar[ld]_{s\ell_{1}'}\ar[r]^{p_{1}'} & \bullet\ar[rd]^{t_{1}'}\\
\bullet & \bullet\ar[d]_{f}\ar[r]|-{yh}\ar[u]^{\alpha} & \bullet\ar[d]^{g}\ar@{=}[u] & \bullet\\
 & \bullet\ar[r]_{kp_{2}'}\ar[ul]^{s\ell_{2}'} & \bullet\ar[ru]_{r_{2}'}
}
\]
The naturality condition then follows from seeing that, where $z=u_{1}t_{1}'=tu_{1}'$,
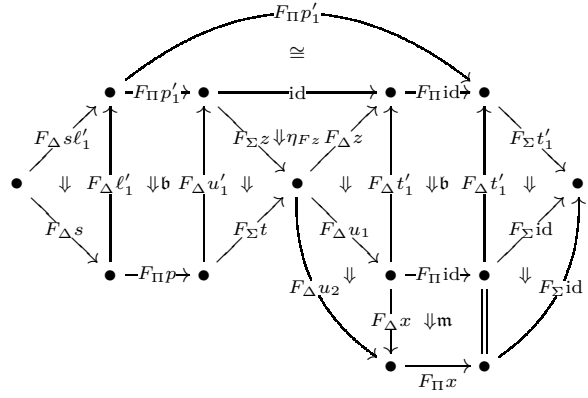
\[
\xymatrix{ &  &  & \;\ar@{}[d]|-{\cong}\\
 & \bullet\ar[r]|-{F_{\Pi}p_{1}'}\ar@/^{2.5pc}/[rrrr]|-{F_{\Pi}p_{1}'} & \bullet\ar[rr]|-{\textnormal{id}}\ar[rd]|-{F_{\Sigma}z} & \;\ar@{}[d]|-{\Downarrow\eta_{Fz}} & \bullet\ar[r]|-{F_{\Pi}\textnormal{id}} & \bullet\ar[rd]|-{F_{\Sigma}t_{1}'}\\
\bullet\ar[ur]|-{F_{\Delta}s\ell_{1}'}\ar[rd]|-{F_{\Delta}s}\ar@{}[r]|-{\Downarrow} & \;\ar@{}[r]|-{\Downarrow\mathfrak{b}} & \ar@{}[r]|-{\Downarrow} & \bullet\ar[rd]|-{F_{\Delta}u_{1}}\ar@{}[r]|-{\Downarrow}\ar[ur]|-{F_{\Delta}z}\ar@/_{1pc}/[ddr]|-{F_{\Delta}u_{2}} & \;\ar@{}[r]|-{\Downarrow\mathfrak{b}} & \;\ar@{}[r]|-{\Downarrow} & \bullet\\
 & \bullet\ar[r]|-{F_{\Pi}p}\ar[uu]|-{F_{\Delta}\ell_{1}'} & \bullet\ar[ur]|-{F_{\Sigma}t}\ar[uu]|-{F_{\Delta}u_{1}'} & \ar@{}[r]|-{\;\;\;\Downarrow} & \bullet\ar[uu]|-{F_{\Delta}t_{1}'}\ar[r]|-{F_{\Pi}\textnormal{id}}\ar[d]|-{F_{\Delta}x}\ar@{}[rd]|-{\Downarrow\mathfrak{m}} & \bullet\ar[uu]|-{F_{\Delta}t_{1}'}\ar[ur]|-{F_{\Sigma}\textnormal{id}}\ar@{=}[d] & \ar@{}[l]|-{\Downarrow\;\;\;}\\
 &  &  &  & \bullet\ar[r]_{F_{\Pi}x} & \bullet\ar@/_{1pc}/[uur]|-{F_{\Sigma}\textnormal{id}}
}
\]
is equal to the pasting, using an analogue of \eqref{mbinterchange},
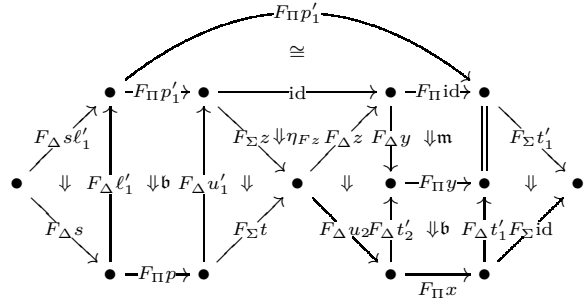
\[
\xymatrix{ &  &  & \;\ar@{}[d]|-{\cong}\\
 & \bullet\ar[r]|-{F_{\Pi}p_{1}'}\ar@/^{2.5pc}/[rrrr]|-{F_{\Pi}p_{1}'} & \bullet\ar[rr]|-{\textnormal{id}}\ar[rd]|-{F_{\Sigma}z} & \;\ar@{}[d]|-{\Downarrow\eta_{Fz}} & \bullet\ar[r]|-{F_{\Pi}\textnormal{id}}\ar[d]|-{F_{\Delta}y}\ar@{}[rd]|-{\Downarrow\mathfrak{m}} & \bullet\ar[rd]|-{F_{\Sigma}t_{1}'}\\
\bullet\ar[ur]|-{F_{\Delta}s\ell_{1}'}\ar[rd]|-{F_{\Delta}s}\ar@{}[r]|-{\Downarrow} & \;\ar@{}[r]|-{\Downarrow\mathfrak{b}} & \ar@{}[r]|-{\Downarrow} & \bullet\ar[ur]|-{F_{\Delta}z}\ar@/_{0pc}/[dr]|-{F_{\Delta}u_{2}}\ar@{}[r]|-{\;\Downarrow} & \bullet\ar[r]|-{F_{\Pi}y}\ar@{}[dr]|-{\Downarrow\mathfrak{b}} & \bullet\ar@{=}[u]\ar@{}[r]|-{\Downarrow} & \bullet\\
 & \bullet\ar[r]|-{F_{\Pi}p}\ar[uu]|-{F_{\Delta}\ell_{1}'} & \bullet\ar[ur]|-{F_{\Sigma}t}\ar[uu]|-{F_{\Delta}u_{1}'} &  & \bullet\ar[r]_{F_{\Pi}x}\ar[u]|-{F_{\Delta}t_{2}'} & \bullet\ar@/_{0pc}/[ur]|-{F_{\Sigma}\textnormal{id}}\ar[u]|-{F_{\Delta}t_{1}'}
}
\]
which is equal to, where $d=u_{2}t_{2}'e=tu_{2}'e$, noting $dm=zy$
and $u_{2}'=u_{1}'y$,
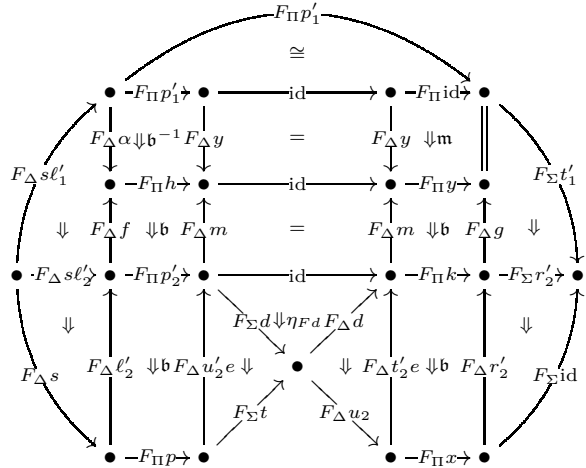
\[
\xymatrix{ &  &  & \ar@{}[d]|-{\cong}\\
 & \bullet\ar@/^{2.5pc}/[rrrr]|-{F_{\Pi}p_{1}'}\ar[d]|-{F_{\Delta}\alpha}\ar[r]|-{F_{\Pi}p_{1}'}\ar@{}[rd]|-{\Downarrow\mathfrak{b}^{-1}} & \bullet\ar[rr]|-{\textnormal{id}}\ar[d]|-{F_{\Delta}y}\ar@{}[rdr]|-{=} &  & \bullet\ar[r]|-{F_{\Pi}\textnormal{id}}\ar[d]|-{F_{\Delta}y}\ar@{}[rd]|-{\Downarrow\mathfrak{m}} & \bullet\ar@{=}[d]\ar@/^{0.8pc}/[rdd]|-{F_{\Sigma}t_{1}'}\\
 & \bullet\ar[r]|-{F_{\Pi}h}\ar@{}[rd]|-{\Downarrow\mathfrak{b}} & \bullet\ar[rr]|-{\textnormal{id}}\ar@{}[rdr]|-{=} & \; & \bullet\ar[r]|-{F_{\Pi}y}\ar@{}[rd]|-{\Downarrow\mathfrak{b}} & \bullet & \ar@{}[dl]|-{\Downarrow}\\
\bullet\ar[r]|-{F_{\Delta}s\ell_{2}'}\ar@/_{0.8pc}/[rdd]|-{F_{\Delta}s}\ar@{}[rd]|-{\Downarrow}\ar@/^{0.8pc}/[ruu]|-{F_{\Delta}s\ell_{1}'}\ar@{}[ru]|-{\Downarrow} & \bullet\ar[r]|-{F_{\Pi}p_{2}'}\ar[u]|-{F_{\Delta}f} & \bullet\ar[rr]|-{\textnormal{id}}\ar[rd]|-{F_{\Sigma}d}\ar[u]|-{F_{\Delta}m} & \;\ar@{}[d]|-{\Downarrow\eta_{Fd}} & \bullet\ar[r]|-{F_{\Pi}k}\ar[u]|-{F_{\Delta}m} & \bullet\ar[r]|-{F_{\Sigma}r_{2}'}\ar[u]|-{F_{\Delta}g} & \bullet\\
 & \;\ar@{}[r]|-{\Downarrow\mathfrak{b}} & \ar@{}[r]|-{\Downarrow} & \bullet\ar[rd]|-{F_{\Delta}u_{2}}\ar@{}[r]|-{\Downarrow}\ar[ur]|-{F_{\Delta}d} & \;\ar@{}[r]|-{\Downarrow\mathfrak{b}} & \;\ar@{}[ur]|-{\Downarrow}\\
 & \bullet\ar[r]|-{F_{\Pi}p}\ar[uu]|-{F_{\Delta}\ell_{2}'} & \bullet\ar[ur]|-{F_{\Sigma}t}\ar[uu]|-{F_{\Delta}u_{2}'e} &  & \bullet\ar[uu]|-{F_{\Delta}t_{2}'e}\ar[r]|-{F_{\Pi}x} & \bullet\ar[uu]|-{F_{\Delta}r_{2}'}\ar@/_{0.8pc}/[uru]|-{F_{\Sigma}\textnormal{id}}
}
\]
finally resulting in 
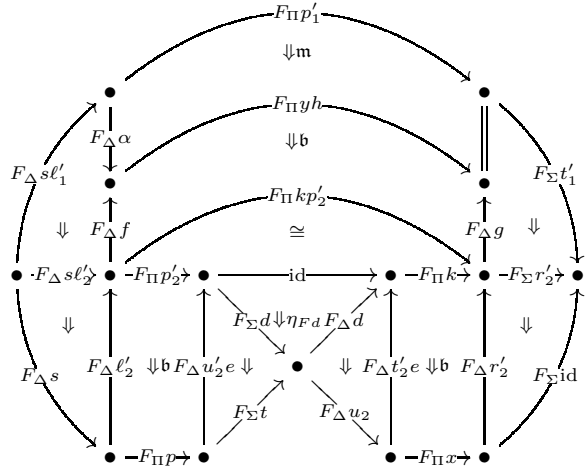
\[
\xymatrix{ &  &  & \ar@{}[d]|-{\Downarrow\mathfrak{m}}\\
 & \bullet\ar@/^{2.5pc}/[rrrr]|-{F_{\Pi}p_{1}'}\ar[d]|-{F_{\Delta}\alpha} &  & \ar@{}[d]|-{\Downarrow\mathfrak{b}} &  & \bullet\ar@{=}[d]\ar@/^{0.8pc}/[rdd]|-{F_{\Sigma}t_{1}'}\\
 & \bullet\ar@/^{2.5pc}/[rrrr]|-{F_{\Pi}yh} &  & \;\ar@{}[d]|-{\cong} &  & \bullet & \ar@{}[dl]|-{\Downarrow}\\
\bullet\ar[r]|-{F_{\Delta}s\ell_{2}'}\ar@/_{0.8pc}/[rdd]|-{F_{\Delta}s}\ar@{}[rd]|-{\Downarrow}\ar@/^{0.8pc}/[ruu]|-{F_{\Delta}s\ell_{1}'}\ar@{}[ru]|-{\Downarrow} & \bullet\ar[r]|-{F_{\Pi}p_{2}'}\ar@/^{2.5pc}/[rrrr]|-{F_{\Pi}kp_{2}'}\ar[u]|-{F_{\Delta}f} & \bullet\ar[rr]|-{\textnormal{id}}\ar[rd]|-{F_{\Sigma}d} & \;\ar@{}[d]|-{\Downarrow\eta_{Fd}} & \bullet\ar[r]|-{F_{\Pi}k} & \bullet\ar[r]|-{F_{\Sigma}r_{2}'}\ar[u]|-{F_{\Delta}g} & \bullet\\
 & \;\ar@{}[r]|-{\Downarrow\mathfrak{b}} & \ar@{}[r]|-{\Downarrow} & \bullet\ar[rd]|-{F_{\Delta}u_{2}}\ar@{}[r]|-{\Downarrow}\ar[ur]|-{F_{\Delta}d} & \;\ar@{}[r]|-{\Downarrow\mathfrak{b}} & \;\ar@{}[ur]|-{\Downarrow}\\
 & \bullet\ar[r]|-{F_{\Pi}p}\ar[uu]|-{F_{\Delta}\ell_{2}'} & \bullet\ar[ur]|-{F_{\Sigma}t}\ar[uu]|-{F_{\Delta}u_{2}'e} &  & \bullet\ar[uu]|-{F_{\Delta}t_{2}'e}\ar[r]|-{F_{\Pi}x} & \bullet\ar[uu]|-{F_{\Delta}r_{2}'}\ar@/_{0.8pc}/[uru]|-{F_{\Sigma}\textnormal{id}}
}
\]
Now, we wish to prove the naturality condition for any left whiskering
of a general triangle morphism, written $P;\theta_{t}$. This can
be written as $P;\left(\theta_{x};R\right)$ for a simpler triangle
morphism $x$ as above, and so we are trying to show region (1) commutes
below (suppressing associators in $\mathscr{C}$)
\[
\xymatrix{L\left(P;\left(\left(u_{1},1,1\right);R\right)\right)\myar{\varphi}{r}\ar[d]_{L\left(P;\left(\theta_{x};R\right)\right)}\ar@{}[rd]|-{\left(1\right)} & LP;L\left(\left(u_{1},1,1\right);R\right)\myar{\varphi}{r}\ar[d]|-{LP;L\left(\theta_{x};R\right)}\ar@{}[rd]|-{\left(2\right)} & LP;L\left(u_{1},1,1\right);LR\ar[d]^{LP;L\theta_{x};LR}\\
L\left(P;\left(\left(u_{2},x,1\right);R\right)\right)\myard{\varphi}{r} & LP;L\left(\left(u_{2},x,1\right);R\right)\myard{\varphi}{r} & LP;L\left(u_{2},x,1\right);LR
}
\]
We now note that for the commutativity of (1) it suffices to prove
the outside diagram above commutes, as both constraints $\varphi$
are invertible in region (2) by gregariousness, and region (2) is
known to commute as naturality for right whiskerings has been shown.

As associativity of the constraints has been verified, this is the
same as showing that the outside of 
\[
\xymatrix{L\left(\left(P;\left(u_{1},1,1\right)\right);R\right)\myar{\varphi}{r}\ar[d]_{L\left(\left(P;\theta_{x}\right);R\right)} & L\left(P;\left(u_{1},1,1\right)\right);LR\myar{\varphi}{r}\ar[d]|-{L\left(P;\theta_{x}\right);LR} & LP;L\left(u_{1},1,1\right);LR\ar[d]^{LP;L\theta_{x};LR}\\
L\left(\left(P;\left(u_{2},x,1\right)\right);R\right)\myard{\varphi}{r} & L\left(P;\left(u_{2},x,1\right)\right);LR\myard{\varphi}{r} & LP;L\left(u_{2},x,1\right);LR
}
\]
commutes. But the left square commutes as naturality with respect
to right whiskerings is known for both triangle and cartesian morphisms,
and the right square above commutes by naturality of left whiskerings
of $\theta_{x}$. This gives the result.
\end{proof}
We now have enough to complete the proof of Theorem \ref{unipolythm}.
\begin{proof}
[Proof of Theorem \ref{unipolythm}] We consider the assignment of
Theorem \ref{unipolythm}, i.e. composition with the embedding $\left(-\right)_{\Sigma}\colon\mathcal{E}\to\mathbf{Poly}\left(\mathcal{E}\right)$
written as the assignment
\[
\xymatrix{\mathbf{Poly}\left(\mathcal{E}\right)\ar@/_{1pc}/[rr]_{G}\ar@/^{1pc}/[rr]^{F} & \Downarrow\alpha & \mathscr{C} & \mapsto & \mathcal{E}\ar@/_{1pc}/[rr]_{G_{\Sigma}}\ar@/^{1pc}/[rr]^{F_{\Sigma}} & \Downarrow\alpha_{\Sigma} & \mathscr{C}}
\]

We start by proving the first universal property.

\emph{\noun{Well defined.}} That each icon is invertible is seen by
restricting to spans and applying Corollary \ref{spaniconinv}.

\noun{Fully faithful.} That the assignment $\alpha\mapsto\alpha_{\Sigma}$
is injective follows from the necessary commutativity of
\[
\xymatrix{F\left(s,p,t\right)\myar{\varphi}{r}\ar[d]_{\alpha_{\left(s,p,t\right)}} & F\left(s,1,1\right);F\left(1,p,1\right);F\left(1,1,t\right)\ar[d]^{\alpha_{\left(s,1,1\right)};\alpha_{\left(1,p,1\right)};\alpha_{\left(1,1,t\right)}}\\
G\left(s,p,t\right)\myard{\psi}{r} & G\left(s,1,1\right);G\left(1,p,1\right);G\left(1,1,t\right)
}
\]
where $\varphi$ and $\psi$ are invertible by gregariousness, and
the identities $\alpha_{\Delta_{f}}=\left(\alpha_{\Sigma_{f}}^{-1}\right)^{*}$
and $\alpha_{\Pi_{f}}=\left(\alpha_{\Delta_{f}}^{-1}\right)^{*}$
forced by Lemma \ref{mateinv}. For surjectivity, one need only check
any collection $\alpha$ consisting of 2-cells
\[
\alpha_{s,p,t}\colon F\left(s,p,t\right)\to G\left(s,p,t\right)
\]
satisfying these properties defines an icon. As composition is the
same in $\mathbf{Poly}_{c}\left(\mathcal{E}\right)$, the compatibility
of the collection $\alpha$ with the oplax constraint cells is the
same calculation as in Section \ref{unipolycart}. Thus one need only
check local naturality of $\alpha$. As local naturality with respect
to the cartesian morphisms is already known, one need only consider
triangle morphisms. But local naturality with respect to triangle
morphism is almost the same calculation as in the case of spans; this
is expected as the triangle morphisms arise from the canonical embedding
$\left(-\right)_{\Delta\Pi}\colon\mathbf{Span}^{\textnormal{co}}\left(\mathcal{E}\right)\to\mathbf{Poly}\left(\mathcal{E}\right)$.

\noun{Essentially surjective}. Given any 2-Beck pseudofunctor $F\colon\mathcal{E}\to\mathscr{C}$
we take the gregarious functor $L\colon\mathbf{Poly}\left(\mathcal{E}\right)\to\mathscr{C}$
from Proposition \ref{polyLoplax} and note that $L_{\Sigma}=F$. 

We now deduce the second universal property.

\noun{Restrictions}. The second property is a restriction of the first.
Indeed, given a pseudofunctor $L\colon\mathbf{Poly}\left(\mathcal{E}\right)\to\mathscr{C}$
the corresponding pseudofunctor $L_{\Sigma}\colon\mathcal{E}\to\mathscr{C}$
satisfies the distributivity condition since $L_{\Sigma}$ is also
the restriction of the pseudofunctor $\overline{L}\colon\mathbf{Poly}_{c}\left(\mathcal{E}\right)\to\mathscr{C}$.
Moreover, given a 2-Beck pseudofunctor $F\colon\mathcal{E}\to\mathscr{C}$
which satisfies the distributivity condition, the corresponding map
$\mathbf{Poly}\left(\mathcal{E}\right)\to\mathscr{C}$ is pseudo since
the map $\mathbf{Poly}_{c}\left(\mathcal{E}\right)\to\mathscr{C}$
(with the same constraint data) arising from the cartesian Beck triple
is pseudo.
\end{proof}

\section{Acknowledgments}

The author is grateful to his supervisor Richard Garner for feedback
on earlier versions of this paper, as well as the anonymous referee
for their helpful suggestions. In addition, the author gratefully
acknowledges the support of an Australian Government Research Training
Program Scholarship. 

\bibliographystyle{siam}
\bibliography{references}

\end{document}